\documentclass[a4paper,10pt]{amsart}
\usepackage[english]{babel}
\usepackage[english]{translator}
\usepackage[T1]{fontenc}
\usepackage[utf8]{inputenc}
\usepackage{amsmath}
\usepackage{amssymb}
\usepackage{amsfonts}
\usepackage{amstext}
\usepackage{amsthm}
\usepackage{bbm}

\usepackage{graphicx}
\usepackage{color}
\usepackage{psfrag}
\usepackage{subfigure}
\usepackage{float}
\usepackage{marvosym}

\usepackage{enumitem}
\usepackage{hyperref}
\usepackage[
  hmarginratio={1:1},     % equal left and right margins
  vmarginratio={1:1},     % equal top and bottom margins
  textwidth=460pt,        % new text width
  heightrounded,          % always useful
  bottom=3.60cm,
  top=3.60cm,
]{geometry}

\hypersetup{
  pdffitwindow=false,
  pdfhighlight=/O,
  pdfnewwindow,
  colorlinks=true,
  citecolor=red,            % Farbe der Zitate(Buchverweise)
  linkcolor=blue,             % Farbe der Hyperref-Links
  menucolor=blue,            % color of Acrobat Reader menu buttons
  urlcolor=blue,             % Farbe der Internetadressverweise
  pdfpagemode=UseOutlines,
  bookmarksnumbered=true,
  linktocpage,
  pdftitle={Freezing traveling waves in second order evolution equations},
  pdfsubject={Freezing traveling waves in second order evolution equations},
  pdfauthor={Wolf-J"urgen Beyn, Denny Otten, Jens Rottmann-Matthes},
  pdfkeywords={damped wave equation, traveling wave, freezing method, second order evolution equation, relative equilibria, stability},
  pdfcreator={latex,bibtex,dvips,ps2pdf},
  pdfproducer={LaTeX mit hyperref}
}

\newcommand{\C}{\mathbb{C}}               % komplexe Zahlen
\newcommand{\R}{\mathbb{R}}               % reelle Zahlen
\newcommand{\N}{\mathbb{N}}                % natuerliche Zahlen
\renewcommand{\Re}{\mathrm{Re}\,}          % Realteil
\renewcommand{\Im}{\mathrm{Im}\,}          % Imaginaerteil
\newcommand{\A}{\mathcal{A}}             % A Generator of a strongly continuous semigroup
\newcommand{\diag}{\mathrm{diag}}          % Diagonalmatrix
\newcommand{\PL}{\mathcal{P}}

\newcommand{\kernel}{\mathcal{N}}
\newcommand{\range}{\mathcal{R}}

\newcommand{\begriff}[1]{\text{#1}}

\renewcommand{\theequation}{\arabic{section}.\arabic{equation}}

\makeatletter
\renewcommand{\@secnumfont}{\bfseries}
\makeatother
\makeatletter
  \def\section{\@startsection{section}{1}%
    \z@{.7\linespacing\@plus\linespacing}{.5\linespacing}%
    {\normalfont\LARGE\bfseries}}
\makeatother
\makeatletter
\def\@seccntformat#1{%
  \protect\textup{%
    \protect\@secnumfont
    \expandafter\protect\csname format#1\endcsname % <--- added
    \csname the#1\endcsname
    \protect\@secnumpunct
  }%
}

\newcommand{\sect}
{
  \setcounter{equation}{0}
  \setcounter{figure}{0}
  \section
}
\theoremstyle{plain}
\newtheorem{definition}{Definition}[section]
\newtheorem{theorem}[definition]{Theorem}
\newtheorem{lemma}[definition]{Lemma}

\newtheorem{assumption}[definition]{Assumption}
\newtheorem{remark}[definition]{Remark}

\theoremstyle{definition}

\newtheorem{example}[definition]{Example}
\newtheorem{proposition}[definition]{Proposition}

\allowdisplaybreaks

\begin{document}
%  -----------
% |   Title   |
%  -----------
\title[Computation and Stability of Traveling Waves\\in Second Order Evolution Equations]{Computation and Stability of Traveling Waves\\in Second Order Evolution Equations}
%\maketitle
\setlength{\parindent}{0pt}
%\vspace*{0.5cm}
\begin{center}
%\normalfont\LARGE\bfseries{\shorttitle}
\normalfont\huge\bfseries{\shorttitle}\\
\vspace*{0.25cm}
%\Large\bfseries\MakeUppercase{\shorttitle}
\end{center}

%  -------------
% |   Authors   |
%  -------------
\vspace*{0.8cm}
\noindent
\begin{minipage}[t]{0.99\textwidth}
\begin{minipage}[t]{0.48\textwidth}
% First Author
\hspace*{1.8cm} 
\textbf{Wolf-J{\"u}rgen Beyn}\footnotemark[1]${}^{,}$\footnotemark[4] \\
% Second Author
\hspace*{1.8cm}
\textbf{Denny Otten}\footnotemark[2]${}^{,}$\footnotemark[4] \\
\hspace*{1.8cm}
Department of Mathematics \\
\hspace*{1.8cm}
Bielefeld University \\
\hspace*{1.8cm}
33501 Bielefeld \\
\hspace*{1.8cm}
Germany
\end{minipage}
\begin{minipage}[t]{0.48\textwidth}
% Third Author
\hspace*{1.8cm}
\textbf{Jens Rottmann-Matthes}\footnotemark[3]${}^{,}$\footnotemark[5] \\
\hspace*{1.8cm}
Institut für Analysis \\
\hspace*{1.8cm}
Karlsruhe Institute of Technology \\
\hspace*{1.8cm}
76131 Karlsruhe \\
\hspace*{1.8cm}
Germany
\end{minipage}
\end{minipage}\\

\footnotetext[1]{e-mail: \textcolor{blue}{beyn@math.uni-bielefeld.de}, phone: \textcolor{blue}{+49 (0)521 106 4798}, \\
                                          fax: \textcolor{blue}{+49 (0)521 106 6498}, homepage: \url{http://www.math.uni-bielefeld.de/~beyn/AG\_Numerik/}.}
\footnotetext[2]{e-mail: \textcolor{blue}{dotten@math.uni-bielefeld.de}, phone: \textcolor{blue}{+49 (0)521 106 4784}, \\
                                 fax: \textcolor{blue}{+49 (0)521 106 6498}, homepage: \url{http://www.math.uni-bielefeld.de/~dotten/}.}
\footnotetext[3]{e-mail: \textcolor{blue}{jens.rottmann-matthes@kit.edu}, phone: \textcolor{blue}{+49 (0)721 608 41632}, \\
                                          fax: \textcolor{blue}{+49 (0)721 608 46530}, homepage: \url{http://www.math.kit.edu/iana2/~rottmann/}.}
\footnotetext[4]{supported by CRC 701 'Spectral Structures and Topological Methods in Mathematics',  Bielefeld University}
\footnotetext[5]{supported by CRC 1173 'Wave Phenomena: Analysis and Numerics', Karlsruhe Institute of Technology}
%  ----------
% |   Date   |
%  ----------
\vspace*{0.6cm}
\noindent
\hspace*{6.2cm}
Date: \today
\normalparindent=12pt

%  --------------
% |   Abstract   |
%  --------------
\vspace{0.4cm}
%\begin{abstract}
\begin{center}
\begin{minipage}{0.9\textwidth}
  {\small
  \textbf{Abstract.} 
  The topic of this paper are nonlinear traveling waves occuring in a system
of damped waves equations in one space variable. We extend the freezing method from first to second order equations in time. When
applied to a Cauchy problem, this method generates a comoving frame
in which the solution becomes stationary. In addition it generates an algebraic 
variable which converges to the speed of the wave, provided the original
wave satisfies certain spectral conditions  and initial 
perturbations are sufficiently small. We develop a rigorous theory 
for this effect by recourse to some recent nonlinear stability
results for waves in first order hyperbolic systems. Numerical
computations illustrate the theory for examples of Nagumo and FitzHugh-Nagumo type.
  }
\end{minipage}
\end{center}
%\end{abstract}

%  ---------------
% |   Key Words   |
%  ---------------

%  --------------------------------
% |   AMS Subject Classification   |
%  --------------------------------

%  -----------------------
% |   Table of contents   |
%  -----------------------
%\tableofcontents

%---------------------------------------------------------------------------------------------------------------------------------------------------
%
%  SECTION 1: (Introduction)
%
%---------------------------------------------------------------------------------------------------------------------------------------------------
\sect{Introduction}
\label{sec:1}
%---------------------------------------------------------------------------------------------------------------------------------------------------
In this paper we study the numerical computation and stability of
traveling waves in second order evolution equations. Our model system 
is a nonlinear wave equation in one space dimension
\begin{equation}
  \begin{aligned}
    \label{equ:2.1}
      Mu_{tt}= Au_{xx} + f(u,u_x,u_t),\,x\in\R,\,t\geqslant 0, u(x,t) \in \R^m.
  \end{aligned}
\end{equation}
Here we use constant matrices $A,M \in \R^{m,m}$ and a
sufficiently smooth nonlinearity $f:\R^{3m}\rightarrow\R^m$. In the numerical
computations we have the simpler case where $f$ is linear in $u_x$ and $u_t$,
i.e.
\begin{equation}
\begin{aligned}
  \label{equ:2.1a}
  f(u,v,w)= g(u)+Cv-Bw, \quad B,C \in \R^{m,m},\; g:\R^m \rightarrow \R^m \;
  \text{smooth},
  \end{aligned}
  \end{equation}
  and $B$ plays the role of a damping matrix.
We also
require $M$ to be invertible and $M^{-1}A$ to be real diagonalizable
with positive eigenvalues (positive diagonalizable for short). This 
ensures that the principal part of equation
\eqref{equ:2.1} is well-posed.

Our main concern are \begriff{traveling wave} solutions
$u_{\star}:\R\times[0,\infty)\rightarrow\R^m$ of \eqref{equ:2.1}, i.e.
\begin{equation}
  \begin{aligned}
  \label{equ:2.2}
    u_{\star}(x,t) = v_{\star}(x-\mu_{\star}t),\;x\in\R,\,t\geqslant 0,
  \end{aligned}
\end{equation}
such that
\begin{equation}
  \label{equ:2.2.1}
  \lim_{\xi\to\pm\infty}v_{\star}(\xi)=v_{\pm}\in\R^m\quad\text{and}\quad
  f(v_{\pm},0,0)=0.
\end{equation}
Here $v_{\star}:\R\rightarrow\R^m$ is a non-constant function
% Jens: Ich habe non-constant eingefuegt, dass haben wir sonst nur
% implizit angenommen!
and denotes the \begriff{profile} (or \begriff{pattern}) of the wave,
$\mu_{\star}\in\R$ its \begriff{translational velocity} and $v_{\pm}$
its \begriff{asymptotic states}. The quantities $v_{\star}$ and
$\mu_{\star}$ are generally unknown, explicit formulas are only
availabe for very specific equations.  As usual, a traveling wave
$u_{\star}$ is called a \begriff{traveling pulse} if $v_{+}=v_{-}$,
and a \begriff{traveling front} if $v_{+}\neq v_{-}$.

We have two main aims for this paper. First, we want to determine
traveling wave solutions of \eqref{equ:2.1} from second order boundary
value problems and investigate their stability for the time-dependent
problem.  Second, we will generalize the method of freezing solutions
of the Cauchy problem associated with \eqref{equ:2.1}, from first
order to second order equations in time
(cf.~\cite{BeynOttenRottmannMatthes2013,BeynThuemmler2004}).
The idea for approximating the traveling wave $u_{\star}$ is to
determine the profile $v_{\star}$ and the velocity $\mu_{\star}$
simultaneously. For this purpose, let us transform \eqref{equ:2.1} via
$u(x,t)=v(\xi,t)$ with $\xi:=x-\mu_{\star}t$ into a co-moving frame
\begin{equation}
  \begin{aligned}
    \label{equ:2.3}
    Mv_{tt} = (A-\mu_{\star}^2 M)v_{\xi\xi} + 2\mu_{\star}Mv_{\xi
      t} + f(v,v_{\xi},v_t-\mu_{\star}v_{\xi}),\;\xi\in\R,\,t\geqslant 0.
  \end{aligned}
\end{equation}
Inserting \eqref{equ:2.2} into \eqref{equ:2.1} shows, that $v_{\star}$
is a stationary solution of \eqref{equ:2.3}, meaning that $v_{\star}$
solves the traveling wave equation
\begin{equation}
  \begin{aligned}
  \label{equ:2.4}
  0 = (A-\mu_{\star}^2 M)v_{\star,\xi\xi}(\xi) +
  f(v_{\star}(\xi),v_{\star,\xi}(\xi),-\mu_{\star} v_{\star,\xi}(\xi)),\;\xi\in\R.
  \end{aligned}
\end{equation}
% For equation \eqref{equ:2.3} to be well-posed we will strengthen our
% assumptions to $M^{-1}A-\mu_{\star}^2 I_m$ being positive
% diagonalizable, which imposes a restriction on the possible wave
% speeds $\mu_{\star}$.  
There are basically two different ways of
determining the profile $v_{\star}$ and the velocity $\mu_{\star}$
from the equations above. In the first approach one solves
\eqref{equ:2.4} as a boundary value problem for
$v_{\star},\mu_{\star}$ by truncating to a finite interval and using
asymptotic boundary conditions as well as a scalar phase condition 
(see \cite{BeynThuemmler2007} for a survey). This method requires 
rather good initial approximations, but has the advantage of being 
applicable to unstable waves as well. The second approach is through 
simulation of \eqref{equ:2.1} via the freezing method which transforms 
the orginal PDE \eqref{equ:2.1} into a partial differential algebraic 
equation (PDAE). Its solutions converge to the unknown profile and the 
unknown velocity simultaneously, provided the initial data lie in the 
domain of attraction of a stable profile.  In Section~\ref{subsec:2.1} 
below we will investigate this approach in more detail. For the numerical
examples we will employ and specify a well known relation of
traveling waves for the hyperbolic system \eqref{equ:2.1}, \eqref{equ:2.1a} to those of
a parabolic system, cf. \cite{GallayRaugel1997,Hadeler1988}
and Section~\ref{subsec:2.2}.

We are also interested in nonlinear stability of traveling waves.
Some far-reaching  global stability results for scalar
damped wave equations have been proved in \cite{GarrayJoly2009,GallayRaugel1997}. Here we consider
local stability only.
For a certain class of first-order evolution equations it
is well-known, that spectral stability implies nonlinear stability,
see \cite{Sandstede2002}, for example. Spectral stability of a traveling wave  refers to the spectrum of the operator obtained
by linearizing about the profile in the co-moving frame. In the case \eqref{equ:2.1}
the  linearization of \eqref{equ:2.3} at 
the wave profile $v_{\star}$ reads
\begin{equation}
  \begin{aligned}
  \label{equ:2.5}
    Mv_{tt}- (A-\mu_{\star}^2 M)v_{\xi\xi} - 2\mu_{\star}Mv_{\xi t}  
    +(\mu_{\star} D_3f(\star)-D_2f(\star))v_{\xi} -D_3f(\star)v_t - D_1f(\star)v= 0,
  \end{aligned}
\end{equation}
where arguments are abbreviated by $(\star)=(v_{\star},v_{\star,\xi}, - \mu_{\star}v_{\star,\xi})$.
Applying separation of variables (or Laplace transform) to  \eqref{equ:2.5} via $v(\xi,t)=e^{\lambda t}w(\xi)$  leads us to the following \begriff{quadratic eigenvalue problem}
\begin{equation}
  \begin{aligned}
  \label{equ:2.6}
    \PL(\lambda)w = \left[\lambda^2 M + \lambda\left(-D_3f(\star) - 2\mu_{\star}M\partial_{\xi}\right) -(A-\mu_{\star}^2 M)\partial_{\xi}^2 
    + (\mu_{\star}D_3f(\star)-D_2f(\star))\partial_{\xi} -
    D_1f(\star)\right]w=0,
  \end{aligned}
\end{equation}
for the eigenfunction $w:\R\rightarrow\C^m$ and its
associated eigenvalue  $\lambda\in\C$ of $\PL$.
As usual $\PL$ has the eigenvalue zero with associated
eigenfunction $v_{\star,\xi}$ due to shift equivariance.
If one requires this eigenvalue to be simple and all other
parts of the spectrum, both essential and point spectrum, to
be strictly to the left of the imaginary axis, then one 
expects the traveling wave to be locally stable with asymptotic phase. This expectation will be confirmed in Section~\ref{sec:3} 
by transforming to a first order hyperbolic system and using the
extensive stability theory developed in  
\cite{RottmannMatthes2010,RottmannMatthes2012a}.
We will also transform the freezing approach and
the spectral problem to the first order formulation. In this
way we obtain a justification of the freezing approach,
showing that the equilibrium $(v_{\star},\mu_{\star})$ of the freezing
PDAE will be stable in the classical Lyapunov sense (w.r.t.\
appropriate norms) provided the conditions on spectral stability
above are satisfied.

Section~\ref{sec:4} is devoted to the study of the spectrum of the
operator $\PL$ from~\eqref{equ:2.6}.  While there is always the zero
eigenvalue present, further isolated eigenvalues in the
point spectrum are often determined by numerical computations (see
\cite{blr14} and the references therein for a variety of
approaches). The essential spectrum can be analyzed by replacing
$v_\star$ in $\PL$ by its limits 
%$v_{\star}(x),x\in \R$, in $\PL$ by its limits
$v_{\pm}$ and the
operator $\partial_{\xi}$ by its Fourier symbol $i
\omega, \omega \in \R$.  The essential spectrum then contains all
values $\lambda \in \C$ satisfying the dispersion relation
\begin{equation}
\label{equ:1.7}
\det\left(\lambda^2 M + \lambda(-D_3f(\pm)-2i \omega\mu_{\star}M)
+\omega^2(A-\mu_{\star}^2M)+i\omega (\mu_{\star}D_3f(\pm)-D_2f(\pm))-
D_1f(\pm)\right)=0
\end{equation}
for some $\omega \in \R$, where the argument is now $(\pm)=(v_{\pm},0,0)$. In Section~\ref{sec:4} we investigate
the shape of these algebraic curves for two examples: a scalar equation with a nonlinearity of Nagumo type and a system of dimension
two with nonlinearity of FitzHugh-Nagumo type.
These examples will also be used for illustrating the effect
of the freezing method from Section~\ref{sec:2} when applied to
 the second order system~\eqref{equ:2.1}.

%---------------------------------------------------------------------------------------------------------------------------------------------------
%
%  SECTION 2: (The freezing method for damped wave equations)
%
%---------------------------------------------------------------------------------------------------------------------------------------------------
\sect{Freezing traveling waves in damped wave equations}
\label{sec:2}
%---------------------------------------------------------------------------------------------------------------------------------------------------
In this section we extend the freezing method
(\cite{BeynOttenRottmannMatthes2013,BeynThuemmler2004}) from
first to second order evolution equations for the case of
translational equivariance.  A generalization to several space
dimensions and more general symmetries is discussed in
\cite{BeynOttenRottmannMatthes2016}.

%---------------------------------------------------------------------------------------------------------------------------------------------------
\subsection{Derivation of the partial differential algebraic equation (PDAE)}
\label{subsec:2.1}
%---------------------------------------------------------------------------------------------------------------------------------------------------
Consider the Cauchy problem associated with \eqref{equ:2.1}
\begin{subequations} 
  \label{equ:2.7}
  \begin{align}
    & Mu_{tt}= Au_{xx} +f(u,u_x,u_t),
    &&x\in\R,\,t\ge 0, \label{equ:2.7a} \\
    & u(\cdot,0) = u_0,\quad u_t(\cdot,0) = v_0,
    &&x\in\R,\,t=0, \label{equ:2.7b}
  \end{align}
\end{subequations}
for some initial data $u_0,v_0:\R\rightarrow\R^m$.  Introducing new
unknowns $\gamma(t)\in\R$ and $v(\xi,t)\in\R^m$ via the
\begriff{freezing ansatz}
\begin{equation}
  \begin{aligned}
  \label{equ:2.8}
  u(x,t) & = v(\xi,t),\quad\xi:=x-\gamma(t),\quad x\in\R,\,t\geqslant
  0,
  \end{aligned}
\end{equation}
we obtain (suppressing arguments)
\begin{align} \label{equ:2.9} u_t &= -\gamma_t v_{\xi} + v_t, \quad
  u_{tt} = -\gamma_{tt} v_{\xi} + \gamma_t^2 v_{\xi\xi} - 2 \gamma_t
  v_{\xi t} + v_{tt}.
\end{align}
Inserting this into \eqref{equ:2.7a} leads to the equation
\begin{equation}
  \begin{aligned}
  \label{equ:2.11}
  Mv_{tt}  & = (A-\gamma_t^2 M)v_{\xi\xi} + 2\gamma_t Mv_{\xi t}
  + \gamma_{tt} M v_{\xi} +f(v,v_{\xi},v_t-\gamma_t v_{\xi}),\quad\xi\in\R,\,t\geqslant 0.
  \end{aligned}
\end{equation}
It is convenient to introduce the time-dependent functions
$\mu_1(t)\in\R$ and $\mu_2(t)\in\R$ via
\begin{equation*}
  \begin{aligned}
    \mu_1(t):=\gamma_t(t),\quad \mu_2(t):=\mu_{1,t}(t) = \gamma_{tt}(t),
  \end{aligned}
\end{equation*}
which transform \eqref{equ:2.11} into the coupled PDE/ODE system
\begin{subequations} 
  \label{equ:2.12}
  \begin{align}
    &Mv_{tt} = (A-\mu_1^2 M)v_{\xi\xi} + 2\mu_1 Mv_{\xi t} +
    \mu_2 M v_{\xi} + f(v,v_{\xi},v_t-\mu_1 v_{\xi}), &&\xi\in\R,\,t\geqslant 0,
    \label{equ:2.12a} \\
    &\mu_{1,t} = \mu_2, &&t\geqslant 0,\label{equ:2.12b} \\
    &\gamma_t = \mu_1, &&t\geqslant 0. \label{equ:2.12c}
  \end{align}
\end{subequations}
The quantity $\gamma(t)$ denotes the \begriff{position}, $\mu_1(t)$
the \begriff{translational velocity} and $\mu_2(t)$ the
\begriff{acceleration} of the wave $v$ at time $t$. 
%Note that system
%\eqref{equ:2.12a} may become ill-posed during integration if the
%leading matrix $M^{-1}A-\mu_1(t)^2 I_m$ loses positivity due to large values
%of $|\mu_1(t)|$. This may occur even if $M^{-1}A-\mu_{\star}^2 I_m$ is positive
%at the final speed $\mu_{\star}$ of the wave, cf. the numerical experiments 
%in Section \ref{subsec:2.3}. 
We next specify initial data for the system \eqref{equ:2.12} as follows,
\begin{equation}
  \begin{aligned}
  \label{equ:2.13}
    v(\cdot,0) = u_0,\quad v_t(\cdot,0)=v_0+\mu_1^0 u_{0,\xi},\quad \mu_1(0)=\mu_1^0,\quad \gamma(0)=0.
  \end{aligned}
\end{equation}
Note that, requiring $\gamma(0)=0$ and $\mu_1(0)=\mu_1^0$, the first equation in \eqref{equ:2.13} follows from \eqref{equ:2.8} and \eqref{equ:2.7b}, while 
the second condition in \eqref{equ:2.13} can be deduced from \eqref{equ:2.9}, \eqref{equ:2.7b}, \eqref{equ:2.12c}.
At first glance the initial value $\mu_1^0$ can be taken arbitrarily and  set to
zero, for example.  But, depending on the solver used, it can be advantageous to define $\mu_1^0$ 
such that it is consistent with the algebraic constraint to be
discussed below.
%\todo[inline]{But we never do in the experiments! Denny please check
%}

To compensate the extra variable $\mu_2$ in the system \eqref{equ:2.12}, we impose an additional scalar algebraic constraint, also known as a \begriff{phase condition}, 
of the general form
\begin{align}
  \label{equ:2.14}
  \psi^{\mathrm{2nd}}(v,v_t,\mu_1,\mu_2)=0,\,t\geqslant 0. 
\end{align}
Together with \eqref{equ:2.12} this will lead to a partial
differential algebraic equation (PDAE).
For the phase condition we require that it vanishes at the traveling
wave solution
\begin{align} \label{equ:2.14a}
\psi^{\mathrm{2nd}}(v_\star,0,\mu_\star,0)=0.
\end{align}
%\todo[inline]{I inserted this, it is needed below!}
In essence, this condition singles out one element from the
family of shifted profiles $v_{\star}(\cdot -\gamma),\gamma \in \R$.

In the following we discuss two possible choices for a phase
condition:
% \todo[inline]{I have discovered a small problem with the phase
%   conditions, which is actually not a problem and really clarifies (in
%   my opinion) the relation of the first order and second order
%   formulation.
%   More precisely, the trouble is with the initial data of the
%   algebraic variables:\\
%   We use the ansatz $u(x,t)=v(x-\gamma(t),t)$ which is equivalent to
%   $v(\xi,t)=u(\xi+\gamma(t),t)$.  The solution really satisfies the
%   differential equation down to the starting time $t=0$.  Therefore,
%   we obtain for the fixed phase condition:
%   \[
%   0=\langle v-\hat{v},\hat{v}_\xi\rangle = \langle
%   v_t,\hat{v}_\xi\rangle.\]
%   Now recall
%   $v_t(\xi,t)=\gamma'(t)u_x(\xi+\gamma(t),t)+u_t(\xi+\gamma(t),t)$, so
%   that at $t=0$ we have
%   $v_t(\xi,0)=\mu_1^0 u_{0,x}(\xi)+v_0$ which can be inserted into the
%   differentiated phase condition and yields
%   \[0=\langle \mu_1^0 u_{0,x}+v_0,\hat{v}_\xi\rangle.\]
%   This can be solved for $\mu_1^0$, so there is no choice! Moreover,
%   As we already saw, also $\mu_2(0)$ is uniquely given, once
%   $\mu_1(0)$ is known (recall, that in \eqref{equ:2.17} all terms
%   beside $\mu_2(0)$ are known at $t=0$ if $u_0\in v_\star+H^2$ and
%   $v_0\in H^1$).
%   The only choice is for $\gamma(0)$! (As it should be!)}

\textbf{Type 1:} (\textbf{fixed phase condition}).  Let
$\hat{v}:\R\rightarrow\R^m$ denote a time-independent and sufficiently
smooth \begriff{template} (or \begriff{reference}) \begriff{function},
e.g. $\hat{v}=u_0$. Then we consider the following \begriff{fixed
  phase condition}
\begin{equation}  \label{equ:2.15}
  \psi_{\mathrm{fix},3}^{\mathrm{2nd}}(v) 
  := \langle v-\hat{v},\hat{v}_{\xi} \rangle_{L^2} = 0,\;t\geqslant 0.
\end{equation} 
This condition is obtained from minimizing the $L^2$-distance of the
shifted versions of $v$ from the template $\hat{v}$ at each time
instance
\begin{align*}
  \rho(\gamma) 
  := \left\|v(\cdot,t)-\hat{v}(\cdot -\gamma)\right\|^2_{L^2}=
  \left\|v(\cdot+\gamma,t)-\hat{v}(\cdot)\right\|^2_{L^2}.
\end{align*}
%  i.e. $v(\cdot,t)$ is to be chosen such that 
% \begin{equation*}
%   \begin{aligned}
%     \min_{\gamma\in\R}\left\|v(\cdot,t)-\hat{v}(\cdot-\gamma)\right\|_{L^2}^2 & = \left\|v(\cdot,t)-\hat{v}(\cdot)\right\|_{L^2}^2 &&,\,t\geqslant 0.
%   \end{aligned}
% \end{equation*}
The necessary condition for a local minimum to occur at $\gamma=0$ is
\begin{align*}
 % \label{equ:2.18}
  0\overset{!}{=}\left[\frac{d}{d\gamma}
    \left\langle  v(\cdot,t)-\hat{v}(\cdot-\gamma),
      v(\cdot,t)-\hat{v}(\cdot-\gamma)  \right\rangle_{L^2}
  \right]_{\gamma=0} = 2
  \left\langle v(\cdot,t)-\hat{v},\hat{v}_{\xi}\right\rangle_{L^2},
  \;t\geqslant 0.
\end{align*}
To reduce the index of the resulting PDAE, we
differentiate \eqref{equ:2.15} w.r.t. $t$ and obtain 
\begin{equation}\label{equ:2.16} 
  \psi_{\mathrm{fix},2}^{\mathrm{2nd}}(v_t) 
  := \langle v_t,\hat{v}_{\xi} \rangle_{L^2} = 0,\;t\geqslant 0.
\end{equation}
Finally, differentiating \eqref{equ:2.16} once more w.r.t. $t$ and
using equation \eqref{equ:2.12a} yields the following
condition
\begin{equation}  \label{equ:2.17}
\begin{aligned}
  \psi_{\mathrm{fix},1}^{\mathrm{2nd}}(v,v_t,\mu_1,\mu_2) := 
  & \langle
  (M^{-1}A-\mu_1^2 I_m)v_{\xi\xi} + 2\mu_1  v_{\xi t}
    + M^{-1} f(v,v_{\xi},v_t-\mu_1 v_{\xi}),\hat{v}_{\xi} \rangle_{L^2}  \\
  &+ \mu_2 \langle v_{\xi},\hat{v}_{\xi} \rangle_{L^2} =
  0,\;t\geqslant 0.
\end{aligned}
\end{equation}
Note that equation \eqref{equ:2.17} can be explicitly solved for
$\mu_2$, if the template $\hat{v}$ is chosen such that 
$\langle v_{\xi},\hat{v}_{\xi} \rangle_{L^2}\neq 0$
for any $t\geqslant 0$.

The numbers $j=1,2,3$ in the notation
$\psi_{\mathrm{fix},j}^{\mathrm{2nd}}$ above indicate the index of the
resulting PDAE (in a formal sense) as the minimum number of
differentiations with respect to $t$, necessary to obtain an explicit
differential equation for the unknowns $(v,\mu_1,\mu_2)$
(cf. \cite[Ch.~1]{HairerLubichRoche1989},
\cite[Ch.~2]{BrenanCampbellPetzold1996}). In general, the value of this
(differential) index may depend on the system formulation. For
example, if we do not introduce $\mu_2$, but omit \eqref{equ:2.12b}
from the system and replace $\mu_2$ by $\mu_{1,t}$ in
\eqref{equ:2.12a}, then we need only two differentiations to obtain an
explicit differential equation for $(v,\mu_1)$. Hence the index is
lowered by one (this methodology is described in the ODE setting in
\cite[Prop.~2.5.3]{BrenanCampbellPetzold1996}).

Let us note that the index $2$ formulation \eqref{equ:2.16} and
the index $1$ formulation \eqref{equ:2.17} enforce constraints
on  $\mu_1(0)=\mu_1^0$ and $\mu_2(0)=\mu_2^0$ in order
to have consistent initial values. Setting $t=0$ in \eqref{equ:2.16}
and using \eqref{equ:2.13} yields the condition
\begin{equation} \label{equ:2.13consist}
 \mu_1^0 \langle u_{0,\xi},\hat{v}_{\xi} \rangle_{L^2} + 
\langle v_0, \hat{v}_{\xi}\rangle_{L^2} =0,
\end{equation}
from which $\mu_1^0$ can be determined.
% initial values $\mu_1(0)=\mu_1^0$ and $\mu_2(0)=\mu_2^0$. Setting
% $\mu_2^0=0$ and using \eqref{equ:2.13} in \eqref{equ:2.17} gives a
% quadratic equation for $\mu_1^0$,
Further, setting $t=0$ in \eqref{equ:2.17} and using \eqref{equ:2.13} 
leads to an equation from which one can determine $\mu_2^0$ 
%as a function of $\mu_1^0$,
from the remaining initial data
\begin{equation} \label{equ:2.14consist} 
  0 = \langle (M^{-1}A+(\mu_1^0)^2I_m) u_{0,\xi \xi} 
  + 2 \mu_1^0 v_{0,\xi} 
  +  M^{-1}f(u_0,u_{0,\xi},v_0),
  \hat{v}_{\xi} \rangle_{L^2} + \mu_2^0 \langle
  u_{0,\xi},\hat{v}_{\xi}
\rangle_{L^2}.
\end{equation}

\textbf{Type 2:} (\textbf{orthogonal phase condition}). The
\begriff{orthogonal phase conditions} read as follows:
\begin{equation}
  \psi_{\mathrm{orth},2}^{\mathrm{2nd}}(v,v_t) := \langle v_t,v_{\xi} \rangle_{L^2} = 0,\;t\geqslant 0, \label{equ:2.15_orth} 
\end{equation}
\begin{equation}\label{equ:2.16_orth}
\begin{aligned}
  \psi_{\mathrm{orth},1}^{\mathrm{2nd}}(v,v_t,\mu_1,\mu_2) := &
  \langle (M^{-1}A-\mu_1^2 I_m)v_{\xi\xi} + 2\mu_1 v_{\xi t}
    + M^{-1} f(v,v_{\xi},v_t-\mu_1v_{\xi}),v_{\xi} \rangle_{L^2}  \\
  & + \langle v_t,v_{\xi t} \rangle_{L^2} + \mu_2 \langle
  v_{\xi},v_{\xi} \rangle_{L^2} = 0,\;t\geqslant 0.
\end{aligned}
\end{equation}
For first order evolution equations, condition \eqref{equ:2.15_orth}
has an immediate interpretation as a necessary condition for
minimizing $\|v_t\|_{L^2}$ (cf. \cite{BeynOttenRottmannMatthes2013}).
The same interpretation is possible here when applied to a
proper formulation as a first order system (cf. \cite[(4.46)]{BeynOttenRottmannMatthes2016b}). 
For the moment, our motivation is, that this condition expresses orthogonality of
$v_t$ to the vector $v_{\xi}$ tangent to the group orbit
$\{v(\cdot-\gamma): \gamma \in \R\}$ at $\gamma=0$.  
For a different kind of orthogonal phase condition that relies on the
formulation as a first order system, see \cite[(4.45)]{BeynOttenRottmannMatthes2016b}. 
The condition \eqref{equ:2.15_orth} leads to a PDAE of index $2$ in the sense
above. Differentiating \eqref{equ:2.15_orth} w.r.t.\ $t$ and using
\eqref{equ:2.12a} implies \eqref{equ:2.16_orth} which yields a PDAE of
index $1$.  Note that equation \eqref{equ:2.16_orth} can be explicitly
solved for $\mu_2$, provided that $\langle v_{\xi},v_{\xi}
\rangle_{L^2}\neq 0$ for any $t\geqslant 0$.

Similar to the type $1$ phase condition, 
we obtain constraints for consistent initial values 
when setting $t=0$ in  \eqref{equ:2.15_orth}, \eqref{equ:2.16_orth}.
Condition \eqref{equ:2.15_orth} leads to an equation for $\mu_1^0$
\begin{equation} \label{equ:2.15consist}
0 = \mu_1^0 \langle u_{0,\xi},u_{0,\xi}\rangle_{L^2} +
           \langle v_0,u_{0,\xi}\rangle_{L^2},
\end{equation}
while \eqref{equ:2.16_orth}, \eqref{equ:2.7b}, \eqref{equ:2.8} give an equation 
%from which one can determine $\mu_2^0$ as a function of $\mu_1^0$,
for $\mu_2^0$
\begin{equation}\label{equ:2.16consist}
\begin{aligned} 
  0 = &\langle 2(\mu_1^0)^2 u_{0,\xi \xi} +
  3\mu_1^0 v_{0,\xi}+ M^{-1}\left(Au_{0,\xi \xi}+ 
    f(u_0,u_{0,\xi},v_0)\right), u_{0,\xi} \rangle_{L^2}\\
 &+  \langle v_0,v_{0,\xi}\rangle_{L^2} + 
\mu_1^0\langle v_0,u_{0,\xi\xi}\rangle_{L^2} + 
\mu_2^0 \langle u_{0,\xi},u_{0,\xi} \rangle_{L^2}.
\end{aligned}
\end{equation}

Let us summarize the system of equations obtained by
the freezing method from the original Cauchy problem 
\eqref{equ:2.7}.
Combining the differential equations \eqref{equ:2.12}, the initial
data \eqref{equ:2.13} and the phase condition \eqref{equ:2.14}, we
arrive at the following PDAE to be solved numerically:
\begin{subequations}    \label{equ:2.19}
  \begin{align}
     \label{equ:2.19a}
     &\begin{aligned}
      Mv_{tt} &= (A-\mu_1^2 M)v_{\xi\xi} + 2\mu_1 Mv_{\xi,t} +
      \mu_2 M v_{\xi} + f(v,v_{\xi},v_t - \mu_1 v_{\xi}), \;\\
      \mu_{1,t} &= \mu_2, \quad
      \gamma_t=\mu_1,
    \end{aligned}&t\geqslant 0,\\
    & 0 = \psi^{\mathrm{2nd}}(v,v_t,\mu_1,\mu_2), &t\geqslant
    0,\label{equ:2.19b}\\
    &\begin{aligned}
      v(\cdot,0) &= u_0,\quad v_t(\cdot,0) = v_0+\mu_1^0 u_{0,\xi},
      \quad \mu_1(0) = \mu_1^0, \quad
      \gamma(0) = 0.
    \end{aligned}\label{equ:2.19c}
  \end{align}
\end{subequations}
The system \eqref{equ:2.19} depends on the choice of phase condition
$\psi^{\mathrm{2nd}}$ 
and is to be solved for $(v,\mu_1,\mu_2,\gamma)$ with given initial
data $(u_0,v_0,\mu_1^0)$.
It consists of a PDE for $v$ that is coupled to two ODEs for
$\mu_1$ and $\gamma$ \eqref{equ:2.19a}
and an algebraic constraint \eqref{equ:2.19b} which closes the system.
A consistent initial value $\mu_1^0$ for $\mu_1$ is computed from the phase
condition and the initial data 
(cf. \eqref{equ:2.13consist}, \eqref{equ:2.15consist}). Further
initialization of the algebraic variable $\mu_2$ is not needed
for a PDAE-solver but can be provided if necessary (cf.
\eqref{equ:2.14consist}, \eqref{equ:2.16consist}).

The ODE for $\gamma$ is called the \begriff{reconstruction equation}
in \cite{RowleyKevrekidisMarsden2003}. It decouples from the other
equations in \eqref{equ:2.19} and can be solved in a postprocessing
step. The ODE for $\mu_1$ is the new feature of the PDAE for second
order systems when compared to first order parabolic and
hyperbolic equations, cf.
\cite{BeynThuemmler2004,RottmannMatthes2010,BeynOttenRottmannMatthes2013}.

Finally, note that $(v,\mu_1,\mu_2)=(v_{\star},\mu_{\star},0)$ satisfies
\begin{align*}
%  \label{equ:2.19-20}
  \begin{split}
  0 & = (A-\mu_{\star}^2 M)v_{\star,\xi\xi}(\xi) 
  + f(v_{\star}(\xi),v_{\star,\xi}(\xi),-\mu_{\star}v_{\star,\xi}(\xi)),\;\xi\in\R,\\
  0 & = \mu_2,\quad
  0 = \psi^{\mathrm{2nd}}(v_{\star},0,\mu_{\star},0),
  \end{split}
\end{align*}
and hence is a stationary solution of \eqref{equ:2.19a},
\eqref{equ:2.19b}. Obviously, in this case we have
$\gamma(t)=\mu_{\star}t$.  For a stable traveling wave we
expect that solutions $(v,\mu_1,\mu_2,\gamma)$ of \eqref{equ:2.19}
show the limiting behavior
\begin{align*}
 % \label{equ:2.20}
  v(t)\rightarrow v_{\star},\quad \mu_1(t)\rightarrow\mu_{\star},\quad
  \mu_2(t)\rightarrow 0\quad\text{as}\quad t\to\infty,
\end{align*}
provided the initial data are close to their limiting values. In
Section~\ref{sec:3} we will provide theorems that justify this
expectation under suitable conditions.
% As an indicator for the convergence in \eqref{equ:2.20} we check the
% quantities
%\begin{align}
%  \label{equ:15}
%  \left\|v_t(\cdot,t)\right\|_{L^2(\R,\R^m)}\quad
%  \text{and}\quad|\mu_{1,t}(t)|
%\end{align}
%at each time instance $t$ during the computation. In fact, both of
% these quantities should be small ($\approx 10^{-16}$), since
% $v_{\star}$ and $\mu_{\star}$ do not vary in time.

%---------------------------------------------------------------------------------------------------------------------------------------------------
\subsection{Traveling waves related to parabolic equations}
\label{subsec:2.2}
%---------------------------------------------------------------------------------------------------------------------------------------------------
The following proposition shows an important relation between
traveling waves \eqref{equ:2.2} of the damped wave equation
\eqref{equ:2.1}, \eqref{equ:2.1a} and
traveling waves
\begin{align} 
  \label{equ:2.31}
  u_{\star}(x,t)=w_{\star}(x-c_{\star}t),\; 
  x\in\R,\,t\geqslant 0,
  % \quad\text{such that}
  % \quad\lim_{\zeta \rightarrow \pm
  % \infty}w_{\star}(\zeta)=w_{\pm}\in\R^m,
\end{align}
with nonvanishing speed $c_\star$ of the parabolic equation
\begin{align} 
  \label{equ:2.32}
  B u_t = \tilde{A} u_{xx} + \tilde{C}u_x+g(u),\; x\in\R,\, t\ge 0.
\end{align}
The matrices $\tilde{A},\tilde{C} \in\R^{m,m}$ in \eqref{equ:2.32} may
differ from $A,C$ in \eqref{equ:2.1}, \eqref{equ:2.1a}.  This observation goes back to
\cite{Hadeler1988} and has also been used in
\cite{GallayRaugel1997}. Note that in this case $w_{\star}:\R\to \R^m$ solves the
traveling wave equation
\begin{align} 
  \label{equ:2.33}
  0 = \tilde{A} w_{\star,\zeta\zeta} + c_{\star} B w_{\star,\zeta} +
\tilde{C}w_{\star,\zeta}+ g(w_{\star}),\;\zeta\in\R.
\end{align}

\begin{proposition} \label{pro1} 
  \begin{itemize}[leftmargin=0.65cm]\setlength{\itemsep}{0cm}
  \item[(i)] Let \eqref{equ:2.31} be a
  traveling wave of the parabolic equation \eqref{equ:2.32}. Then for
  every $0\neq k\in\R$ and $A,C,M\in\R^{m,m}$, satisfying
  $\tilde{A}=k^2A-c_{\star}^2 M$, $\tilde{C}=kC$, equation
  \eqref{equ:2.2} with
  \begin{align}
    \label{equ:2.34}
    v_{\star}(\xi)=w_{\star}(k\xi),\quad
    \mu_{\star}=\frac{c_{\star}}{k}
  \end{align}
  defines a traveling wave of the damped wave equation
  \eqref{equ:2.1}, \eqref{equ:2.1a}.
  \item[(ii)] Conversely, let \eqref{equ:2.2} be a traveling wave of
  \eqref{equ:2.1}, \eqref{equ:2.1a}.  Then for every $0\neq k\in\R$ equation
  \eqref{equ:2.31} with
  \begin{align}
    \label{equ:2.35}
    w_{\star}(\zeta)=v_{\star}(\frac{\zeta}{k}),\quad
    c_{\star}=\mu_{\star}k
  \end{align}
  defines a traveling wave of \eqref{equ:2.32} with
  $\tilde{A}=k^2(A-\mu_{\star}^2 M)$, $\tilde{C}=kC$.
  \end{itemize}
\end{proposition}

\begin{proof}
  \begin{itemize}[leftmargin=0.65cm]\setlength{\itemsep}{0cm}
  \item[(i)] By assumption, $w_\star$ satisfies
    \eqref{equ:2.33}.  Let $0\neq k\in\R$ and
    $A,C,M\in\R^{m,m}$ be such that $\tilde{A}=k^2A-c_\star^2M$,
    $\tilde{C}=kC$ hold and define $v_\star, \mu_\star$ by
    \eqref{equ:2.34}.  Then $u_\star(x,t)=v_\star(x-\mu_\star
    t)=w_\star\bigl(k(x-\mu_\star t)\bigr)$ satisfies
    \[
      - Mu_{\star,tt} - Bu_{\star,t} + Au_{\star,xx} + Cu_{\star,x}+g(u_{\star})
      = \tilde{A}w_{\star,\zeta\zeta} + c_{\star}Bw_{\star,\zeta} + \tilde{C}w_{\star,\zeta}+ g(w_{\star})
      = 0.
    \]
% JENS: Ich habe versucht u_\star nicht mehrfach hier stehen zu haben    
%    Let $u_{\star}(x,t)=w_{\star}(x-c_{\star}t)$ from
%  \eqref{equ:2.31} be a traveling wave of \eqref{equ:2.32}, then
%  $w_{\star}$ satisfies \eqref{equ:2.33}.  Now, let $0\neq k\in\R$ and
%  $A,M\in\R^{m,m}$ be such that $\tilde{A}=k^2A-c_{\star}^2M$ is
%  satisfied and define $(v_{\star},\mu_{\star})$ as in
%  \eqref{equ:2.34}.  Then, $u_{\star}$ from \eqref{equ:2.2}
%  satisfies $u_{\star}(x,t) = v_{\star}(x-\mu_{\star}t) = w_{\star}(k(x-\mu_{\star}t))$ 
%  and
%  \begin{align*}
%    - Mu_{\star,tt} - Bu_{\star,t} + Au_{\star,xx} + Cu_{\star,x}+g(u_{\star})
%    = \tilde{A}w_{\star,\zeta\zeta} + c_{\star}Bw_{\star,\zeta} + \tilde{C}w_{\star,\zeta}+ g(w_{\star})
%    = 0.
%  \end{align*}
%  %\begin{align*}
%  %  u_{\star}(x,t) = v_{\star}(x-\mu_{\star}t) 
%  %  = w_{\star}(k(x-\mu_{\star}t))
%  %\end{align*}
%  %satisfies 
%  %\begin{align*}
%  %  - Mu_{\star,tt} - Bu_{\star,t} + Au_{\star,xx} + 
%  %  & \, Cu_{\star,x}+f(u_{\star})
%  %  = - k^2\mu_{\star}^2 Mw_{\star,\zeta\zeta} + k\mu_{\star}Bw_{\star,\zeta}
%  %  + k^2Aw_{\star,\zeta\zeta}  \\
%  %  + &\,  kCw_{\star,\zeta}+ f(w_{\star})
%  %  = \tilde{A}w_{\star,\zeta\zeta} + c_{\star}Bw_{\star,\zeta} 
%  %  + \tilde{C}w_{\star,\zeta}+ f(w_{\star})
%  %  = 0.
%  %\end{align*}
  \item[(ii)] By assumption, $v_\star,\mu_\star$ from \eqref{equ:2.2}
    satisfy \eqref{equ:2.4}.  Let $0\neq k\in\R$ and define
    $\tilde{A}:=k^2 (A-\mu_{\star}^2 M), \tilde{C}:=kC\in\R^{m,m}$
    and $w_{\star},c_{\star}$ by \eqref{equ:2.35}. Then 
    $u_{\star}(x,t) = w_{\star}(x-c_{\star}t) 
    = v_{\star}\left(\frac{x-c_{\star}t}{k}\right)$ satisfies
    \[
      -Bu_{\star,t} + \tilde{A}u_{\star,xx} + \tilde{C}u_{\star,x}+ g(u_{\star})
      = (A-\mu_{\star}^2 M)v_{\star,\xi\xi} + \mu_{\star}Bv_{\star,\xi} + C v_{\star,\xi}+ g(v_{\star})
      = 0.
    \]
%    Jens: Ich wollte die doppelte definition von u_\star loswerden...
%    Let $u_{\star}(x,t)=v_{\star}(x-\mu_{\star}t)$ from
%  \eqref{equ:2.2} be a traveling wave of \eqref{equ:2.1}\eqref{equ:2.1a}, then
%  $v_{\star}$ satisfies \eqref{equ:2.4}.  Now, let $0\neq k\in\R$
%  and define $\tilde{A}:=k^2(A-\mu_{\star}^2 M)\in\R^{m,m}$,
%  $\tilde{C}:=kC$ and $(w_{\star},c_{\star})$ as in
%  \eqref{equ:2.35}. Then, $u_{\star}$ from \eqref{equ:2.31} 
%  satisfies $u_{\star}(x,t) = w_{\star}(x-c_{\star}t) = v_{\star}\left(\frac{x-c_{\star}t}{k}\right)$ 
%  and
%  \begin{align*}
%      -Bu_{\star,t} + \tilde{A}u_{\star,xx} + \tilde{C}u_{\star,x}+ g(u_{\star})
%    = (A-\mu_{\star}^2 M)v_{\star,\xi\xi} + \mu_{\star}Bv_{\star,\xi} + C v_{\star,\xi}+ g(v_{\star})
%    = 0.
%  \end{align*}
%  %\begin{align*}
%  %  u_{\star}(x,t) = w_{\star}(x-c_{\star}t) = v_{\star}\left(\frac{x-c_{\star}t}{k}\right)
%  %\end{align*}
%  %satisfies
%  %\begin{align*}
%  %    -Bu_{\star,t} + \tilde{A}u_{\star,xx} + \tilde{C}u_{\star,x}+ f(u_{\star})
%  %  =& \frac{c_{\star}}{k}Bv_{\star,\xi} + \frac{1}{k^2}\tilde{A}v_{\star,\xi\xi} 
%  %   + \frac{1}{k}\tilde{C}v_{\star,\xi}+ f(v_{\star}) \\
%  %  =& (A-\mu_{\star}^2 M)v_{\star,\xi\xi} + \mu_{\star}Bv_{\star,\xi} 
%  %      + C v_{\star,\xi}+ f(v_{\star})
%  %  = 0.
%  %\end{align*}
  \end{itemize}
\end{proof}

According to Proposition~\ref{pro1}, any  traveling wave \eqref{equ:2.31} of the parabolic equation \eqref{equ:2.32} 
leads to a traveling wave \eqref{equ:2.2} of the damped wave equation \eqref{equ:2.1},\eqref{equ:2.1a} and vice versa.

\begin{remark} \label{rem1} Note that the profiles
  $v_{\star},w_{\star}$ and the velocities $\mu_{\star},c_{\star}$
  coincide if $k=1$. In this case $\tilde{A}=A-c_{\star}^2 M$, and the
  matrices $A$ and $\tilde{A}$ are different (provided $c_{\star}\neq
  0$).  If we insist on $A=\tilde{A}$ then the profiles will be
  different.

  In case $C=0$ both systems \eqref{equ:2.1}, \eqref{equ:2.1a} and
  \eqref{equ:2.32} share
  a symmetry property: if $v_{\star}(\xi)(\xi \in \R),c_{\star}$
  resp.  $w_{\star}(\zeta)(\zeta \in \R),\mu_{\star}$ is a traveling
  wave then so is the reflected pair $v_{\star}(-\xi) (\xi \in
  \R),-c_{\star}$ resp.  $w_{\star}(-\zeta)(\zeta \in
  \R),-\mu_{\star}$. Thus, choosing $k <0$ in \eqref{equ:2.34}
  resp. \eqref{equ:2.35} will not produce new waves other than those
  induced by reflection symmetry. Therefore, we will assume $k$ to be
  positive in the following.

  It is instructive to consider two limiting cases of the
  transformation \eqref{equ:2.34} when a traveling wave $w_{\star}$
  with velocity $c_{\star}\neq 0$ is given for the parabolic equation
  \eqref{equ:2.32}.

  First assume $A= \tilde{A}$ and let $M \rightarrow 0$. Then the
  relation $\tilde{A}=k^2A -c_{\star}^2M$ implies $k \rightarrow 1$
  and $v_{\star} \rightarrow w_{\star}$, $\mu_{\star} \rightarrow
  c_{\star}$. Thus the profile and the velocity of the traveling waves
  \eqref{equ:2.2} of the system~\eqref{equ:2.1}, \eqref{equ:2.1a} converge to the
  correct limit in the parabolic case.  Second, consider the scalar
  case, fix $A>0$ and let $M \rightarrow \infty$. Then the relation
  $\tilde{A}=k^2A -c_{\star}^2M$ implies $k \rightarrow \infty$ and
  $\mu_{\star} =\frac{c_{\star}}{k}\rightarrow 0$ . Thus a large value
  of $M$ creates a slow wave for the system~\eqref{equ:2.1}, \eqref{equ:2.1a}
  which has
  steep gradients in its profile due to $v_{\star,\xi}(\xi)=k
  w_{\star,\zeta}(k \xi)$.
\end{remark}

%---------------------------------------------------------------------------------------------------------------------------------------------------
\subsection{Applications and numerical examples}
\label{subsec:2.3}
%---------------------------------------------------------------------------------------------------------------------------------------------------

In the following we consider two examples with nonlinearities of
Nagumo and FitzHugh-Nagumo type. We use the mechanism from
Proposition~\ref{pro1} to obtain traveling waves of these damped wave
equations. Then we solve the PDAE~\eqref{equ:2.19} providing us with wave
profiles, their positions, velocities and accelerations. All numerical 
computations in this paper were done with Comsol Multiphysics 5.2, 
\cite{ComsolMultiphysics52}. Specific data of time and space discretization 
are given below.

\begin{example}[Nagumo wave equation] \label{exa1} 
  Consider the scalar
  parabolic Nagumo equation, \cite{Miura1981,Murray1989},
  \begin{align} 
    \label{equ:2.36}
    u_t = u_{xx} + g(u) ,\; x\in\R,\, t\ge 0, \quad g(u)=u(1-u)(u-b),
  \end{align}
  with $u=u(x,t)\in\R$ and some fixed $b\in(0,1)$.  It is well known
  that \eqref{equ:2.36} has an explicit traveling front solution
  $u_{\star}(x,t)=w_{\star}(x-c_{\star}t)$ given by
  \begin{align*} 
  %  \label{equ:2.37}
    w_{\star}(\zeta) =
    \left(1+\exp\left(-\tfrac{\zeta}{\sqrt{2}}\right)\right)^{-1},
    \quad c_{\star}=-\sqrt{2}\left(\frac{1}{2}-b\right),
  \end{align*}
  with asymptotic states $w_{-}=0$ and $w_{+}=1$. Note that
  $c_{\star}<0$ if $b<\frac{1}{2}$ and $c_{\star}>0$ if
  $b>\frac{1}{2}$.  Proposition~\ref{pro1}(i) implies that the
  corresponding Nagumo wave equation
  \begin{align}
    \label{equ:2.38}
    \varepsilon u_{tt} + u_t = u_{xx} + g(u),\; x\in \R,\, t\ge
    0,
  \end{align}
  has a traveling front solution
  $u_{\star}(x,t)=v_{\star}(x-\mu_{\star}t)$ given by
  \begin{align}
    \label{equ:2.39}
    v_{\star}(\xi)=w_{\star}(k\xi),\quad
    \mu_{\star}=\frac{-\sqrt{2}\left(\frac{1}{2}-b\right)}{k},\quad k=
    \left(1+2\varepsilon\left(\frac{1}{2}-b\right)^2\right)^{1/2}.
  \end{align}
  % Depending on the sign of $k=k_{\pm}$ from \eqref{equ:2.39}, the
  % asymptotic states are $v_{+}=1$, $v_{-}=0$ for $k_+$ and
  % $v_{+}=0$, $v_{-}=1$ for $k_-$.
  Figure~\ref{fig:DampedNagumoTraveling1Front} shows a numerical
  approximation of the time
  evolution of the traveling front solution $u$ of \eqref{equ:2.38} on
  the spatial domain $(-50,50)$ with homogeneous Neumann
  boundary conditions and initial data
  \begin{align} \label{equ:initdata} 
    u_0(x)=\frac{1}{\pi}\arctan(x)+\frac{1}{2},\quad v_0(x)=0,\quad x\in(-50,50).
    %u_0=\frac{1}{2}(v_{\star}+
    %v_{\mathrm{ramp}}),\quad v_0=0,\quad
    %v_{\mathrm{ramp}}(x)=\begin{cases}0&,\,x<-50\\
    %  \frac{x+50}{100}&,\,-50\leqslant
    %  x\leqslant 50\\1&,\,x>50\end{cases},
\end{align}
%where $v_{\star}$ is from \eqref{equ:2.39}.
%% Jens: Dieses wird gar nicht benutzt:
% and $k$ are  
Further parameter values are $\varepsilon=b=\frac{1}{4}$. For the space 
discretization we used continuous piecewise linear finite elements with 
spatial stepsize $\triangle x=0.1$. For the time discretization we used 
the BDF method of order $2$ with absolute tolerance $\mathrm{atol}=10^{-3}$, 
relative tolerance $\mathrm{rtol}=10^{-2}$, temporal stepsize $\triangle t=0.1$
and final time $T=150$.

\begin{figure}[ht]
  \centering
  \subfigure[]{\includegraphics[height=3.8cm] {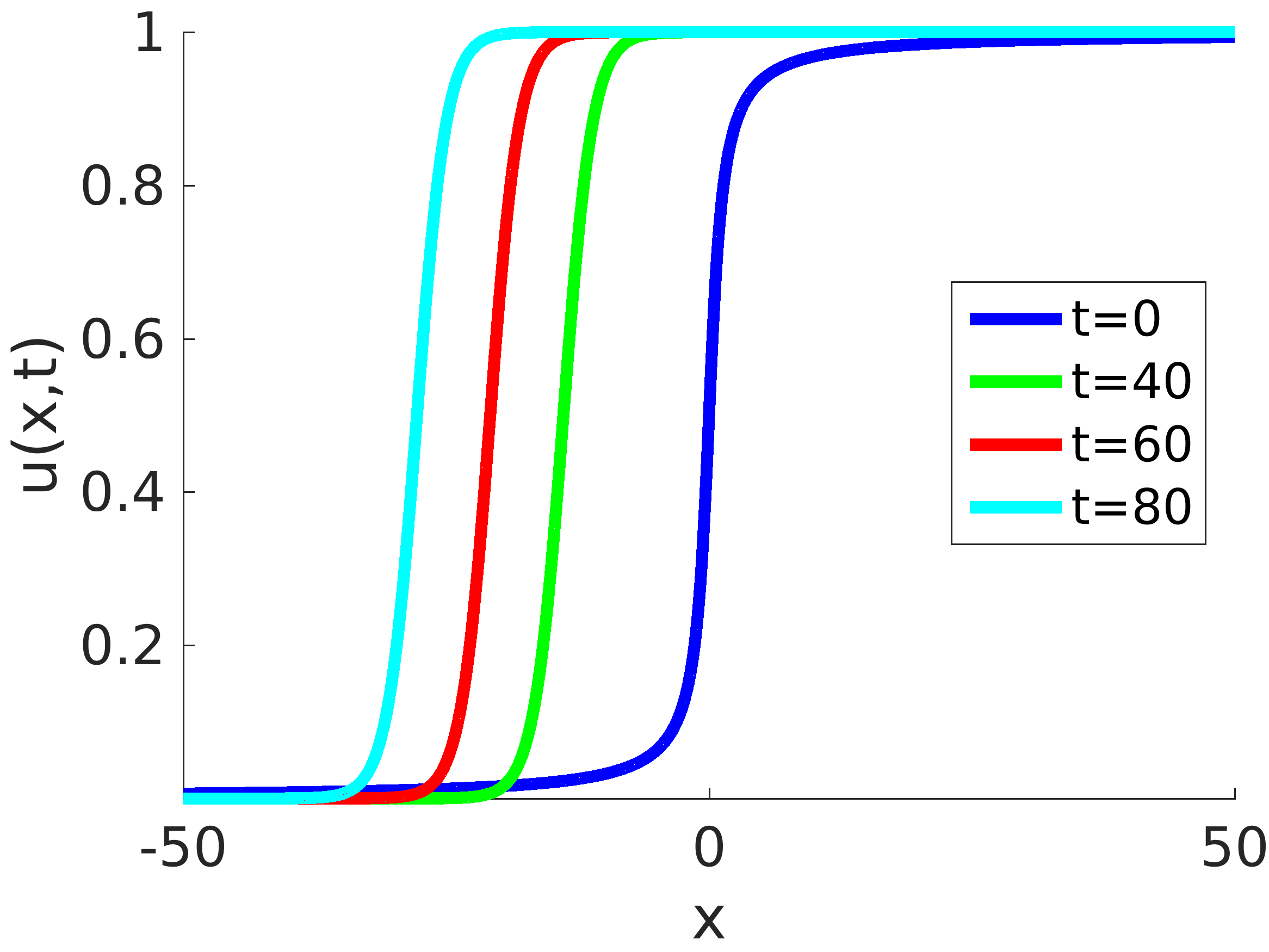}\label{fig:DampedNagumoTraveling1FrontEndtime}}
  \subfigure[]{\includegraphics[height=3.8cm] {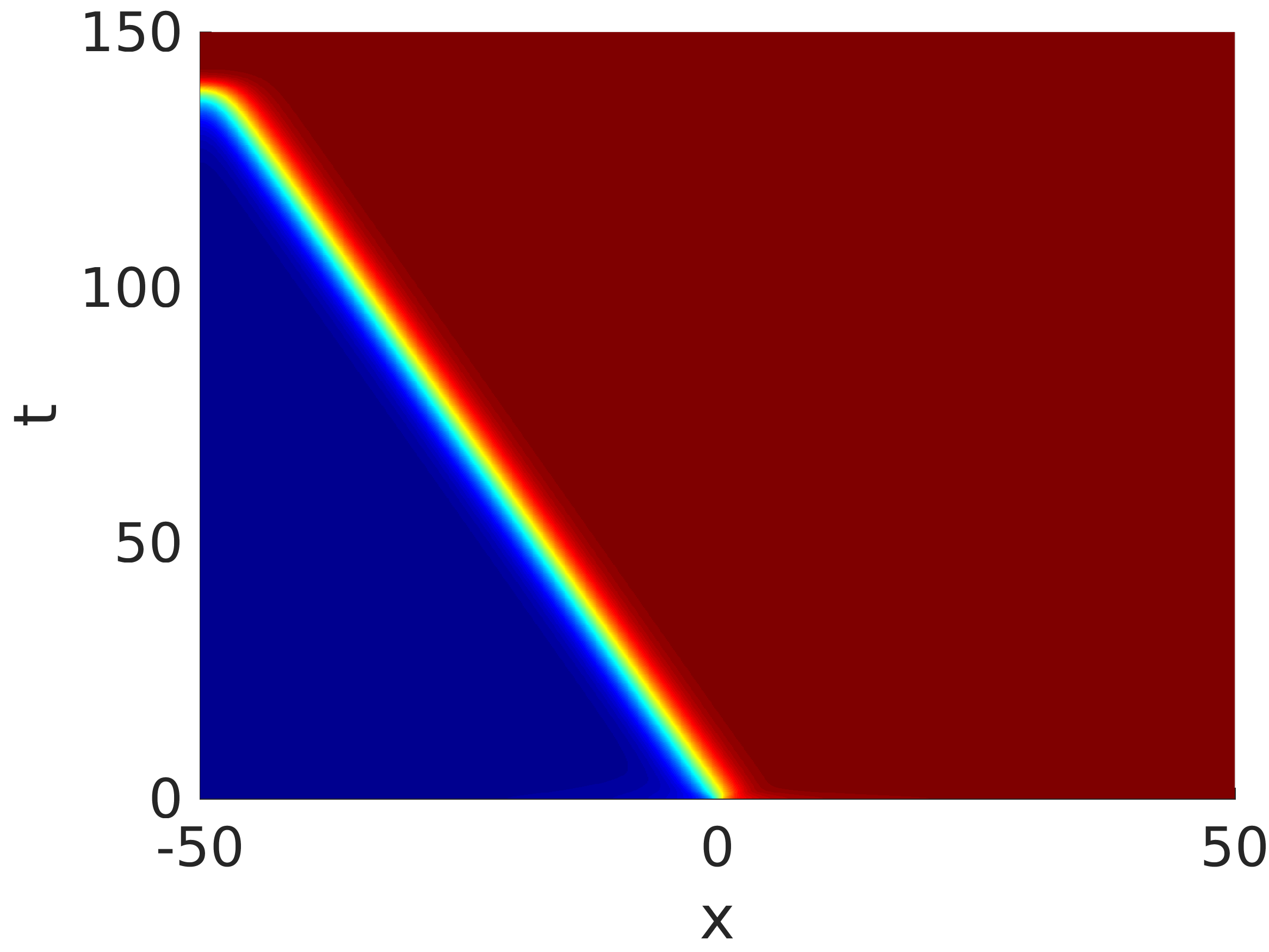} \label{fig:DampedNagumoTraveling1FrontSpaceTime}}
  \caption{Traveling front of Nagumo wave equation \eqref{equ:2.38} at different time instances (a) and its time evolution (b) for parameters $\varepsilon=b=\frac{1}{4}$.}
  \label{fig:DampedNagumoTraveling1Front}
\end{figure} 

\begin{figure}[ht]
  \centering
  \subfigure[]{\includegraphics[height=3.8cm] {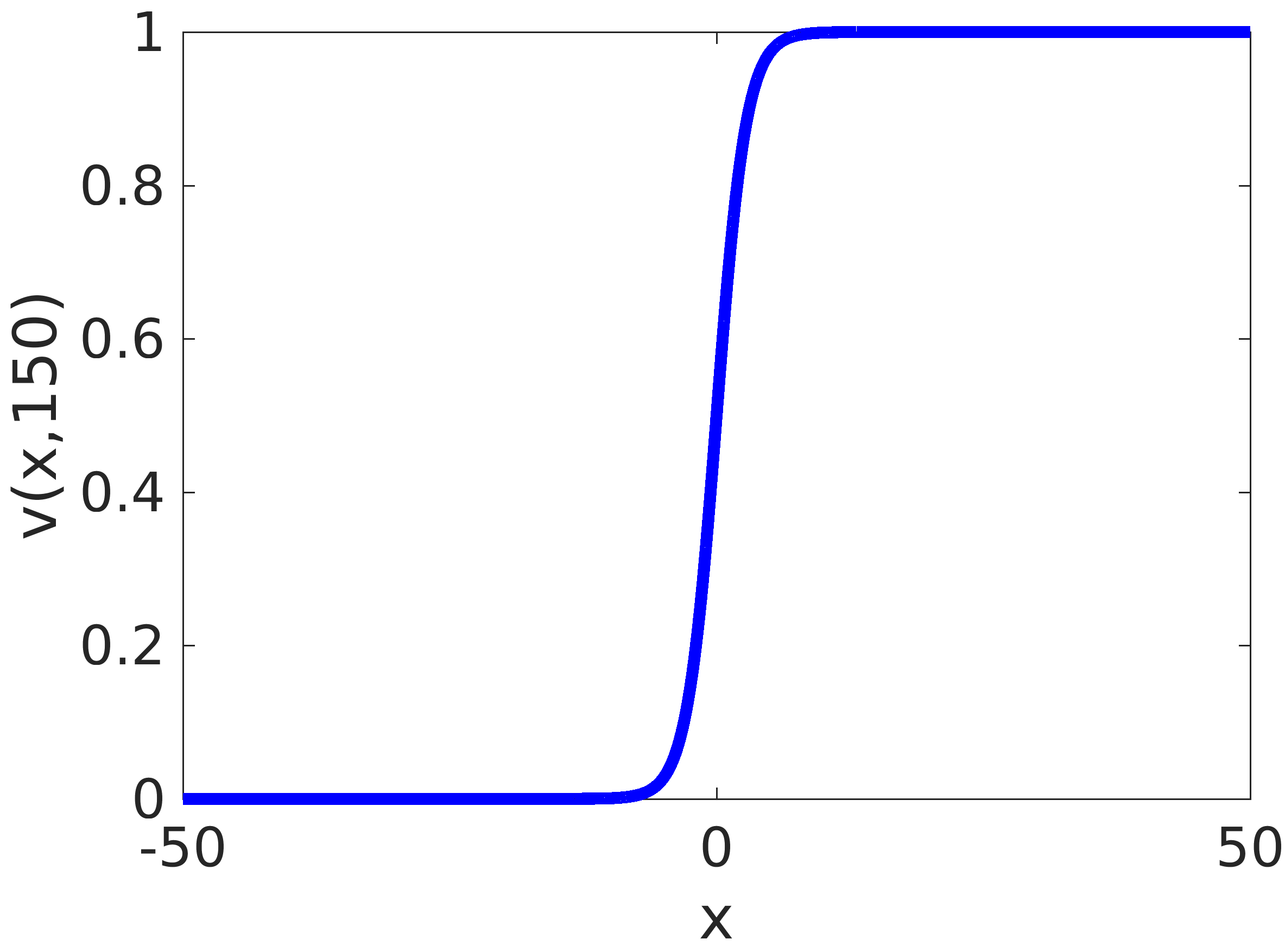}\label{fig:DampedNagumoFrozen1FrontProfile}}
  \subfigure[]{\includegraphics[height=3.8cm] {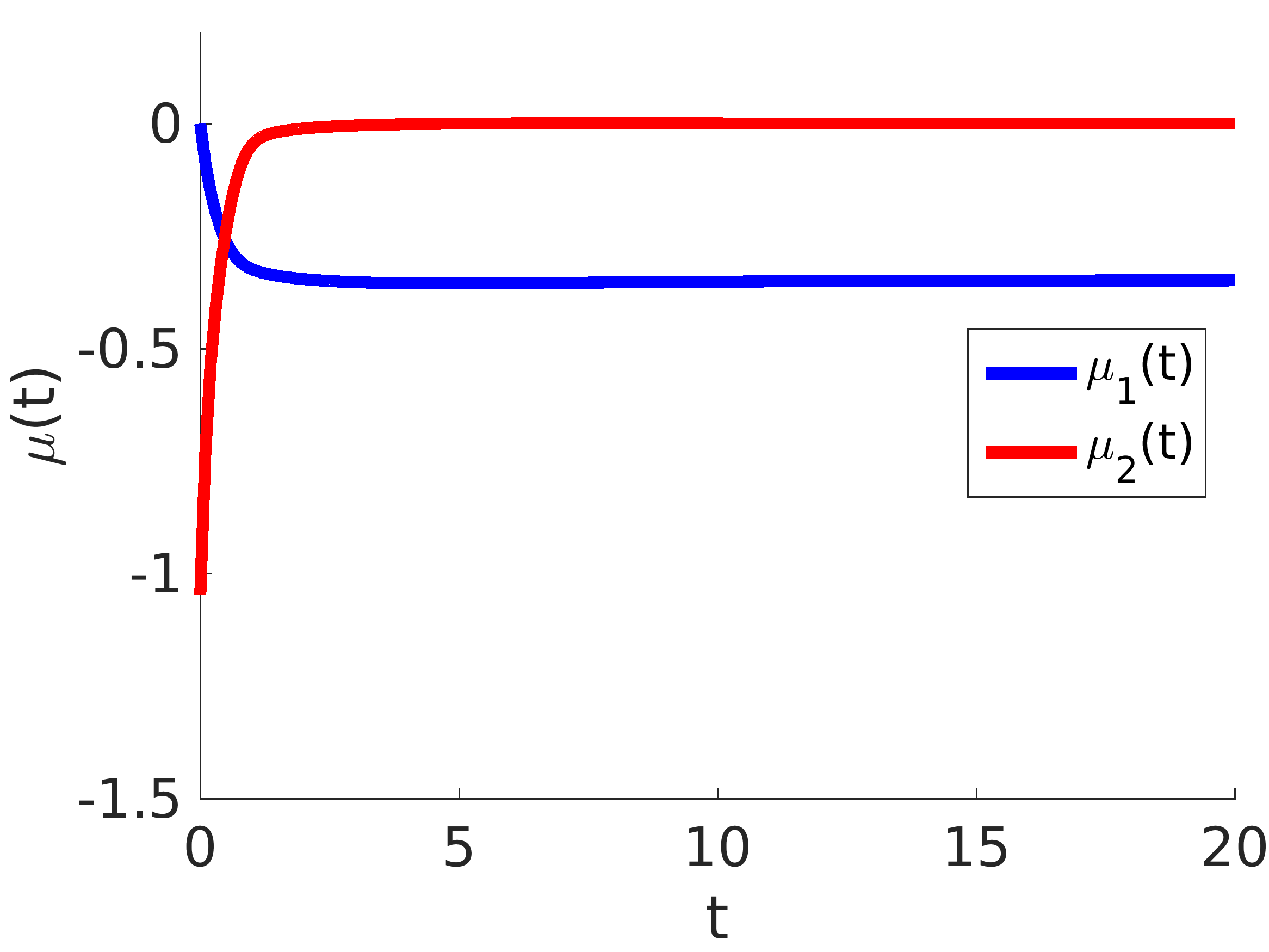} \label{fig:DampedNagumoFrozen1FrontVelocities}}
  \subfigure[]{\includegraphics[height=3.8cm] {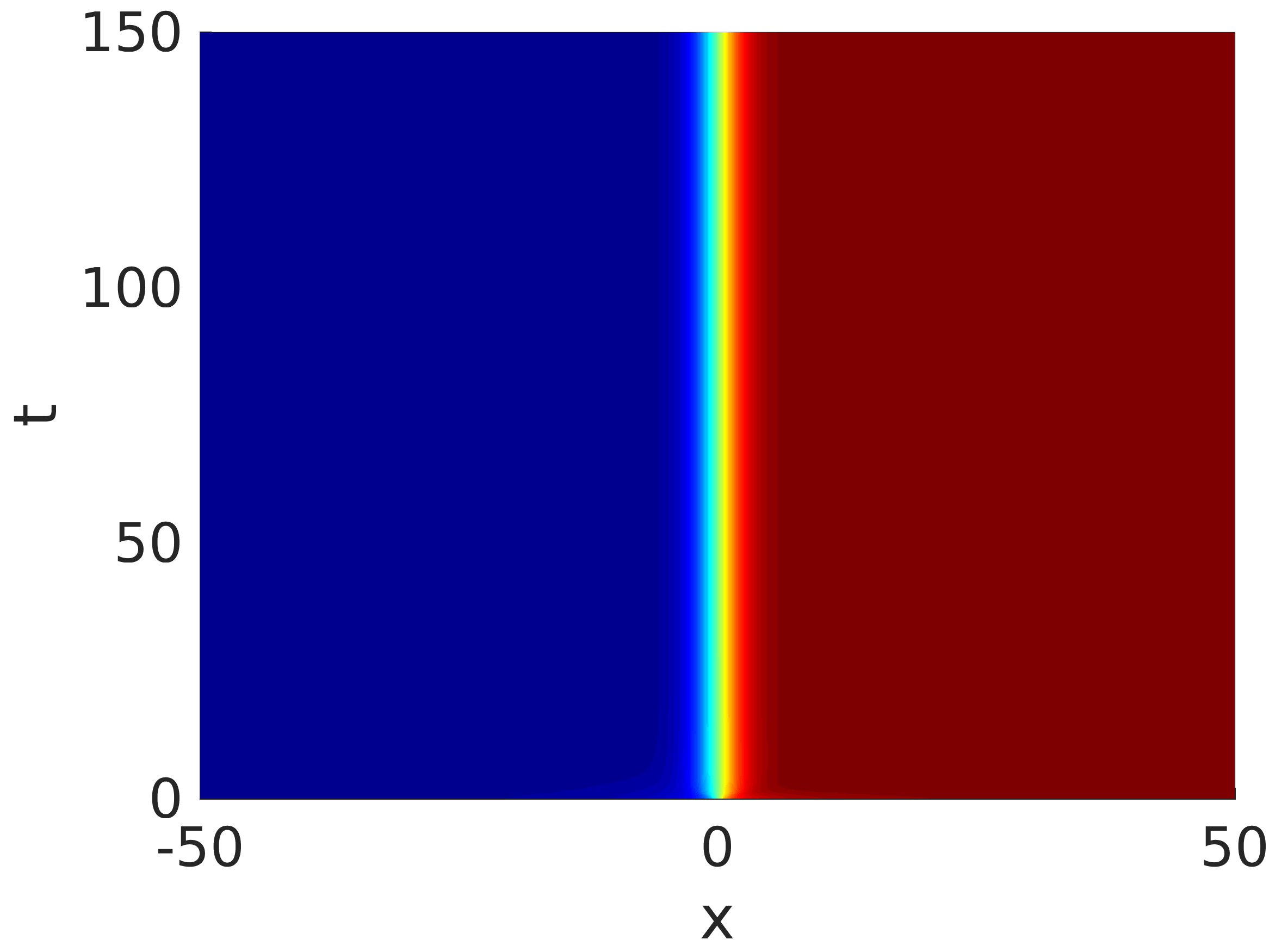}\label{fig:DampedNagumoFrozen1FrontSpaceTime}}
  \caption{Solution of the frozen Nagumo wave equation \eqref{equ:2.45}: Approximation of profile $v(x,150)$ (a) and time evolutions of velocity $\mu_1$ and acceleration $\mu_2$ (b) 
  and of the profile $v$ (c) for parameters $\varepsilon=b=\frac{1}{4}$.}
  \label{fig:FrozenDampedNagumoTraveling1Front}
\end{figure}

%\begin{figure}[ht]
%  \centering
%  \subfigure[]{\includegraphics[height=4cm] {Paper_DampedNagumo_1D_TravelingFront_Frozen_Profile.png}\label{fig:DampedNagumoFrozen1FrontProfile}}
%  \subfigure[]{\includegraphics[height=4cm] {Paper_DampedNagumo_1D_TravelingFront_Frozen_Velocity.png} \label{fig:DampedNagumoFrozen1FrontVelocities}} \\
%  \subfigure[]{\includegraphics[height=4cm] {Paper_DampedNagumo_1D_TravelingFront_Frozen_SpaceTime.png}\label{fig:DampedNagumoFrozen1FrontSpaceTime}}
%  \subfigure[]{\includegraphics[height=4cm] {Paper_DampedNagumo_1D_TravelingFront_Frozen_Position.png} \label{fig:DampedNagumoFrozen1FrontPosition}} \\
%  \caption{Solution of the frozen Nagumo wave equation \eqref{equ:2.45}: approximation of profile $v(x,150)$ (a), time evolutions of velocity $\mu_1$ and acceleration $\mu_2$ (b), 
%  of the profile $v$ (c), and of the position $\gamma$ (d) for parameters $\varepsilon=b=\frac{1}{4}$.}
%  \label{fig:FrozenDampedNagumoTraveling1Front}
%\end{figure}

Next we solve with the same data the frozen Nagumo wave equation
resulting from \eqref{equ:2.19}
% \begin{equation}
%   \begin{aligned}
%     \label{equ:2.45}
%     \begin{split}
%       &\varepsilon v_{tt} + v_t = (1-\mu_1^2 \varepsilon)v_{\xi\xi} + 2\mu_1 \varepsilon v_{\xi,t} + (\mu_2 \varepsilon + \mu_1)v_{\xi} + v(1-v)(v-b), &&\xi\in\R,\,t>0,\\
%       &v(\cdot,0) = u_0,\quad v_t(\cdot,0) = v_0+\mu_1^0 u_{0,\xi}, &&\xi\in\R,\,t=0,\\
%       & 0 = \left(v(\cdot,t)-\hat{v},\hat{v}_{\xi}\right)_{L^2(\R,\R)}, &&t\geqslant 0,\\
%       &\mu_{1,t} = \mu_2,\quad\mu_1(0) = \mu_1^0, &&t\geqslant 0, \\
%       &\gamma_t = \mu_1,\quad \gamma(0) = 0, &&t\geqslant 0.
%     \end{split}
%   \end{aligned}
% \end{equation}
\begin{subequations}    \label{equ:2.45}
  \begin{align}
     \label{equ:2.45a}
     &\begin{aligned}
      \varepsilon v_{tt} + v_t &= (1-\mu_1^2 \varepsilon)v_{\xi\xi} +
      2\mu_1 \varepsilon v_{\xi,t} +
      (\mu_2 \varepsilon + \mu_1)v_{\xi} + g(v),\\
      \mu_{1,t} &= \mu_2, \quad \gamma_t=\mu_1,
    \end{aligned}&t\geqslant 0,\\
%    & 0 = \bigl\langle v(\cdot,t)-\hat{v},\hat{v}_\xi\bigr\rangle_{L^2(\R,\R)}, &t\geqslant 0,\label{equ:2.45b}\\
    & 0 = \bigl\langle v_t(\cdot,t),\hat{v}_\xi\bigr\rangle_{L^2(\R,\R)}, &t\geqslant 0,\label{equ:2.45b}\\
    &\begin{aligned}
      v(\cdot,0) &= u_0,\quad v_t(\cdot,0) = v_0+\mu_1^0 u_{0,\xi}, \quad
      \mu_1(0) = \mu_1^0, \quad
      \gamma(0) = 0.
    \end{aligned}
  \end{align}
\end{subequations}
Figure \ref{fig:FrozenDampedNagumoTraveling1Front} shows the solution
$(v,\mu_1,\mu_2,\gamma)$ of \eqref{equ:2.45} on the spatial domain
$(-50,50)$ with homogeneous Neumann boundary conditions,
initial data $u_0$, $v_0$ from \eqref{equ:initdata}, and reference
function $\hat{v}=u_0$. For the computation we used the fixed phase condition 
$\psi_{\mathrm{fix},2}^{\mathrm{2nd}}(v_t)$ from \eqref{equ:2.16} with consistent 
intial data $\mu_1^0$, $\mu_2^0$, c.f. \eqref{equ:2.13consist} and \eqref{equ:2.14consist}. 
Note that $v_0=0$ from \eqref{equ:initdata} implies $\mu_1^0=0$ according to \eqref{equ:2.13consist}. 
Then, inserting $\mu_1^0=0$, $u_0,v_0$ from \eqref{equ:initdata}, $\hat{v}=u_0$, 
$M=\varepsilon$, $A=B=1$, $C=0$ and $g$ from \eqref{equ:2.36} into \eqref{equ:2.14consist}, 
finally implies $\mu_2^0=-1.0312$.
%% Jens: Ich habe dies nach unten verschoben, um es von der
%% Diskretisierung zu trennen.  Mir ist immer noch nicht klar, wofuer
%% das benoetigt wird.
% The value of
% $k$ is taken from \eqref{equ:2.39} and $c_{\star}=-0.35355$ from
% \eqref{equ:2.37} for $b=\frac{1}{4}$.  This yields the value
% $\mu_{\star}=-0.34816$ for the wave speed.
%\todo[inline]{This confuses me.  In my point of view, $k$ is no
%  parameter and neither are $c_{\star}$ and $\mu_{\star}$
The discretization data are taken as in the nonfrozen case.
The diagrams show that after a very short transition phase the profile
becomes stationary, the acceleration $\mu_2$ converges to zero, and
the speed $\mu_1$ approaches an asymptotic value
$\mu_{\star}^{\mathrm{num}}$ 
which is close to the exact value $\mu_{\star}\approx-0.34816$, given by 
\eqref{equ:2.39}. We expect 
$|\mu_{\star}-\mu_{\star}^{\mathrm{num}}|\to 0$ as the domain $(-R,R)$ 
grows and stepsizes tend to zero.
%$\mu_{\star}(50)$, 
%which depends in general on the size of the domain $(-50,50)$. More precisely, 
%as $R\to\infty$, the value $\mu_{\star}(R)$ obtained by solving \eqref{equ:2.19-20} 
%on $(-R,R)$ converges to $\mu_{\star}=\mu_{\star}(\infty)=-0.34816$, 
%which can be calculated from \eqref{equ:2.39} for $\varepsilon=b=\frac{1}{4}$.

%\todo[inline]{Das ``clearly''
%  sehe ich nicht. Konvergiert die numerische Loesung f\"ur $\mu_1$
%  tats\"achlich gegen $\mu_\star$ oder ist der numerische Grenzwert
%  nicht doch eine leichte Stoerung davon? F\"ur ein ``clearly'' mit
%  dieser Genauigkeit von $\mu_\star$ finde ich die Grafik auch nicht
%  geeignet, besser geeignet w\"are vielleicht der Vergleich
%  $|\mu_1-\mu_\star|$ oder $\log|\mu_1-\mu_\star|$.}  
%Figure \ref{fig:FrozenDampedNagumoTraveling1Front}(d) shows the
%function $\gamma(t), t\in [0,150]$, obtained by integrating the last
%equation in \eqref{equ:2.45a}. From its values one can still recover
%the position of the front in the original system \eqref{equ:2.38}. 
%In particular, one can read off from Figure
%\ref{fig:FrozenDampedNagumoTraveling1Front}(d) that the wave hits the
%left boundary at $x=-50$ at time $t \approx 143.82$ (cf. Figure
%\ref{fig:DampedNagumoTraveling1Front}(b)).
Note that the unknown function $\gamma(t)$ (not shown), $t\in [0,150]$, is obtained by integrating 
the last equation in \eqref{equ:2.45a}. From its values one can still recover
the position of the front in the original system \eqref{equ:2.38}. It turns out 
that the wave hits the left boundary at $x=-50$ at time $t \approx 143.82$ 
(cf. Figure \ref{fig:DampedNagumoTraveling1Front}(b)).
%\todo[inline]{Comparison to the exact solution? Why not the initial
 % data for $\mu_1^0$ as above? (Too complicated?)}

\begin{figure}[ht]
  \centering
  \subfigure[]{\includegraphics[height=3.8cm] {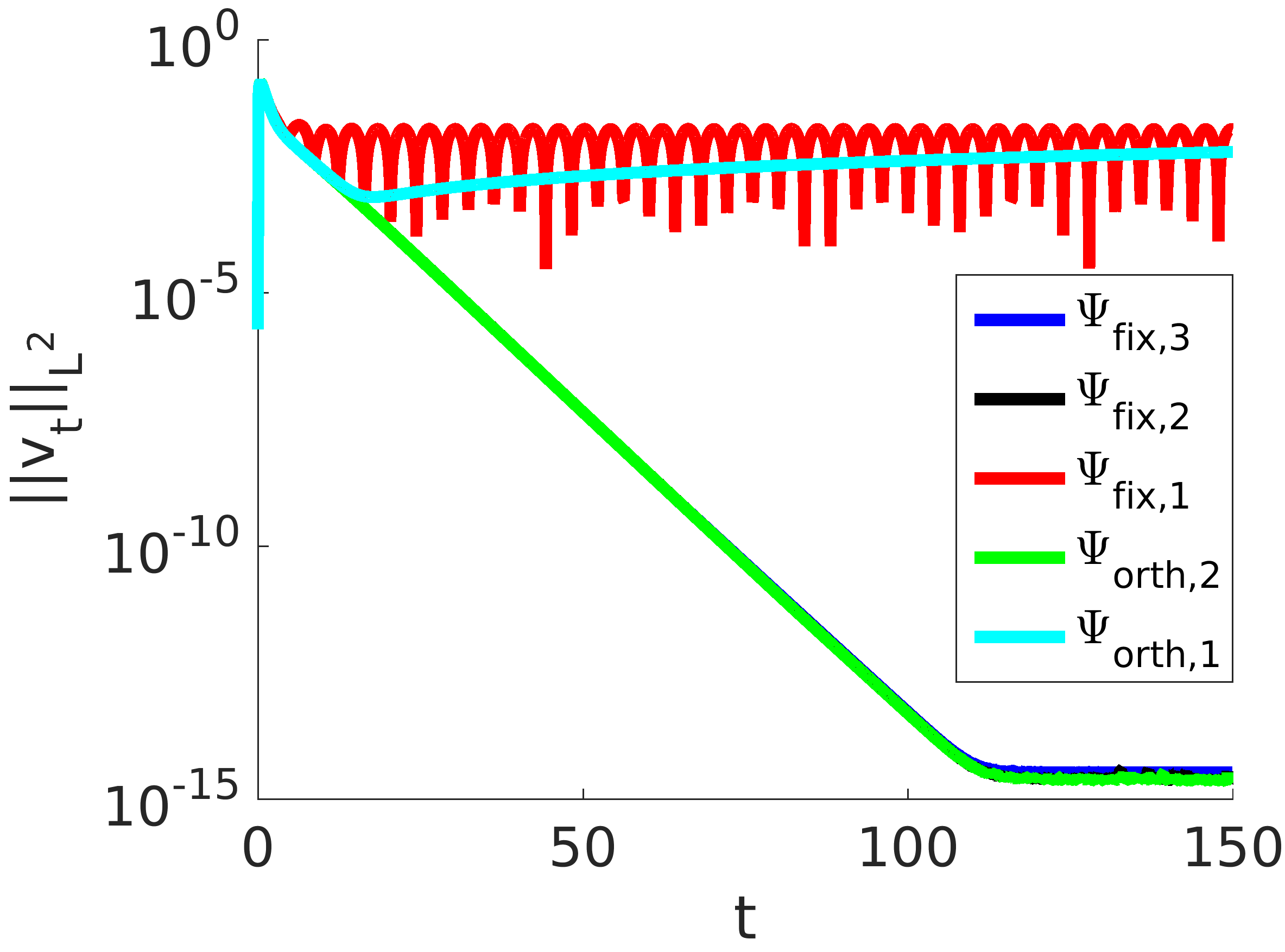}\label{fig:DampedNagumoFrozen1FrontVt}}
  \subfigure[]{\includegraphics[height=3.8cm] {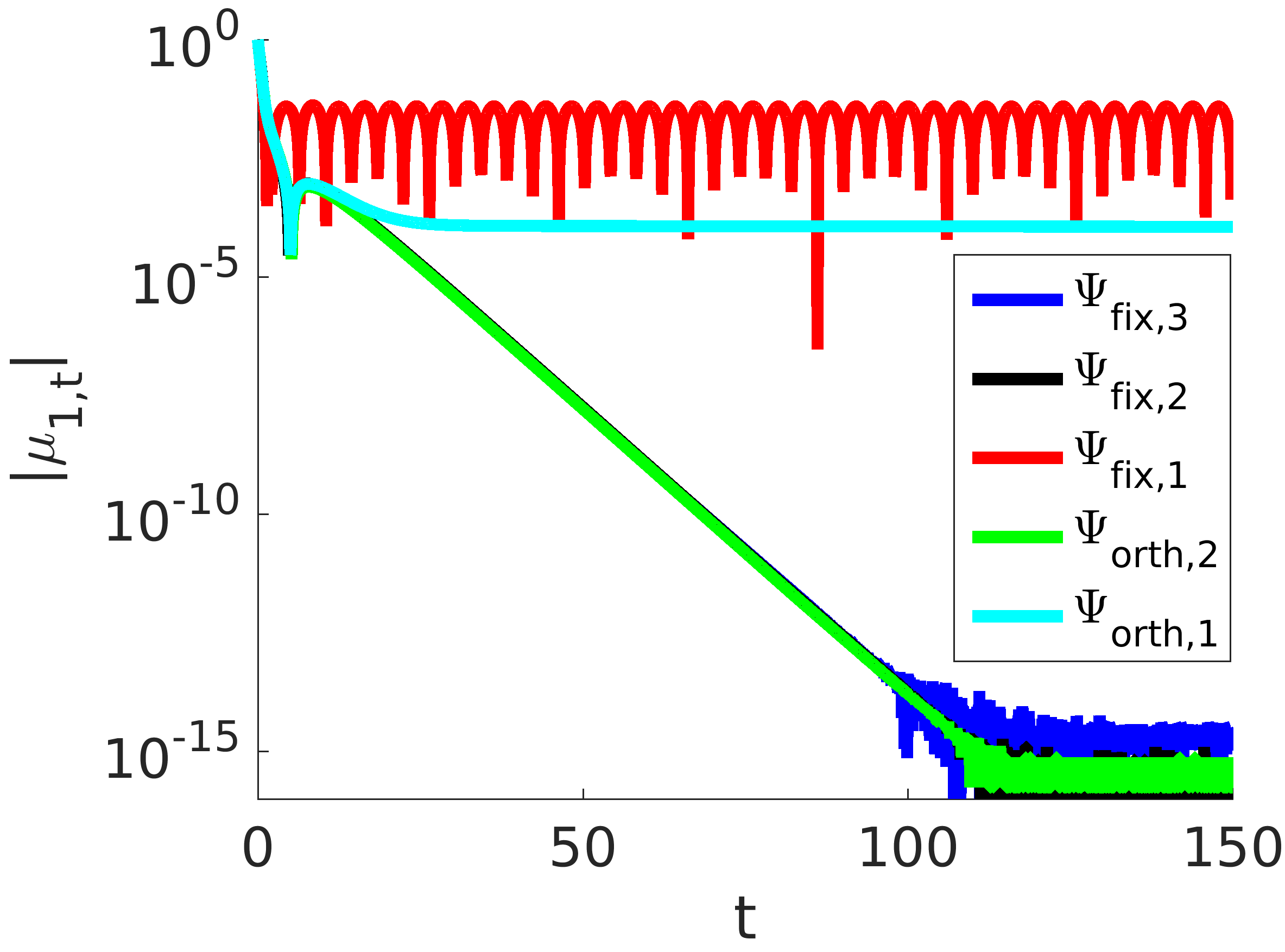} \label{fig:DampedNagumoFrozen1FrontMu1t}}
  \caption{Comparison of the phase conditions for the frozen Nagumo wave equation \eqref{equ:2.45}: Time evolution of $\left\|v_t\right\|_{L^2}$ (a) and $|\mu_{1,t}|$ (b) for parameters 
$\varepsilon=b=\frac{1}{4}$.}
  \label{fig:DampedNagumoFrozen1FrontPhaseConditions}
\end{figure} 

If we replace the phase condition $\psi_{\mathrm{fix},2}^{\mathrm{2nd}}$ in \eqref{equ:2.45} by $\psi_{\mathrm{fix},3}^{\mathrm{2nd}}$ 
or $\psi_{\mathrm{orth},2}^{\mathrm{2nd}}$, we obtain very similar results as those from Figure~\ref{fig:FrozenDampedNagumoTraveling1Front}. 
The profile again becomes stationary, the acceleration $\mu_2$ converges to zero, and the speed $\mu_1$ approaches to an 
asymptotic value. Since we expect $v_t(t)\to 0$ and $\mu_{1,t}(t)\to 0$ as $t\to\infty$, we use these quantities as an indicator checking whether 
the solution has become stationary. 
Figure~\ref{fig:DampedNagumoFrozen1FrontPhaseConditions} shows the time evolution of 
$\left\|v_t\right\|_{L^2}$ and $|\mu_{1,t}|$ when solving \eqref{equ:2.45} for different phase conditions. While the phase conditions of index $2$ and $3$
behave as expected, the index $1$ formulation yields small but oscillating
values for the norms of $v_t$ and $\mu_{1,t}$. We attribute this behavior to the fact, that our adaptive solver enforces the differentiated conditions
\eqref{equ:2.17}, \eqref{equ:2.16_orth}, but does not control $v_t,\mu_t$ directly. Further investigations 
show that the consistency condition for $\mu_2^0$ does not really affect the numerical results for the different phase conditions. 
Therefore, in the next example we do not compute the expression for
$\mu_2^0$ but use the expected limiting 
value as initial datum $\mu_2^0=0$. 
%By contrast the consistency condition for $\mu_1^0$ is crucial since otherwise we may loose positivity of the leading 
%term $M^{-1}A-\mu_1^2I_m$.% at the beginning of our computation. 
\end{example}

\begin{example}[FitzHugh-Nagumo wave system] \label{exa2}
Consider the $2$-dimensional parabolic FitzHugh-Nagumo system, \cite{FitzHugh1961},
\begin{align} 
  \label{equ:2.40}
  u_t = \tilde{A}u_{xx}+g(u), \;x\in\R,\,t\ge 0,
\quad \tilde{A}=\begin{pmatrix}1&0\\0&\rho\end{pmatrix}, \;
g(u)=\begin{pmatrix}u_1-\frac{1}{3}u_1^3-u_2\\\phi(u_1+a-bu_2)\end{pmatrix},
\end{align}
with $u=u(x,t)\in\R^2$ and positive parameters
$\rho,a,b,\phi\in\R$. Equation \eqref{equ:2.40} is known to exhibit traveling
wave solutions in a wide range of parameters, but there are apparently no
explicit formulas. For the values
\begin{equation} \label{equ:2.40a}
\rho=0.1,\quad  a=0.7, \quad \phi=0.08,\quad  b=0.8
\end{equation}
one finds a traveling pulse with
\begin{align} 
  \label{equ:2.41}
  %w_{\pm} = \begin{pmatrix}-1.19941\\-0.62426 \end{pmatrix}, \quad
  w_{\pm} \approx (-1.19941,\,-0.62426)^{\top}, \quad
c_{\star} \approx -0.7892.
\end{align}
% \todo[inline]{I found that with the parameter values quite
%   confusing... $w_{\pm}$ and $c_{\star}$ are no parameter values, but
%   $b$ (which was part of this line) is...}
For the same   $\rho, a, \phi$ but $b=3$, there is a traveling front 
with asymptotic states and velocity given by
\begin{align*} 
%  \label{equ:2.42}
  %w_- = \begin{pmatrix}1.18779\\0.62923\end{pmatrix},\quad w_+ = \begin{pmatrix}-1.56443\\-0.28814\end{pmatrix}, \quad
  w_- \approx (1.18779,\,0.62923)^{\top},\quad w_+ \approx (-1.56443,\,-0.28814)^{\top}, \quad
c_{\star} \approx -0.8557.
\end{align*}
Applying Proposition \ref{pro1}(i) with $M=\varepsilon I_2$  requires
 the equality $\tilde{A} + c_{\star}^2 M = k^2 A$, i.e.
\begin{align*}
  1+c_{\star}^2\varepsilon=k^2A_{11}, \quad \rho+c_{\star}^2\varepsilon=k^2A_{22}, 
\quad A_{12}=A_{21}=0.
\end{align*}
 Setting
$A_{11}:=1$ und using parameter values from \eqref{equ:2.40a},  Proposition~\ref{pro1}(i) shows that the corresponding FitzHugh-Nagumo wave system
\begin{align} 
  \label{equ:2.43}
  Mu_{tt}+Bu_t = Au_{xx} + g(u),\;x\in\R,\,t\ge 0,
\end{align}
with
\begin{align*}
%  \label{equ:2.44}
  M=\varepsilon I_2,\quad B=I_2,\quad 
  %A=\begin{pmatrix}1&0\\0&\frac{\rho+c_{\star}^2\varepsilon}{1+c_{\star}^2\varepsilon}\end{pmatrix},\quad 
  A=\mathrm{diag}(1,\tfrac{\rho+c_{\star}^2\varepsilon}{1+c_{\star}^2\varepsilon}),\quad 
  k=\sqrt{1+c_{\star}^2\varepsilon},\quad 
  \varepsilon>0,\quad \rho,c_{\star}\text{ given}
\end{align*}
has a traveling pulse (or a traveling front) solution with a scaled
profile $v_{\star}$, limits $v_{\pm}=w_{\pm}$, and velocity $\mu_{\star}=\frac{c_{\star}}{k}$.

In the following we show the computations for the traveling pulse.
Results for the traveling front are very similar and are not displayed
here.  In the frozen and the nonfrozen case, we choose 
 $\varepsilon=10^{-2}$ and parameter value \eqref{equ:2.40a}. Space and time
are  discretized as in Example~\ref{exa1}.
Figure~\ref{fig:DampedFitzHughNagumoTraveling1Pulse} shows the time
evolution of the traveling pulse solution $u=(u_1,u_2)^T$ of
\eqref{equ:2.43} on the spatial domain $(-50,50)$ with homogeneous 
Neumann boundary conditions. The initial data are
\begin{align}
  \label{equ:2.44a}
  %u_0(x)=\begin{pmatrix}\frac{1}{\pi}\arctan(x)+\frac{1}{2}\\0\end{pmatrix}+v_{\pm},\quad v_0(x)=\begin{pmatrix}0\\0\end{pmatrix},\quad x\in\R,
  u_0(x)=(\tfrac{1}{\pi}\arctan(x)+\tfrac{1}{2},\,0)^{\top}+v_{\pm},\quad v_0(x)=(0,\,0)^{\top},\quad x\in\R,
  %u_{0,1}(x)=\begin{cases}v_{\pm,1}&,\,-200\leqslant
  %  x<-10\\\frac{x+10}{20}+v_{\pm,1}&,\,-10\leqslant x\leqslant
  %  10\\v_{\pm,1}+1&,\,10<x\leqslant -200\end{cases},\quad
  %u_{0,2}(x)=v_{\pm,2},\quad v_0(x)=\begin{pmatrix}0\\0\end{pmatrix},
\end{align}
where $v_{\pm}=w_{\pm}$ is the asymptotic state from \eqref{equ:2.41}.
% \todo[inline]{I found the initial-data-text confusing, it was way
%   after the images and directly before the frozen wave system, where
%   again different data were chosen...}
% again continuous piecewise linear finite elementes with spatial stepsize $\triangle x=0.1$. For the time discretization we used the BDF method of order $2$ with absolute 
% tolerance $\mathrm{atol}=10^{-3}$, relative tolerance $\mathrm{rtol}=10^{-2}$, temporal stepsize $\triangle t=0.1$ and final time $T=150$.

\begin{figure}[ht]
  \centering
  \subfigure[]{\includegraphics[height=3.8cm] {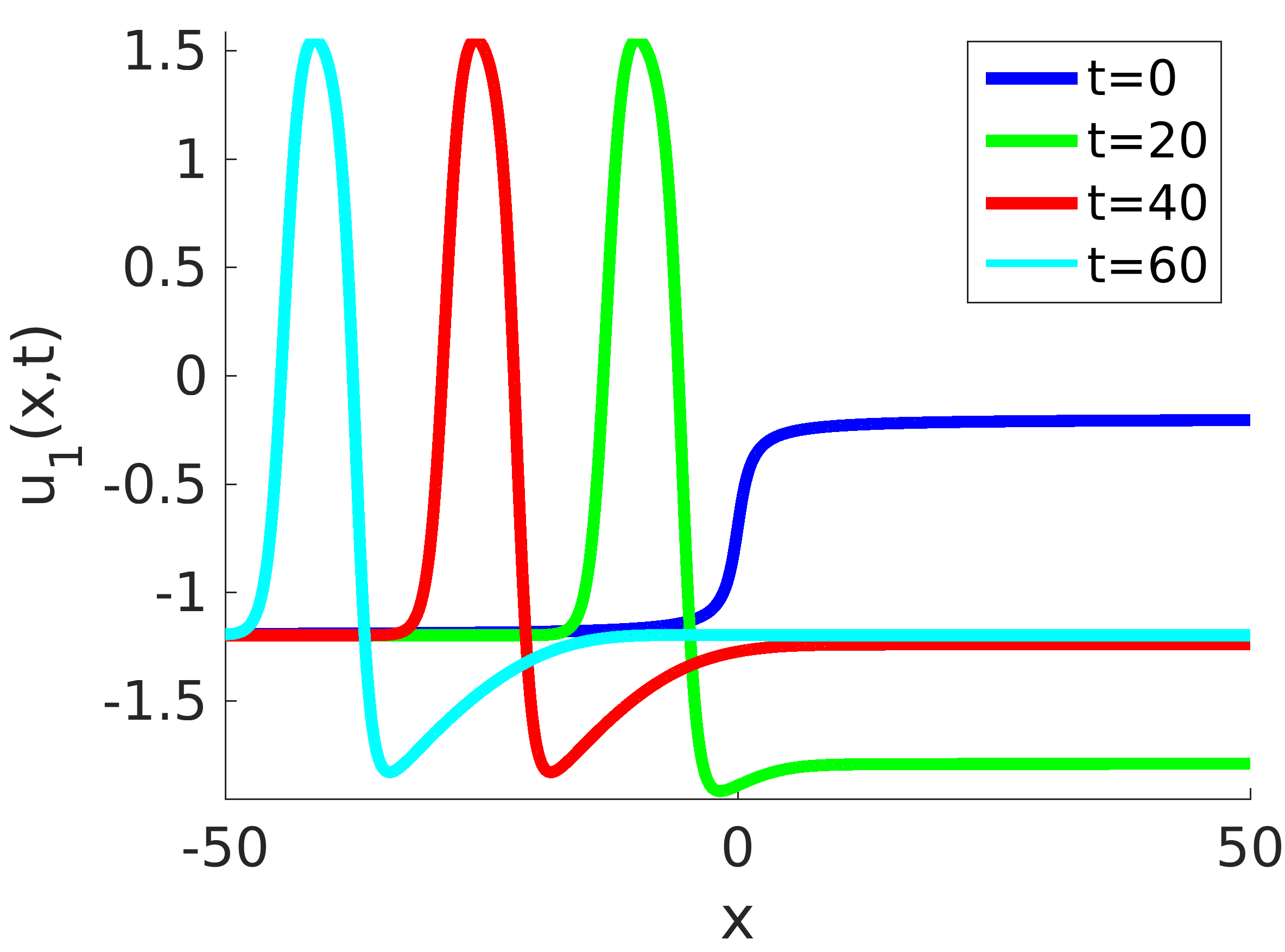}\label{fig:DampedFitzHughNagumoTraveling1PulseU1}}
  \subfigure[]{\includegraphics[height=3.8cm] {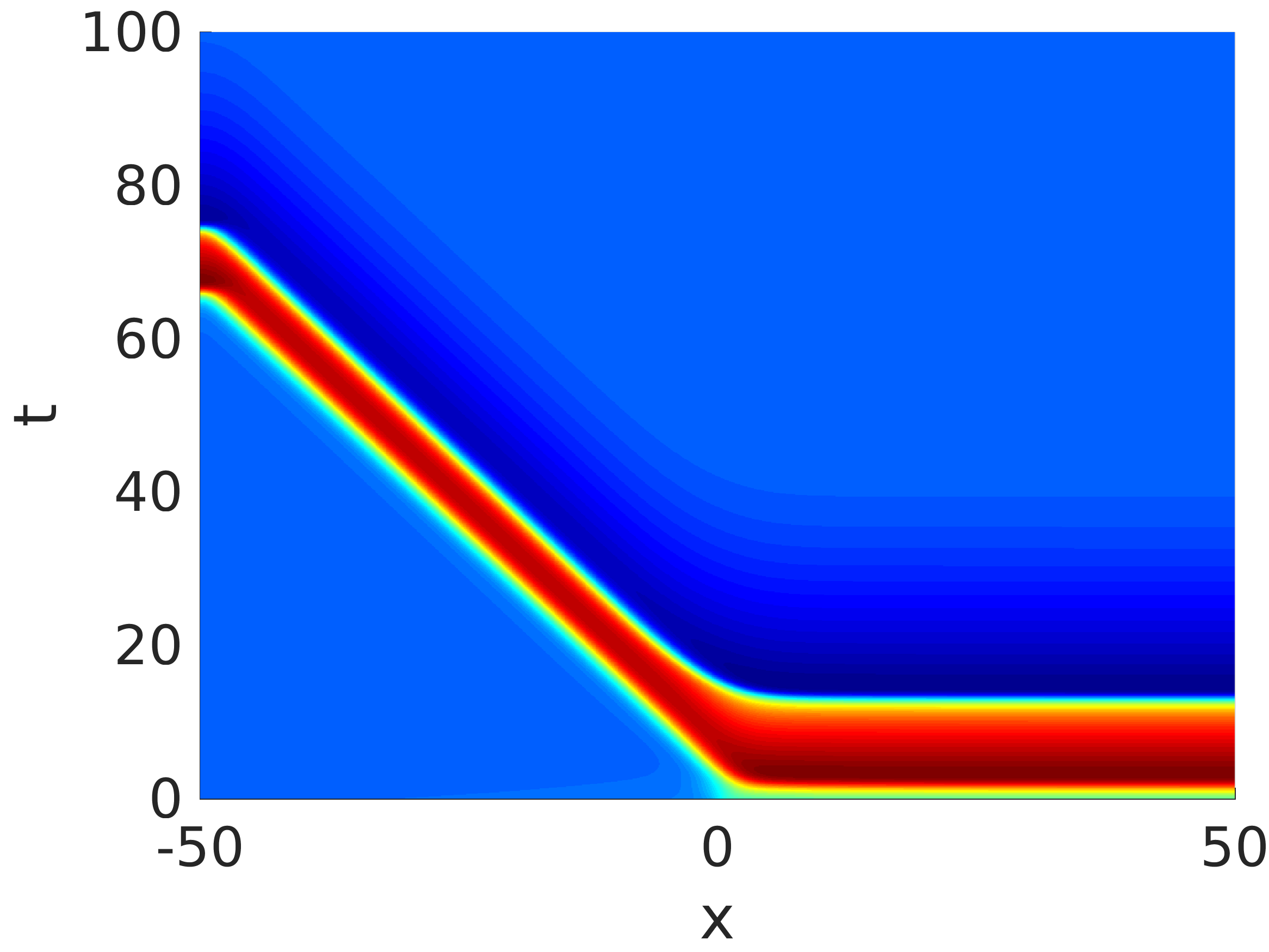} \label{fig:DampedFitzHughNagumoTraveling1PulseSpaceTimeU1}}
  \caption{Traveling pulse of FitzHugh-Nagumo wave system
    \eqref{equ:2.43} at different time instances for $u_1$ (a) as well as its time evolutions (c) for parameters $\varepsilon=10^{-2}$, $\rho=0.1$, $a=0.7$, $\phi=0.08$ and $b=0.8$.}
  \label{fig:DampedFitzHughNagumoTraveling1Pulse}
\end{figure} 

%\begin{figure}[ht]
%  \centering
%  \subfigure[]{\includegraphics[height=3.8cm] {Paper_DampedFitzHughNagumo_1D_Traveling1Pulse_Nonfrozen_Traveling1Pulse_u1.png}\label{fig:DampedFitzHughNagumoTraveling1PulseU1}}
%  \subfigure[]{\includegraphics[height=3.8cm] {Paper_DampedFitzHughNagumo_1D_Traveling1Pulse_Nonfrozen_Traveling1Pulse_u2.png}\label{fig:DampedFitzHughNagumoTraveling1PulseU2}}\\
%  \subfigure[]{\includegraphics[height=3.8cm] {Paper_DampedFitzHughNagumo_1D_Traveling1Pulse_Nonfrozen_SpaceTime_u1.png} \label{fig:DampedFitzHughNagumoTraveling1PulseSpaceTimeU1}}
%  \subfigure[]{\includegraphics[height=3.8cm] {Paper_DampedFitzHughNagumo_1D_Traveling1Pulse_Nonfrozen_SpaceTime_u2.png} \label{fig:DampedFitzHughNagumoTraveling1PulseSpaceTimeU2}}
%  \caption{Traveling pulse of FitzHugh-Nagumo wave system
%    \eqref{equ:2.43} at different time instances for $u_1$ (a) and
%    $u_2$ (b), as well as their time evolutions (c), (d) 
%  for parameters $\varepsilon=10^{-2}$, $\rho=0.1$, $a=0.7$, $\phi=0.08$ and $b=0.8$.}
%  \label{fig:DampedFitzHughNagumoTraveling1Pulse}
%\end{figure} 

\begin{figure}[ht]
  \centering
  \subfigure[]{\includegraphics[height=3.8cm] {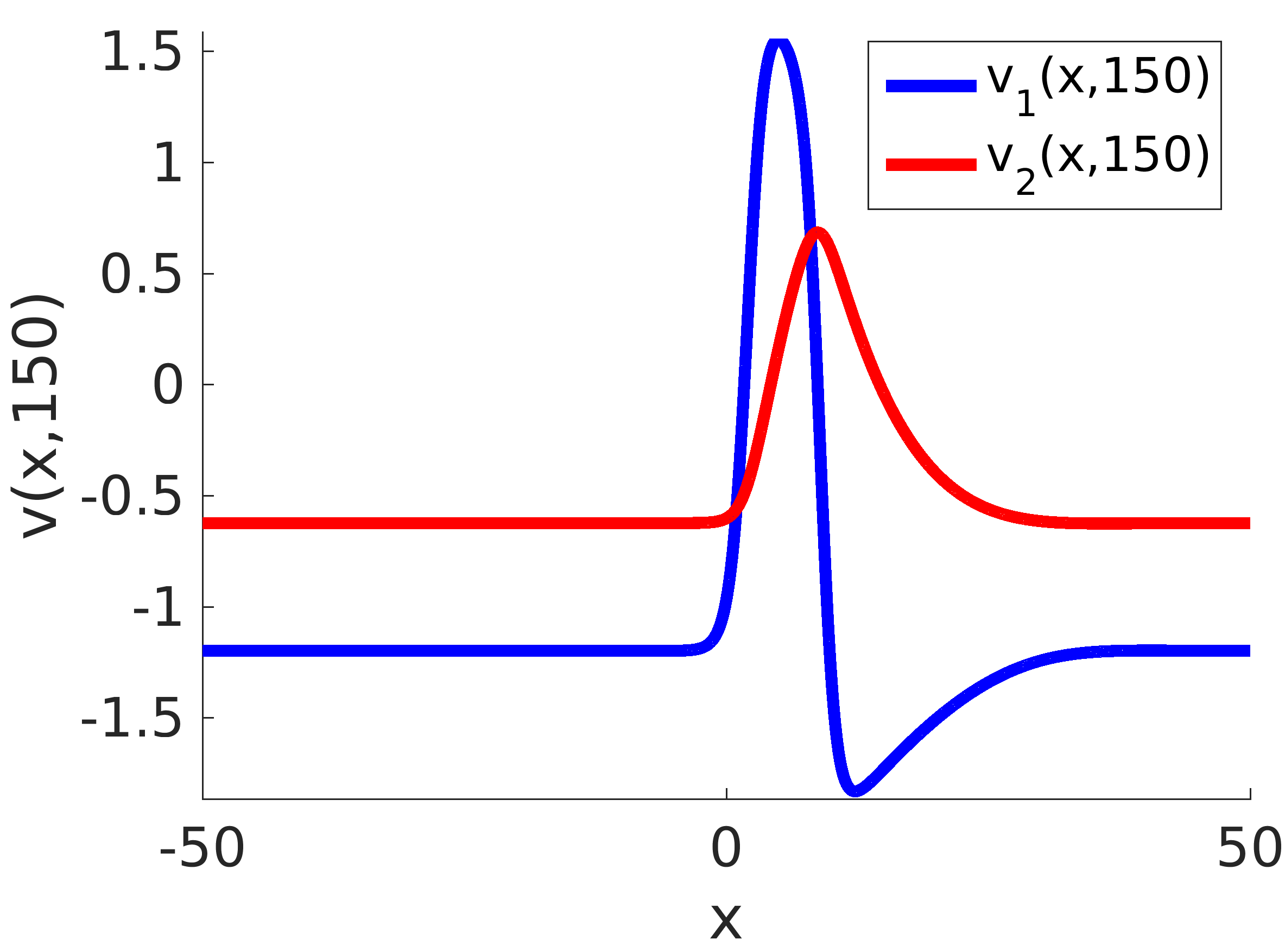}\label{fig:DampedFitzHughNagumoFrozen1PulseProfile}}
  \subfigure[]{\includegraphics[height=3.8cm] {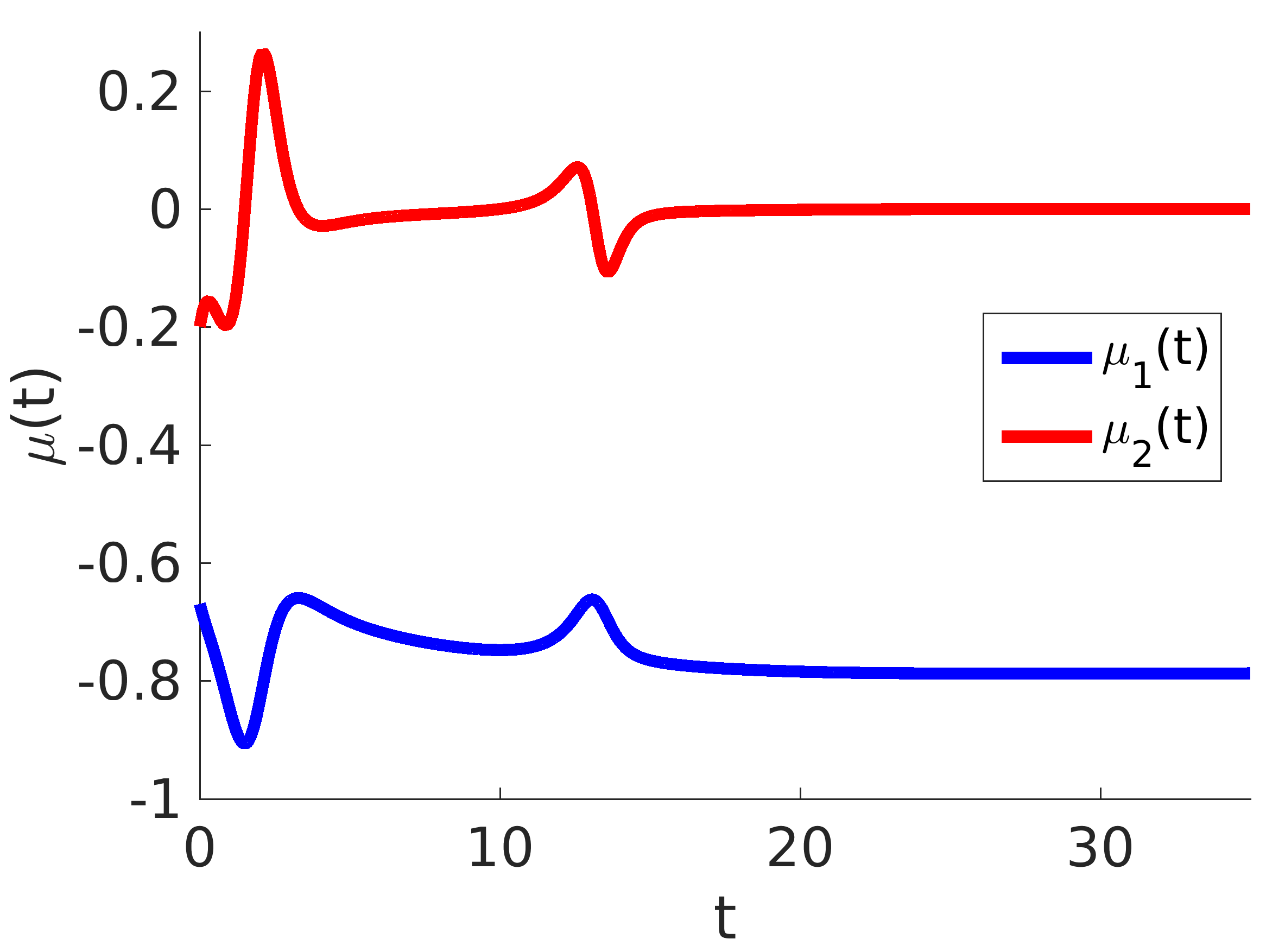} \label{fig:DampedFitzHughNagumoFrozen1PulseVel}}
  \subfigure[]{\includegraphics[height=3.8cm] {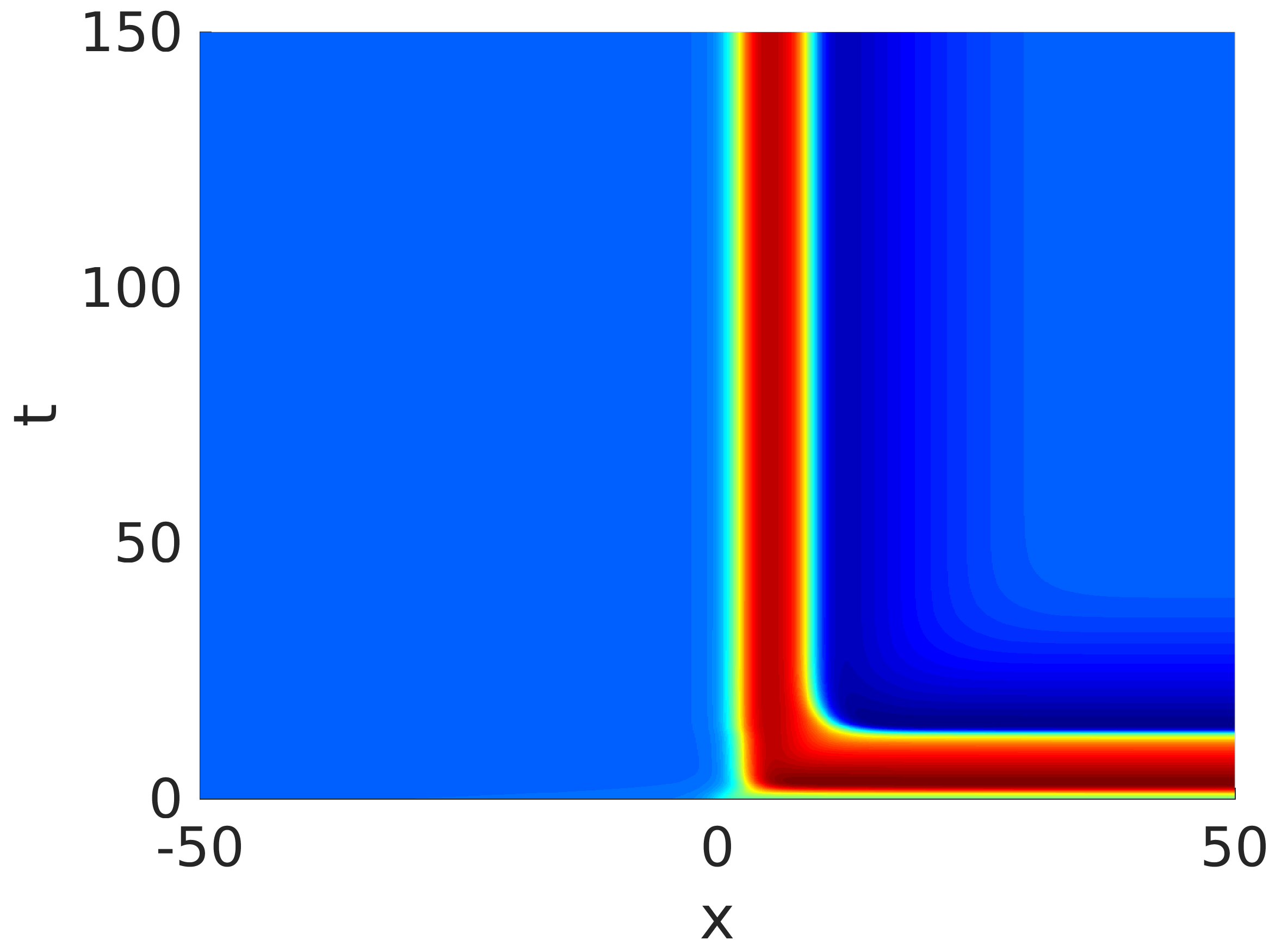}\label{fig:DampedFitzHughNagumoFrozen1PulseU1}}
  \caption{Solution of the frozen FitzHugh-Nagumo wave system \eqref{equ:2.19}: Approximation of profile components $v_1(x,150)$, $v_2(x,150)$ (a), and time evolutions of 
  velocity $\mu_1$ and acceleration $\mu_2$ (b) and of the profile's component $v_1$ (c) for parameters $\varepsilon=10^{-2}$, $\rho=0.1$, $a=0.7$, 
  $\phi=0.08$ and $b=0.8$.}
  \label{fig:FrozenDampedFitzHughNagumoTraveling1Pulse}
\end{figure}

%\begin{figure}[ht]
%  \centering
%  \subfigure[]{\includegraphics[height=3.8cm] {Paper_DampedFitzHughNagumo_1D_Traveling1Pulse_Frozen_Profile.png}\label{fig:DampedFitzHughNagumoFrozen1PulseProfile}}
%  \subfigure[]{\includegraphics[height=3.8cm] {Paper_DampedFitzHughNagumo_1D_Traveling1Pulse_Frozen_Velocity.png} \label{fig:DampedFitzHughNagumoFrozen1PulseVel}}
%  \subfigure[]{\includegraphics[height=3.8cm] {Paper_DampedFitzHughNagumo_1D_Traveling1Pulse_Frozen_Position.png}\label{fig:DampedFitzHughNagumoFrozen1PulsePos}}\\
%  \subfigure[]{\includegraphics[height=3.8cm] {Paper_DampedFitzHughNagumo_1D_Traveling1Pulse_Frozen_SpaceTime_v1.png}\label{fig:DampedFitzHughNagumoFrozen1PulseU1}}
%  \subfigure[]{\includegraphics[height=3.8cm] {Paper_DampedFitzHughNagumo_1D_Traveling1Pulse_Frozen_SpaceTime_v2.png}\label{fig:DampedFitzHughNagumoFrozen1PulseU2}}
%  \subfigure[]{\includegraphics[height=3.8cm] {Paper_DampedFitzHughNagumo_1D_Traveling1Pulse_Frozen_DiffCoeff.png}\label{fig:DampedFitzHughNagumoFrozen1PulseDiffCoeff}}
%  \caption{Solution of the frozen FitzHugh-Nagumo wave system \eqref{equ:2.19}: approximation of profile components $v_1(x,150)$, $v_2(x,150)$ (a), time evolutions of 
%  velocity $\mu_1$ and acceleration $\mu_2$ (b), of the position
%  $\gamma$ (c), of the profile's components $v_1$ (d) and $v_2$ (e), and of the diffusion matrix component $(M^{-1}A-\mu_1^2I_2)_{22}$ (f)  
%for parameters $\varepsilon=10^{-2}$, $\rho=0.1$, $a=0.7$, $\phi=0.08$ and $b=0.8$.}
%  \label{fig:FrozenDampedFitzHughNagumoTraveling1Pulse}
%\end{figure}

Next consider for the same parameter values the corresponding frozen
FitzHugh-Nagumo wave system
\begin{subequations}    \label{equ:2.46}
  \begin{align}
     \label{equ:2.46a}
     &\begin{aligned}
      M v_{tt} + B v_t &= (A-\mu_1^2 M)v_{\xi\xi} +
      2\mu_1 M v_{\xi,t} +
      (\mu_2 M+\mu_1 B)v_{\xi} + g(v),\\
      \mu_{1,t} &= \mu_2, \quad \gamma_t=\mu_1,
    \end{aligned}&t\geqslant 0,\\
%    & 0 = \bigl\langle v(\cdot,t)-\hat{v},\hat{v}_\xi\bigr\rangle_{L^2(\R,\R)}, &t\geqslant 0,\label{equ:2.46b}\\
    & 0 = \bigl\langle v_t(\cdot,t),\hat{v}_\xi\bigr\rangle_{L^2(\R,\R)}, &t\geqslant 0,\label{equ:2.46b}\\
    &\begin{aligned}
      v(\cdot,0) &= u_0,\quad v_t(\cdot,0) = v_0+\mu_1^0 u_{0,\xi}, \quad
      \mu_1(0) = \mu_1^0, \quad
      \gamma(0) = 0.
    \end{aligned}
  \end{align}
\end{subequations}
Figure~\ref{fig:FrozenDampedFitzHughNagumoTraveling1Pulse} shows the solution
$(v,\mu_1,\mu_2,\gamma)$ of \eqref{equ:2.46} on the spatial domain
$(-50,50)$, with homogeneous Neumann boundary conditions, initial data 
$u_0$, $v_0$ from \eqref{equ:2.44a}, and reference function $\hat{v}=u_0$. 
For the computation we used again the fixed phase condition 
$\psi_{\mathrm{fix},2}^{\mathrm{2nd}}(v_t)$ from \eqref{equ:2.16} with 
consistent intial data for $\mu_1^0$. Note that $v_0=0$ from \eqref{equ:2.44a} 
implies $\mu_1^0=0$ according to \eqref{equ:2.13consist}. We further set 
$\mu_2^0=0$ which does not satisfy the consistency condition 
\eqref{equ:2.14consist}. 
%$(-60,60)$, computed on $(-200,200)$ with homogeneous Neumann
%boundary conditions. The larger domain was chosen to increase the
%spectral resolution in Figure \ref{fig:DampedNagumo_1D_Traveling1Front_EssentialSpectrum}(b) below,
%it is not needed for the accuracy here. The initial data are chosen as
%$u_0=u(\cdot,60)$ and $v_0=u_t(\cdot,60)$ and $\mu_1^0=0$, where $u$
%denotes the solution of \eqref{equ:2.43} at time $60$ which was
%computed above.  
%\todo[inline]{better take the consistent value ?  Was
%  passiert, wenn man alternativ $v_0=0$ als Startwert nimmt, dann ist
%  $\mu_1^0=0$ korrekt.  Woher bekommt man eigentlich $u_t(\cdot,60)$?
%  Spuckt Comsol dies mit aus?}  
Time and space discretization are done as in
the nonfrozen case. Again the profile quickly stabilizes and the
velocity and the acceleration reach their asymptotic values. 
%Figure \ref{fig:FrozenDampedFitzHughNagumoTraveling1Pulse}(f)
%shows the time evolution for the smallest diagonal entry of the diffusion 
%matrix. Since the function $(M^{-1}A-\mu_1(t)^2)_{22}$ does not fall below 
%zero, the diffusion matrix keep positivity during the whole computation.

\begin{figure}[ht]
  \centering
  \subfigure[]{\includegraphics[height=3.8cm] {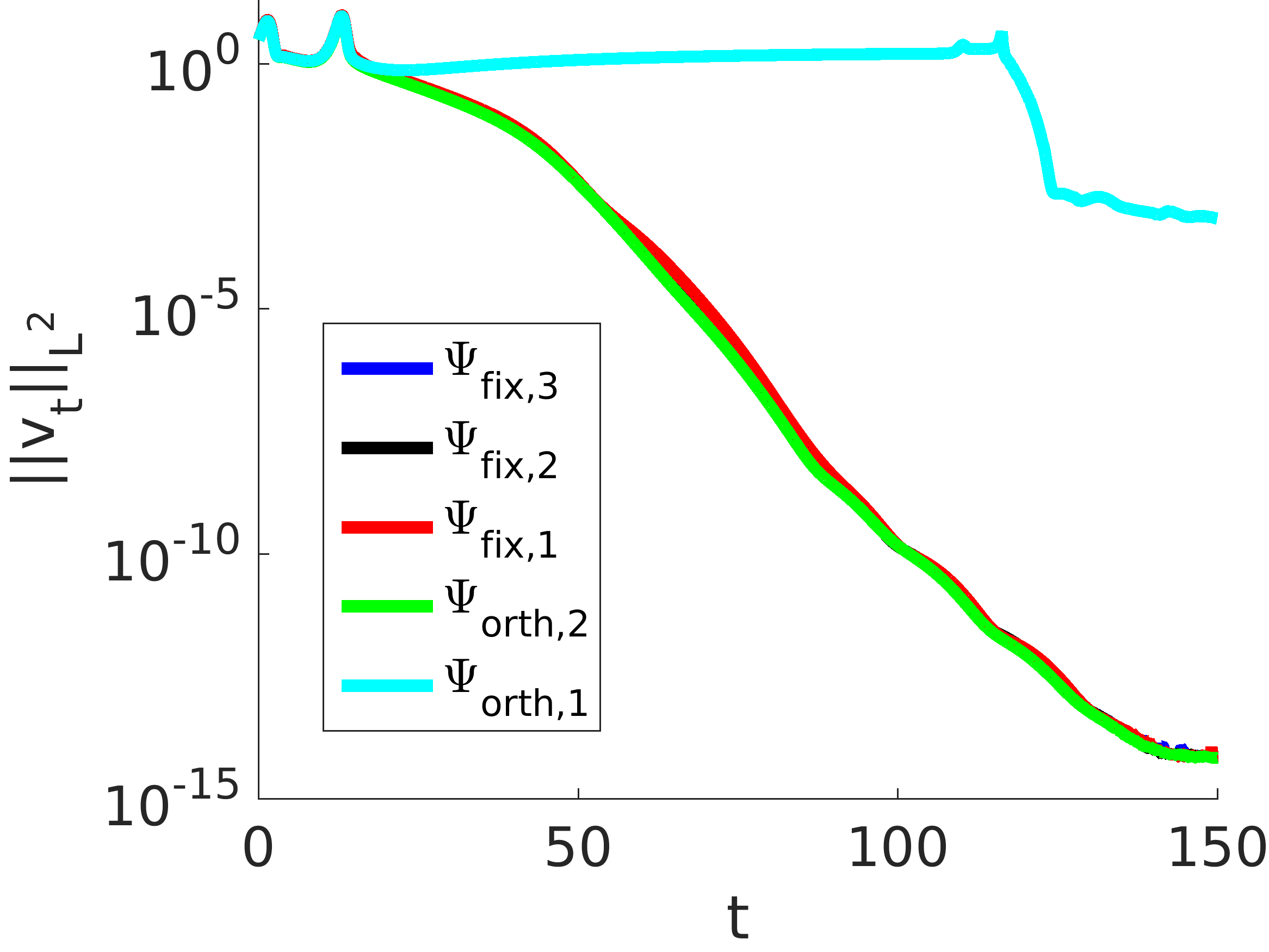}\label{fig:DampedFitzHughNagumoFrozen1PulseVt}}
  \subfigure[]{\includegraphics[height=3.8cm] {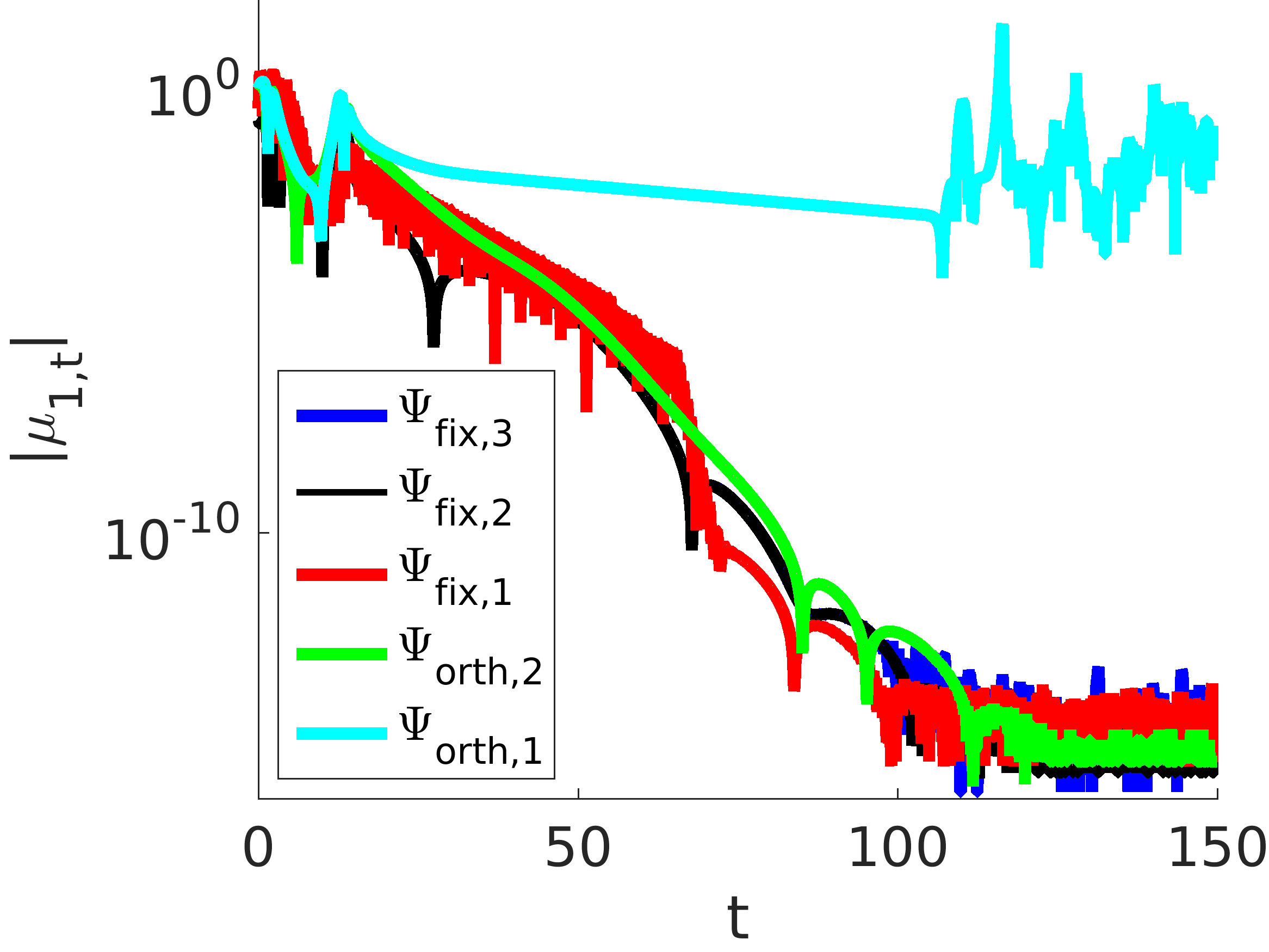} \label{fig:DampedFitzHughNagumoFrozen1PulseMu1t}}
  \caption{Comparison of the phase conditions for the frozen FitzHugh-Nagumo wave system \eqref{equ:2.46}: Time evolution of $\left\|v_t\right\|_{L^2}$ (a) and $|\mu_{1,t}|$ (b) for 
parameters $\varepsilon=10^{-2}$, $\rho=0.1$, $a=0.7$, $\phi=0.08$ and $b=0.8$.}
  \label{fig:DampedFitzHughNagumoFrozen1PulsePhaseConditions}
\end{figure} 
Finally, Figure~\ref{fig:DampedFitzHughNagumoFrozen1PulsePhaseConditions} shows that 
similar results are obtained if we replace the phase condition $\psi_{\mathrm{fix},2}^{\mathrm{2nd}}$ 
in the frozen FitzHugh-Nagumo wave system \eqref{equ:2.46} by $\psi_{\mathrm{fix},3}^{\mathrm{2nd}}$, 
$\psi_{\mathrm{fix},1}^{\mathrm{2nd}}$, $\psi_{\mathrm{orth},2}^{\mathrm{2nd}}$, or even by $\psi_{\mathrm{orth},1}^{\mathrm{2nd}}$. 
Contrary to our first example, the fixed phase condition of index $1$ provides good results in this case, while the index $1$ formulation of the orthogonal phase condition $\psi_{\mathrm{orth},1}^{\mathrm{2nd}}$ 
continues to show small oscillations of the time derivatives. 
\end{example}

% ---------------------------------------------------------------------------------------------------------------------------------------------------
%
%  SECTION 3: (Spectra and eigenfunctions of traveling waves)
%
%---------------------------------------------------------------------------------------------------------------------------------------------------
\sect{Spectra and eigenfunctions of traveling waves}
\label{sec:4}
%--------------------------------------------------------------------------------------------------------------------------------------------------- 
In this section we study the spectrum of the quadratic operator
polynomial (cf. \eqref{equ:2.6})
\begin{equation}
  \begin{aligned}
 \label{equ:4.1}
    \PL(\lambda) := \lambda^2 P_2 + \lambda P_1 + P_0, \quad \lambda
    \in \C.
  \end{aligned}
\end{equation}
Here the differential operators $P_j$ are defined by
\begin{equation} \label{equ:defPj}
  \begin{aligned}
 % \label{equ:4.2}
      P_2 = M, \; \;
      P_1 = -D_3f(\star) - 2\mu_{\star}M\partial_{\xi}, \; \;
      P_0 = -(A-\mu_{\star}^2 M)\partial_{\xi}^2 
      + (\mu_{\star}D_3f(\star)-D_2f(\star))\partial_{\xi} - D_1f(\star),
  \end{aligned}
\end{equation}
where $(\star)=(v_{\star},v_{\star,\xi}, - \mu_{\star}v_{\star,\xi})$ and $v_{\star}$,
$\mu_{\star}$ denote the profile and velocity
of a traveling wave solution
$u_{\star}(x,t)=v_{\star}(x-\mu_{\star}t)$ of \eqref{equ:2.1}. Note
that $P_j$ is a differential operator of order $2-j$ for $j=0,1,2$. In
the following we recall some standard notions of point and essential
spectrum for operator polynomials.
\begin{definition}
  \label{def:3.1}
  Let $(X,\left\|\cdot\right\|_X)$ and $(Y,\left\|\cdot\right\|_Y)$ be
  complex Banach spaces and let 
  $\PL(\lambda)=\sum_{j=0}^q P_j \lambda^j, \lambda \in \C$ be an
  operator polynomial with linear 
  continuous coefficients $P_j:Y\rightarrow X$, $j=0,\ldots,q$.
  \begin{itemize}[leftmargin=0.65cm]\setlength{\itemsep}{0cm}
  \item[(a)]The resolvent set $\rho(\PL)$ and  the spectrum $\sigma(\PL)$ are defined by
    \begin{align*}
\rho(\PL)=\{\lambda \in \C: \PL(\lambda)\; \text{is bijective and} \quad
    \PL(\lambda)^{-1}:X\rightarrow Y \;\text{is bounded}\}, \;
    \sigma(\PL):=\C\backslash\rho(\PL).
\end{align*}
    \item[(b)] $\lambda_0\in\sigma(\PL)$ is called isolated
    if there is $\varepsilon >0$ such that $\lambda \in \rho(\PL)$ for
    all $\lambda_0\ne \lambda\in\C$ with
    $|\lambda-\lambda_0|<\varepsilon$.
  \item[(c)] If $\PL(\lambda_0) y_0=0$ for some $\lambda_0 \in C$ and
    $y_0\in Y\setminus\{0\}$, then $\lambda_0$ is called an eigenvalue with eigenvector $y_0$.
      The eigenvalue
    $\lambda_0$ has finite multiplicity if $\dim(\kernel(\PL(\lambda_0))) < \infty$  and if there is a maximum number $n\in \N$, for
    which polynomials $y(\lambda)=\sum_{j=0}^r (\lambda-\lambda_0)^j
    y_j$ exist in $Y$ satisfying
    \begin{align} \label{equ:rootprop}
      y_0 \neq 0, \quad (\PL y)^{(\nu)}(\lambda_0)=0, \quad
      \nu=0,\ldots,n-1.
    \end{align}
This maximum number $n=n(\lambda_0)$ is  called the maximum partial multiplicity,
and $\dim(\kernel(\PL(\lambda_0)))$ is called the geometric multiplicity 
of $\lambda_0$.
     \item[(d)] The point spectrum is defined by
   \begin{align*}
   \sigma_{\mathrm{point}}(\PL)=\{\lambda\in\sigma(\PL):
    \;\text{$\lambda$ is isolated eigenvalue of finite
      multiplicity}\}.
  \end{align*}
Points in $\rho(\PL) \cup \sigma_{\mathrm{point}}(\PL)$ are called normal,
   and the essential spectrum of $\PL$ is defined by 
  \begin{align*}
    \sigma_{\mathrm{ess}}(\PL):=\{\lambda\in\C:\;\text{$\lambda$ is
      not a normal point of $\PL$}\}.
  \end{align*}
  \end{itemize}
 \end{definition}
\begin{remark}\label{rem:3.2}
There is no loss of generality in assuming the root polynomials in (c) to
be of the form $y(\lambda)=\sum_{j=0}^{n-1} (\lambda- \lambda_0)^j y_j$.
For if $r<n-1$ we simply set $y_j=0,j=r+1,\ldots,n-1$. And if $r\ge n$
we subtract from $y$  the term $\sum_{j=n}^{r} (\lambda- \lambda_0)^j y_j$ 
which has $\lambda_0$ as a zero  of order at least $n$ and thus does not
change the root property \eqref{equ:rootprop}. The eigenvalue $\lambda_0$ is simple iff the geometric and the 
maximum partial multiplicity are equal to $1$.
 In this case $\kernel(\PL(\lambda_0))=\mathrm{span}(y_0)$ for some $y_0\neq 0$ and
  $\PL'(\lambda_0)y_0 \notin \range(\PL(\lambda_0))$. 
  For more details on root polynomials, partial
  and algebraic multiplicities we refer to 
  \cite{KozlovMazja1999,Markus1988,MennickenMoeller2003}.
    % \cite{GohbergLancasterRodman2009} 
     Our definition of essential spectrum follows
    \cite{Henry1981}.
\end{remark}
By definition, the spectrum $\sigma(\PL)$ of $\PL$ can be decomposed
into its point spectrum and its essential spectrum
\begin{equation*}
  \sigma(\PL)=\sigma_{\mathrm{ess}}(\PL)\, \dot{\cup}\, 
\sigma_{\mathrm{point}}(\PL).
\end{equation*}
The function spaces underlying the definition of spectra are
subspaces of $L^2(\R,\R^m)$ which will be specified in Section~\ref{sec:3}
and Appendix \ref{sec:A}.
In this section we carry out formal calculations
without reference to a specific function space.
%---------------------------------------------------------------------------------------------------------------------------------------------------
\subsection{Point spectrum  on the imaginary axis}
\label{subsec:4.1}
%---------------------------------------------------------------------------------------------------------------------------------------------------
Applying $\partial_{\xi}$ to the traveling wave equation
\eqref{equ:2.4},
% \begin{equation}
%   \label{equ:4.3}
%   0 = (A-\mu_{\star}^2 M)v_{\star,\xi\xi}(\xi) +
%   (C+\mu_{\star}B)v_{\star,\xi}(\xi) + f(v_{\star}(\xi)),\,\xi\in\R
% \end{equation}
leads to the equation
\begin{equation*}
  \begin{aligned}
   % \label{equ:4.5}
    \begin{split}
      0 = & (A-\mu_{\star}^2 M)v_{\star,\xi\xi\xi} +
      D_2f(\star)v_{\star,\xi\xi} + 
      D_1f(\star)v_{\star,\xi} - \mu_{\star}D_3f(\star)v_{\star,\xi \xi} =
      -P_0 v_{\star,\xi},\,\xi\in\R,
    \end{split}
  \end{aligned}
\end{equation*}
provided that $v_{\star}\in C^3(\R,\R^m)$ and $f\in C^1(\R^{3m},\R^m)$.
Therefore, $w=v_{\star,\xi}$ solves the quadratic eigenvalue
problem $\PL(\lambda)w=0$ for $\lambda=0$, and $w=v_{\star,\xi}$ is an
eigenfunction if the wave profile $v_{\star}$ is nontrivial (i.e. not
constant).  This behavior is to be expected since the original
equation is equivariant with respect to the shift, and the spatial
derivative $\partial_{\xi}$ is the generator of shift equivariance.
\begin{proposition}[Point spectrum of traveling waves]\label{thm:4.1}
  Let $v_{\star}\in C^3(\R,\R^m)$, $\mu_\star$ be a nontrivial
  classical solution of \eqref{equ:2.4} and $f\in
  C^1(\R^{3m},\R^m)$. Then $\lambda=0$ is an eigenvalue with
  eigenfunction $v_{\star,\xi}$ of the quadratic eigenvalue problem
  $\PL(\lambda)w=0$.  In particular, $0\in  \sigma_{\mathrm{point}}(\PL)$.
\end{proposition}
As usual, further isolated eigenvalues are difficult to detect analytically,
and we refer to the extensive literature on solving quadratic eigenvalue problems and on locating zeros of the so-called Evans function, see e.g.
\cite{blr14,Sandstede2007}.

% \begin{remark}
%   In fact the result from Theorem \ref{thm:4.1} is not rigorous since investigating spectra of differential operators requires 
%   a suitable underlying function space $X$ to which the eigenfunction $v_{\star,\xi}$ must belong. However, if we transfer the 
%   second order system to a first order system as in Section \ref{sec:3} below, an application of results for first order systems 
%   implies that $\sigma_{\mathrm{point}}^{\mathrm{part}}(\PL)$ belongs to the $L^p$-point spectrum of $\PL$, i.e. $v_{\star,\xi}\in X=L^p$, 
%   $1<p<\infty$. Similarly, one can show under appropriate assumptions that the profile $v_{\star}$ and the eigenfunction $v_{\star,\xi}$ 
%   belong to some exponentially weighted $L^p$-space, which means that both functions decay exponentially in space.
% \end{remark}

\begin{example}[Nagumo wave equation]\label{exa5}
  Recall from Example~\ref{exa1} that the Nagumo wave equation~\eqref{equ:2.38}
   has an explicit traveling front solution $u_{\star}(x,t)=v_{\star}(\xi)$, $\xi=x-\mu_{\star}t$,
  %\begin{equation*}
  %  v_{\star}(\xi)=\frac{1}{1+\exp\left(-\frac{k\xi}{\sqrt{2}}\right)},\quad \mu_{\star}=\frac{\sqrt{2}\left(\frac{1}{2}-b\right)}{k},\quad k=\pm\sqrt{1+2\varepsilon\left(\frac{1}{2}-b\right)^2},
  %\end{equation*} 
  with $v_{\star}$ and $\mu_{\star}$ from \eqref{equ:2.39}, i.e. $v_{\star}$ and $\mu_{\star}$ solve the associated traveling wave equation
  \begin{equation*}
    0 = (1-\mu_{\star}^2\varepsilon)v_{\star,\xi\xi}(\xi) + \mu_{\star}
    v_{\star,\xi}(\xi) +
    v_{\star}(\xi)\left(1-v_{\star}(\xi)\right)\left(v_{\star}(\xi)-b\right),\quad\xi\in\R.
  \end{equation*}
  The quadratic eigenvalue problem for the linearization then reads as follows,
  \begin{equation*}
  [\PL(\lambda) w](\xi)=\varepsilon \left(\lambda-\mu_{\star}\partial_{\xi}
\right)^2w(\xi) +\left(\lambda-\mu_{\star}\partial_{\xi}\right)
 w(\xi) - w_{\xi \xi}(\xi)+\left(3v_{\star}^2(\xi) - 2(b+1)v_{\star}(\xi) -b\right)
 w(\xi) =0,\;\xi\in\R.
  \end{equation*}
  With $k$ from \eqref{equ:2.39}, it has the solution
  \begin{equation*}
    \lambda=0,\quad w(\xi)=v_{\star,\xi}(\xi) = \tfrac{k}{\sqrt{2}}\exp\left(-\tfrac{k\xi}{\sqrt{2}}\right)\left(1+\exp\left(-\tfrac{k\xi}{\sqrt{2}}\right)\right)^{-2},\quad \xi\in\R.
  \end{equation*}
\end{example}

%---------------------------------------------------------------------------------------------------------------------------------------------------
\subsection{Essential spectrum and dispersion relation of traveling waves}
\label{subsec:4.2}
%---------------------------------------------------------------------------------------------------------------------------------------------------
The essential spectrum of $\PL$ from \eqref{equ:4.1}, \eqref{equ:defPj}, is 
determined by the 
constant coefficient operators obtained by letting
$\xi \rightarrow \pm \infty$ in the coefficient operators $P_0,P_1$
 (recall $(\pm)=(v_{\pm},0,0)$),
\begin{equation} \label{equ:4.asyop}
   \begin{aligned}
   % \label{equ:4.1a}
    \PL^{\pm}(\lambda) = & \lambda^2 P_2 + \lambda P_1^{\pm} + P_0^{\pm},
     \quad \lambda \in \C, \\
     P_1^{\pm}= &-D_3f(\pm)- 2 \mu_{\star}M \partial_{\xi},
   \quad  P_0^{\pm} = -(A-\mu_{\star}^2 M)\partial_{\xi}^2
      + (\mu_{\star}D_3f(\pm)-D_2f(\pm))\partial_{\xi} -
    D_1f(\pm).
    \end{aligned}
\end{equation}
We seek bounded solutions $w$ of $\PL^{\pm}(\lambda)w=0$ by the
Fourier ansatz $w(\xi)= e^{i \omega \xi}z, z \in \C^m, |z|=1$ and
arrive at the following quadratic eigenvalue problem
\begin{equation*}
 % \label{equ:4.10}
  \A_{\pm}(\lambda,\omega)z = \left(\lambda^2 A_2 + \lambda A_1^{\pm}(\omega) + A_0^{\pm}(\omega)\right)z = 0
\end{equation*}
with matrices
% $A_{0}^{\pm}(\omega)$, $A_1(\omega)\in\C^{m,m}$ and
% $A_2\in\R^{m,m}$% given by
\begin{equation}
  \label{equ:4.11}
  A_2=M,\quad A_1^{\pm}(\omega)=-D_3f(\pm)-2i\omega\mu_{\star}M,
  \quad A_0^{\pm}(\omega) = \omega^2(A-\mu_{\star}^2 M)+
  i\omega(\mu_{\star}D_3f(\pm)-D_2f(\pm)) - D_1f(\pm).
\end{equation}
Every $\lambda\in\C$ satisfying the \begriff{dispersion relation}
\begin{equation}
  \label{equ:4.12}
  \det\left(\lambda^2 A_2 + \lambda A_1^{\pm}(\omega) + A_0^{\pm}(\omega)\right) = 0 
\end{equation}
for some $\omega\in\R$ and either sign, belongs to the essential spectrum of $\PL$. A proof of this statement
is obtained in the standard way by cutting off $w(\xi)$ at $\xi\notin
[n,2n]$
resp. $\xi \notin [-2n,-n]$
and letting $n \rightarrow \infty$. Then this contradicts
the continuity of the resolvent at $\lambda$ in appropriate function
spaces.
 This proves the following result:
 \begin{proposition}[Essential spectrum of traveling waves]\label{thm.4.2}
   Let $f\in C^1(\R^{3m},\R^m)$ with $f(v_{\pm},0,0)=0$ for some
   $v_{\pm}\in\R^m$.  Let $v_{\star}\in C^2(\R,\R^m)$, $\mu_\star$ be
   a nontrivial classical solution of \eqref{equ:2.4} satisfying
   $v_{\star}(\xi)\to v_{\pm}$ as $\xi\to\pm\infty$. 
  Then, the dispersion set
   set %\begriff{algebraic curves} (\begriff{asymptotic parabolas})
  \begin{equation} \label{equ:4.11a}
  \begin{aligned}
    \sigma_{\mathrm{disp}}(\PL) :=
    \left\{\lambda\in\C\mid\text{$\lambda$ satisfies \eqref{equ:4.12}
        for some $\omega\in\R$ and some sign $\pm$} \right\}
  \end{aligned}
  \end{equation}
  belongs to the essential spectrum $\sigma_{\mathrm{ess}}(\PL)$ of
  $\PL$. 
\end{proposition}

  In the general matrix case it is not easy
to analyze the shape of the algebraic set
$\sigma_{\mathrm{disp}}(\PL)$, since \eqref{equ:4.12}
amounts to finding the zeroes of a polynomial of degree $2m$. In view
of the stability results in Theorem~\ref{thm:4.17} and Theorem~\ref{thm:freezestab} our main interest is
in finding a spectral gap, i.e.\ a constant $\beta >0$ such that
\begin{equation} \label{equ:spectralgap}
\Re\lambda \le - \beta <0 \quad \text{for all} \quad 
\lambda \in\sigma_{\mathrm{disp}}(\PL).
\end{equation}
We discuss this condition for three subcases of the special structure \eqref{equ:2.1a}.
  \begin{itemize}[leftmargin=0.66cm]\setlength{\itemsep}{0.1cm}
  % Parabolic case
    \item[(i)] \textbf{Parabolic case: ($\mathbf{M=0}$, $\mathbf{B=I_m}$, $\mathbf{C=0}$).} The dispersion relation \eqref{equ:4.12} reads 
  \begin{equation}
    \label{equ:4.12a}
    \det\left(\tilde{\lambda} I_m + \omega^2 A  - Dg(v_{\pm})\right) = 0,
\quad \tilde{\lambda}=\lambda-i\omega\mu_{\star},
  \end{equation}
and the corresponding eigenvalue problem may be written as
\begin{equation}
    \label{equ:4.12b}
    \tilde{\lambda}z = -\left(\omega^2 A - Dg(v_{\pm})\right)z,
\quad 0\neq z\in\C^m,\quad \tilde{\lambda}=\lambda-i\omega\mu_{\star}.
  \end{equation}
 Let us assume positivity of $A$ and $-Dg(v_{\pm})$  in the 
sense that
\begin{equation} \label{equ:Afpositive}
\Re z^HAz >0, \quad \Re z^H Dg(v_{\pm})z<0 \quad \text{for all} \; z \in \C^m.
\end{equation}
Multiplying \eqref{equ:4.12b} by $z^H$ and taking the real part, shows
that the solutions $\tilde{\lambda}$ of \eqref{equ:4.12a} have 
negative real parts and the gap is guaranteed. This is still true
if $A$ is nonnegative but has zero eigenvalues. Note that in this
case, equation 
\eqref{equ:2.7} is of mixed hyperbolic-parabolic type and the
nonlinear stability theory becomes considerably more involved,
see   \cite{RottmannMatthes2012c}.
  % Hyperbolic case
  \item[(ii)] \textbf{Undamped hyperbolic case: ($\mathbf{M=I_m}$, $\mathbf{B=0}$, $\mathbf{C=0}$).} The dispersion relation \eqref{equ:4.12} reads 
  \begin{equation*}
   % \label{equ:4.12c}
    \det\left(\tilde{\lambda}^2 I_m + \omega^2 A- Dg(v_{\pm})\right) = 0, \quad
\tilde{\lambda}=\lambda-i\omega\mu_{\star}
  \end{equation*}
  Whenever $\lambda\in \C,\omega \in \R$ solve this system, so does
  the pair $-\lambda,-\omega$. Hence, the eigenvalues lie either on
  the imaginary axis or on both sides of the imaginary axis.
  Therefore, a spectral gap cannot exist. This is the Hamiltonian case, where
one can only 
  expect stability (but not asymptotic stability) of the wave. We refer to
 the local stability theory developed in
  \cite{GrillakisShatahStrauss1987},\cite{GrillakisShatahStrauss1990}
  (see also \cite{KapitulaPromislov2013} for a recent account).  Note
  that in this case the positivity assumption \eqref{equ:Afpositive}
  only guarantees $\Re \tilde{\lambda}^2 <0$, i.e.
  $\tfrac{\pi}{4}<|\mathrm{arg}(\tilde{\lambda})|\leqslant\tfrac{\pi}{2}$ for
  $\tilde{\lambda}=\lambda-i\omega\mu_{\star}$ and all eigenvalues $\lambda \in \sigma(\A(\cdot,\omega))$.
\item[(iii)] {\bf Scalar case: ($\mathbf{M=1}$, $\mathbf{B=\eta}$, $\mathbf{C=0}$).} It is instructive to discuss the dispersion
relation \eqref{equ:4.12} in the scalar case with $A=a$, $-Dg(v_{\pm})=\delta$ and real numbers $a,\eta,\delta >0$
\begin{equation} \label{equ:dispersescalar}
\tilde{\lambda}^2+ \eta \tilde{\lambda}+ a\omega^2 + \delta = 0,
\quad \tilde{\lambda}=\lambda-i \omega \mu_{\star}.
\end{equation}
 This case occurs with the Nagumo wave equation below. The solutions are
\begin{equation*}
\lambda= i \omega \mu_{\star} - \frac{\eta}{2} \pm \left(\frac{\eta^2}{4} -\delta 
- \omega^2 a  \right)^{1/2}, \quad \omega \in \R.
\end{equation*}
If $\eta^2 \le 4 \delta$, then all solutions $\lambda$ of
\eqref{equ:dispersescalar} lie on the vertical line $\Re \lambda =
-\frac{\eta}{2} < 0$.  A short discussion shows that they actually cover this line under the
assumption $\mu_{\star}^2 <a$, which corresponds to positivity of the matrix
$A-\mu_{\star}^2M$ occuring in  \eqref{equ:2.4}.
 If $\eta^2 > 4 \delta$ then the solutions
$\lambda$ of \eqref{equ:dispersescalar} lie again on this line
(resp.\ cover it if $\mu_{\star}^2 <a$) for values $|\omega|\ge
\omega_0:= (\frac{1}{a}(\frac{\eta^2}{4}-\delta))^{1/2}$. But for
values $|\omega| \le \omega_0$ they form the ellipse
\begin{equation} \label{equ:ellipse}
\frac{(\Re \lambda + \frac{\eta}{2})^2}{p_1^2} + \frac{(\Im \lambda)^2}{p_2^2}=1,\quad \text{with semiaxes} \quad p_1=a^{1/2} \omega_0, \; \; p_2 = |\mu_{\star}|\omega_0.
\end{equation}
 The rightmost point  of the ellipse
$-\beta:= -\frac{\eta}{2} +\left(\frac{\eta^2}{4}-\delta\right)^{1/2}$ is
still negative and therefore can be taken for the spectral gap
\eqref{equ:spectralgap}.
  \end{itemize}

\begin{example}[Spectrum of Nagumo wave equation]\label{exa6}
  As in Example~\ref{exa1}, consider the Nagumo wave equation~\eqref{equ:2.38}
  % \begin{equation*}
  %   \varepsilon u_{tt} + u_t = u_{xx} + u(1-u)(u-b),\,x\in\R,\,t\geqslant 0,\,0<b<1,\,\varepsilon>0
  % \end{equation*}
  with coefficients 
  \begin{equation*}
    M=\varepsilon>0,\quad A=B=1,\quad C=0.
  \end{equation*}

 There is a  traveling front solution $u_{\star}(x,t)=v_{\star}(x-\mu_{\star}t)$ with $v_{\star}$, $\mu_{\star}$ from \eqref{equ:2.39}. 
 %  % Let us restrict our considerations to the case where $k$ from \eqref{equ:2.39} has a positive sign, i.e. $k=k_+$. Then, the asymptotic 
 With the asymptotic  states $v_+=1$, $v_-=0$ and $g'(v_+)=b-1$, $g'(v_-)=-b$  from \eqref{equ:2.36}, we find the dispersion relation 
  \begin{equation*}
    \begin{aligned}
   % \label{equ:4.13}
    \varepsilon\tilde{\lambda}^2 + \tilde{\lambda} + \omega^2+b=0
    \quad\text{or}\quad
    \varepsilon\tilde{\lambda}^2 + \tilde{\lambda}
    + \omega^2-b+1=0.
    % \varepsilon\tilde{\lambda}^2 + \tilde{\lambda} +
    % \omega^2-g'(v_{\pm})=0.
    \end{aligned}
  \end{equation*}
 %  % for some $\omega\in\R$ belongs to $\sigma_{\mathrm{ess}}(\PL)$, meaning that
 %  % \begin{align*}
 %  %   \sigma_{\mathrm{disp}}(\PL) = \sigma(\A_{\pm}) = \left\{\lambda\in\C\mid\text{$\lambda$ satisfies \eqref{equ:4.13} for some $\omega\in\R$}\right\} \subseteq \sigma_{\mathrm{ess}}(\PL).
 %  % \end{align*} 
 %  The quadratic problem \eqref{equ:4.13} has the solutions
 %  \begin{equation}
 %  \begin{aligned}
 %    \label{equ:4.14}
 %    \lambda(\omega) = -\frac{1}{2\varepsilon}+i\omega\mu_{\star} \pm \frac{\sqrt{1-4\varepsilon(\omega^2-g'(v_{\pm}))}}{2\varepsilon},\,\omega\in\R.
 %  \end{aligned}
 %  \end{equation}
 %  Note that
 %  \begin{align*}
 %    \sqrt{1-4\varepsilon(\omega^2-g'(v_{\pm}))}\begin{cases}\in\R&,\,|\omega|\leqslant\omega_{\pm}\\\in i\R&,\,|\omega|>\omega_{\pm}\end{cases}\quad\text{for}\quad
 %    \omega_{\pm}=\omega_{\pm}(\varepsilon)=\sqrt{\frac{1}{4\varepsilon}+g'(v_{\pm})}.
 %  \end{align*}
  The scalar case discussed above applies with the settings
  $\eta=\frac{1}{\varepsilon}=a$, $\delta_{\pm}=-
  \frac{g'(v_{\pm})}{\varepsilon}$. Thus the subset
  $\sigma_{\mathrm{disp}}(\PL)$ of the essential
  spectrum lies on the union of the line $\Re \lambda = -
  \frac{1}{2\epsilon}$ and possibly two ellipses defined by
  \eqref{equ:ellipse} with $\omega_0=\omega_{\pm}= \left(\frac{1}{4
      \varepsilon} + g'(v_{\pm})\right)^{1/2}$.  The ellipse belonging
  to $v_+$ occurs if $1-b < \frac{1}{4 \varepsilon}$, and the one
  belonging to $v_-$ occurs if $b < \frac{1}{4 \varepsilon}$.  Since
  $0<b<1$ both ellipses show up in
  $\sigma_{\mathrm{disp}}(\PL)$ if $\varepsilon \le
  \frac{1}{4}$. In any case, there is a gap beween the essential
  spectrum and the imaginary axis in the sense of
  \eqref{equ:spectralgap} with
\begin{align*} 
\beta= \frac{1}{2 \varepsilon}\left(1 -   \left(1 - 4 \varepsilon^2 \min(b,1-b)\right)^{1/2} \right).
\end{align*}

Figure~\ref{fig:DampedNagumo_1D_Traveling1Front_EssentialSpectrum}(a)
shows that piece of spectrum which is guaranteed by our
propositions at parameter values $\varepsilon=b=\frac{1}{4}$.  It is
subdivided into point spectrum (blue circle) determined by
Proposition~\ref{thm:4.1}, and essential spectrum (red lines) determined
by Proposition~\ref{thm.4.2}. There may be further isolated
eigenvalues. The numerical spectrum of the Nagumo wave   
on the spatial domain $[-R,R]$ and subject to periodic boundary conditions, 
is shown in Figure~\ref{fig:DampedNagumo_1D_Traveling1Front_EssentialSpectrum}(b) 
for $R=50$ and in Figure~\ref{fig:DampedNagumo_1D_Traveling1Front_EssentialSpectrum}(c) 
for $R=400$. Each of them consists of the approximations of the point spectrum 
(blue circle) and of the essential spectrum (red dots). The missing line 
inside the ellipse in Figure~\ref{fig:DampedNagumo_1D_Traveling1Front_EssentialSpectrum}(b) 
gradually appears numerically when enlarging the spatial domain, see 
Figure~\ref{fig:DampedNagumo_1D_Traveling1Front_EssentialSpectrum}(c).
The second ellipse only develops on even larger domains.
%Figure \ref{fig:DampedNagumo_1D_Traveling1Front_EssentialSpectrum}(c)
%shows the approximation of the eigenfunction $w(x)=v_{\star,x}(x)$
%from Example \ref{exa5} belonging to the approximation of
%$\lambda=0$. An approximation of $v_{\star}$ has been shown in Figure
%\ref{fig:FrozenDampedNagumoTraveling1Front}(a) before.
\end{example}

\begin{figure}[ht]
  \centering
  \subfigure[]{\includegraphics[height=3.3cm] {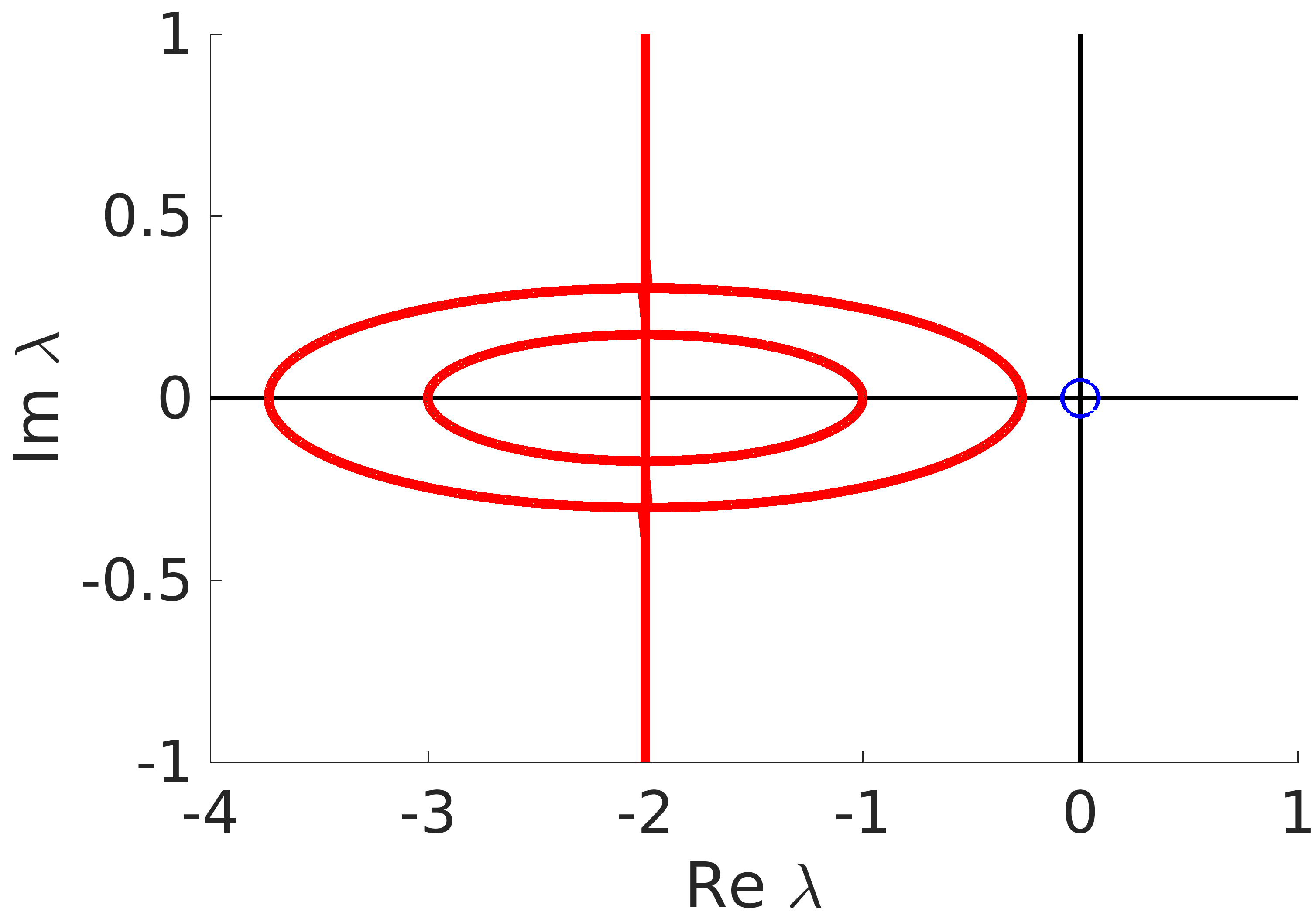}\label{fig:DampedNagumoEssentialSpectrum}}
  \subfigure[]{\includegraphics[height=3.3cm] {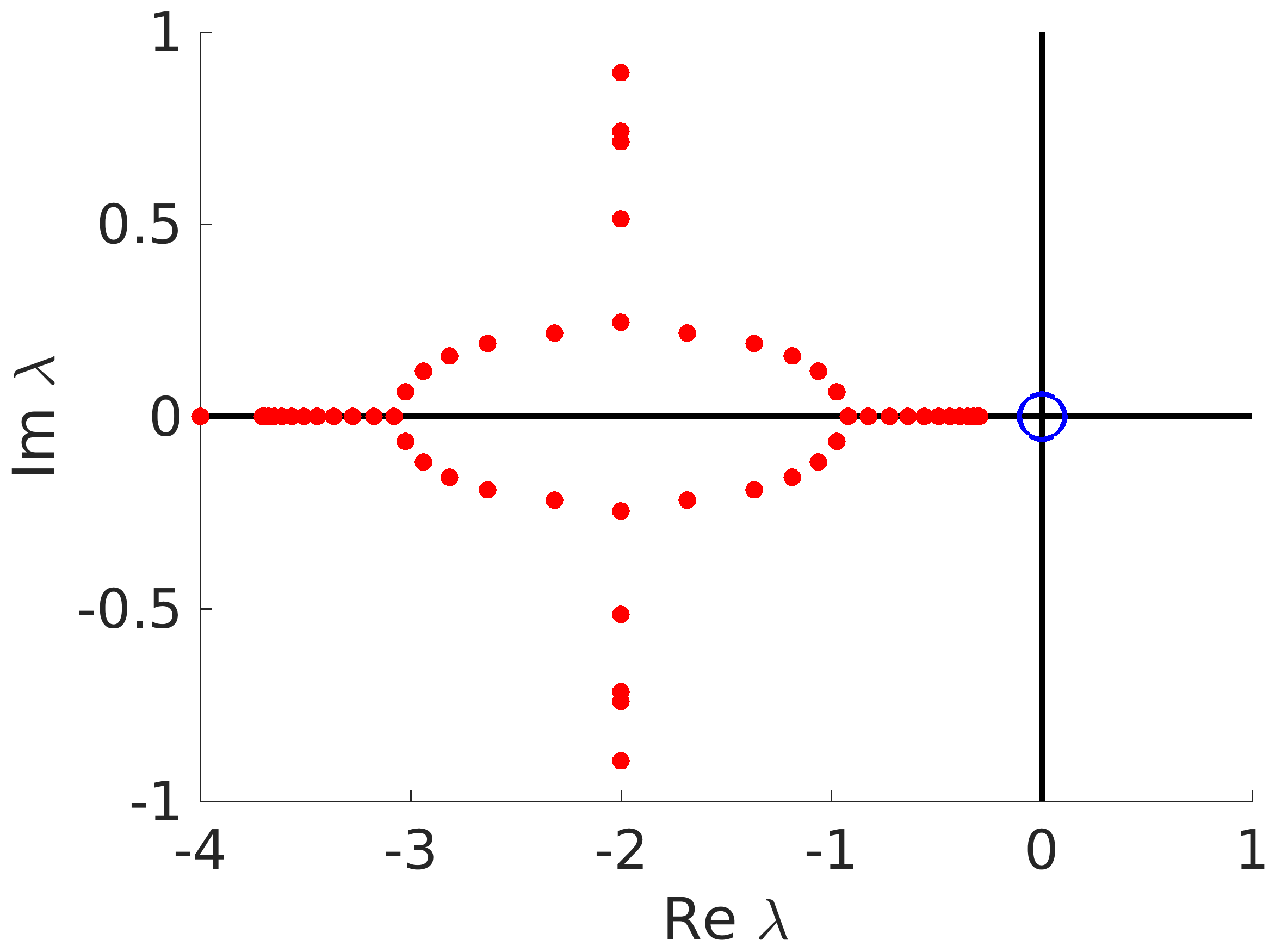} \label{fig:DampedNagumoNumericalSpectrum}}
  \subfigure[]{\includegraphics[height=3.3cm] {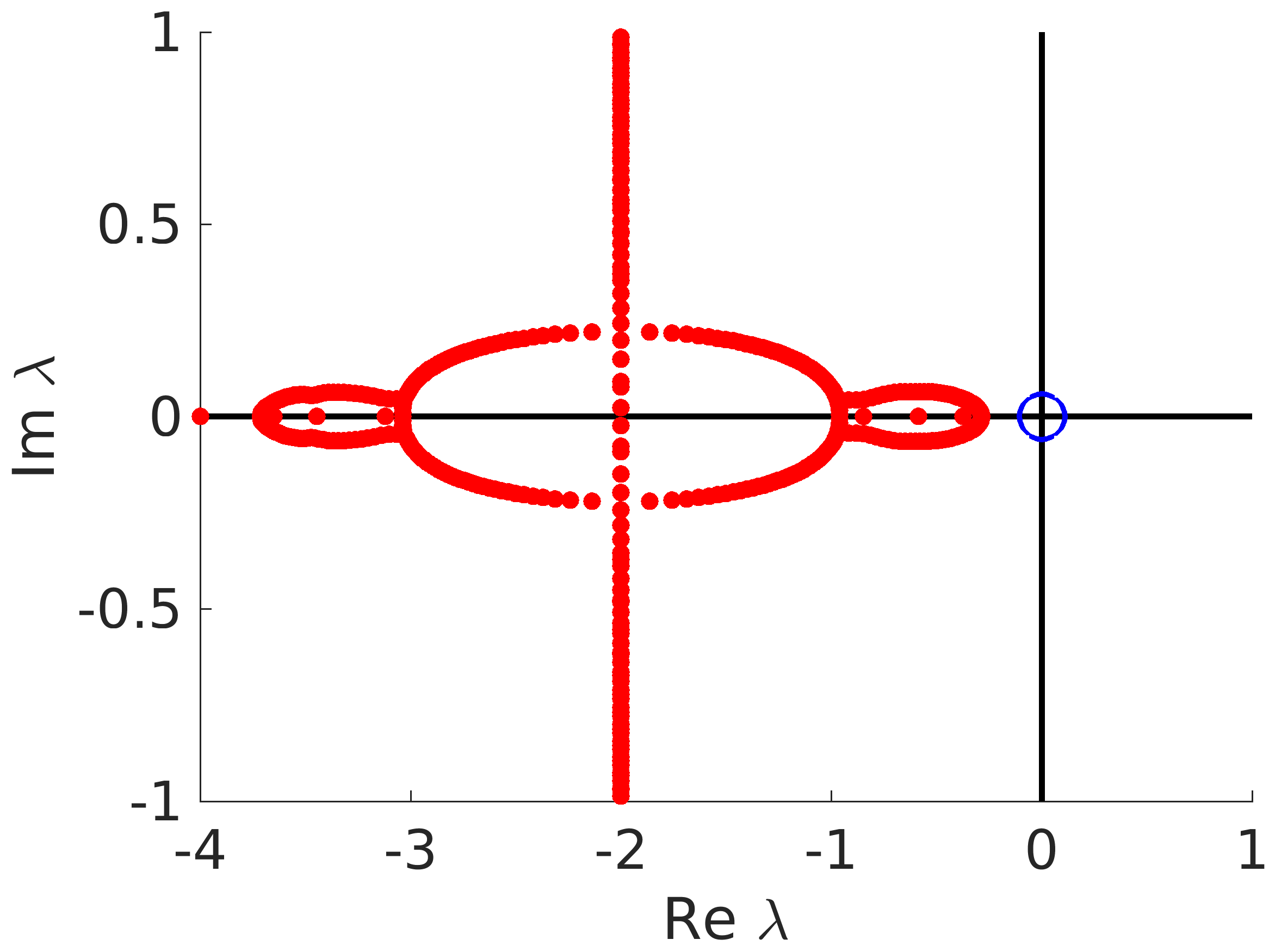}\label{fig:DampedNagumoEigenfunction}}
  \caption{Essential spectrum of the Nagumo wave equation for parameters $\varepsilon=b=\frac{1}{4}$ (a) and the numerical spectrum on the spatial domain $[-R,R]$ 
  for $R=50$ (b) and $R=400$ (c).}
  \label{fig:DampedNagumo_1D_Traveling1Front_EssentialSpectrum}
\end{figure}

\begin{example}[Spectrum of FitzHugh-Nagumo wave system]\label{exa7}
  As shown in Example~\ref{exa2}, the FitzHugh-Nagumo wave system
\eqref{equ:2.43}
  % \begin{equation*}
  %   M u_{tt} + u_t = A u_{xx} + \begin{pmatrix}u_1-\frac{1}{3}u_1^3-u_2\\\phi(u_1+a-bu_2)\end{pmatrix},\,x\in\R,\,t\geqslant 0
  % \end{equation*}
  with coefficient matrices
  \begin{equation*}
    M=\varepsilon I_2,\quad A=\mathrm{diag}(1,\tfrac{\rho+c_{\star}^2\varepsilon}{1+c_{\star}^2\varepsilon}), \quad B=I_2,\quad 
    %A=\begin{pmatrix}1&0\\0&\frac{\rho+c_{\star}^2\varepsilon}{1+c_{\star}^2\varepsilon}\end{pmatrix}, \quad
    C=0
  \end{equation*}
  and parameters from \eqref{equ:2.40a}
  % \begin{equation} \label{equ:FHNpar}
  %   \rho=0.1,\quad \phi=0.08,\quad a=0.7,\quad b=0.8,\quad\varepsilon>0
  % \end{equation}
  has a traveling pulse solution
  $u_{\star}(x,t)=v_{\star}(x-\mu_{\star}t)$ with
  \begin{equation*}
    \mu_{\star}=\frac{c_{\star}}{k},\quad k=\sqrt{1+c_{\star}^2\varepsilon},\quad c_{\star}\approx-0.7892.
  \end{equation*} 
  The profile $v_{\star}$ connects the asymptotic state $v_{\pm}=w_{\pm}$ from \eqref{equ:2.41} with itself, 
  i.e. $v_{\star}(\xi)\to v_{\pm}$ as $\xi\to\pm\infty$. The profile 
$v_{\star}$ and the velocity $\mu_{\star}$ are obtained from the simulation  
performed in Example \ref{exa2}. The FitzHugh-Nagumo nonlinearity $g$ from
\eqref{equ:2.40} 
  % \begin{equation*}
  %   f:\R^2\rightarrow\R^2,\quad g(u)=\begin{pmatrix}u_1-\frac{1}{3}u_1^3-u_2\\\phi(u_1+a-bu_2)\end{pmatrix}
  % \end{equation*}
  satisfies
  \begin{equation*}
    g(v_{\pm})=\begin{pmatrix}0\\0\end{pmatrix}\quad\text{and}\quad Dg(v_{\pm})=\begin{pmatrix}1-\left(v_{\pm,1}\right)^2&-1\\\phi&-b\phi\end{pmatrix}.
  \end{equation*}
  The dispersion relation for the FitzHugh-Nagumo pulse states that every $\lambda\in\C$ satisfying
  \begin{equation}
    \begin{aligned}
    \label{equ:4.15}
    \det\begin{pmatrix}\varepsilon\lambda^2+p(\omega)\lambda + q_1(\omega)&1\\
    -\phi&\varepsilon\lambda^2+p(\omega)\lambda + q_2(\omega)\end{pmatrix}=0.
    \end{aligned}
  \end{equation}
  for some $\omega\in\R$ belongs to $\sigma_{\mathrm{ess}}(\PL)$, where we used the abbreviations 
  \begin{equation*}
  \begin{aligned}
    p(\omega) = 1-2i\omega\mu_{\star}\varepsilon,\;
    q_1(\omega) = \omega^2(1-\mu_{\star}^2\varepsilon)-i\omega\mu_{\star}-(1-(v_{\pm,1})^2),\;
    q_2(\omega) = \omega^2\left(\frac{\rho+c_{\star}^2\varepsilon}{1+c_{\star}^2\varepsilon}-\mu_{\star}^2\varepsilon\right)
    -i\omega\mu_{\star}+b\phi.
  \end{aligned}
  \end{equation*}
  % This implies
  % \begin{align*}
  %   \sigma_{\mathrm{disp}}(\PL)  = \left\{\lambda\in\C\mid\text{$\lambda$ satisfies \eqref{equ:4.15} for some $\omega\in\R$}\right\} \subseteq \sigma_{\mathrm{ess}}(\PL).
  % \end{align*} 
  Note that \eqref{equ:4.15} leads to the quartic problem
  \begin{equation*}
  \begin{aligned}
    0 %& = \det\begin{pmatrix}a_1-\lambda-\varepsilon\lambda^2&-1\\\phi&a_2-\lambda-\varepsilon\lambda^2\end{pmatrix} = (a_1-\lambda-\varepsilon\lambda^2)(a_2-\lambda-\varepsilon\lambda^2)+\phi\\
      %& = \varepsilon^2\lambda^4 + 2\varepsilon\lambda^3 + (1-\varepsilon(a_1+a_2))\lambda^2 - (a_1+a_2)\lambda + \phi + a_1a_2 \\
      & = a_4\lambda^4 + a_3\lambda^3 + a_2\lambda^2 + a_1\lambda + a_0
  \end{aligned}
  \end{equation*}
  with $\omega$-dependent coefficients
  \begin{equation*}
    a_4=\varepsilon^2,\quad a_3=2\varepsilon p,\quad a_2=\varepsilon(q_1+q_2)+p^2,\quad a_1=p(q_1+q_2),\quad a_0=q_1q_2+\phi.
  \end{equation*}

  \begin{figure}[ht]
  \centering
  \subfigure[]{\includegraphics[height=3.3cm] {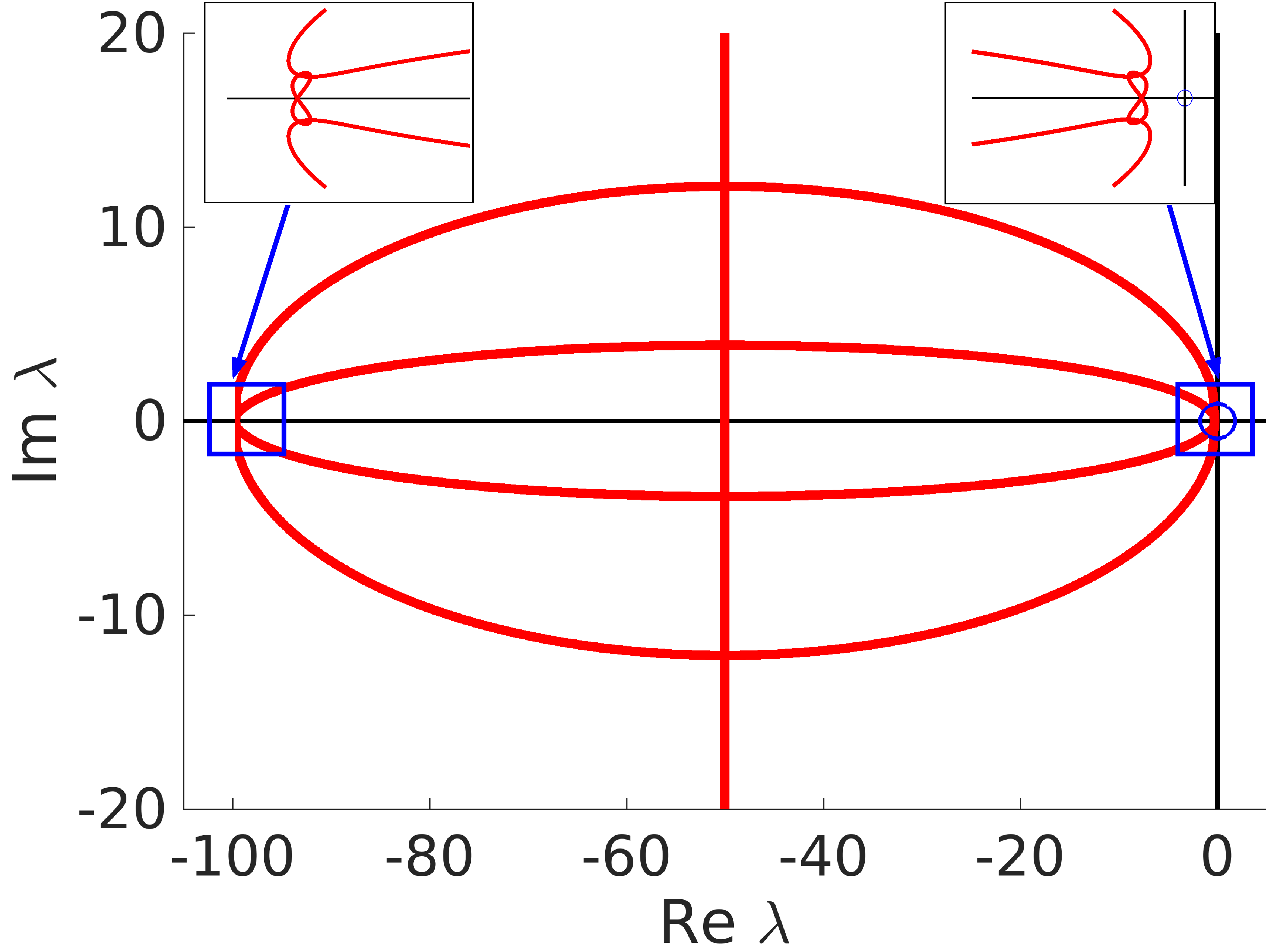}\label{fig:DampedFitzHughNagumoEssentialSpectrum}}
  \subfigure[]{\includegraphics[height=3.3cm] {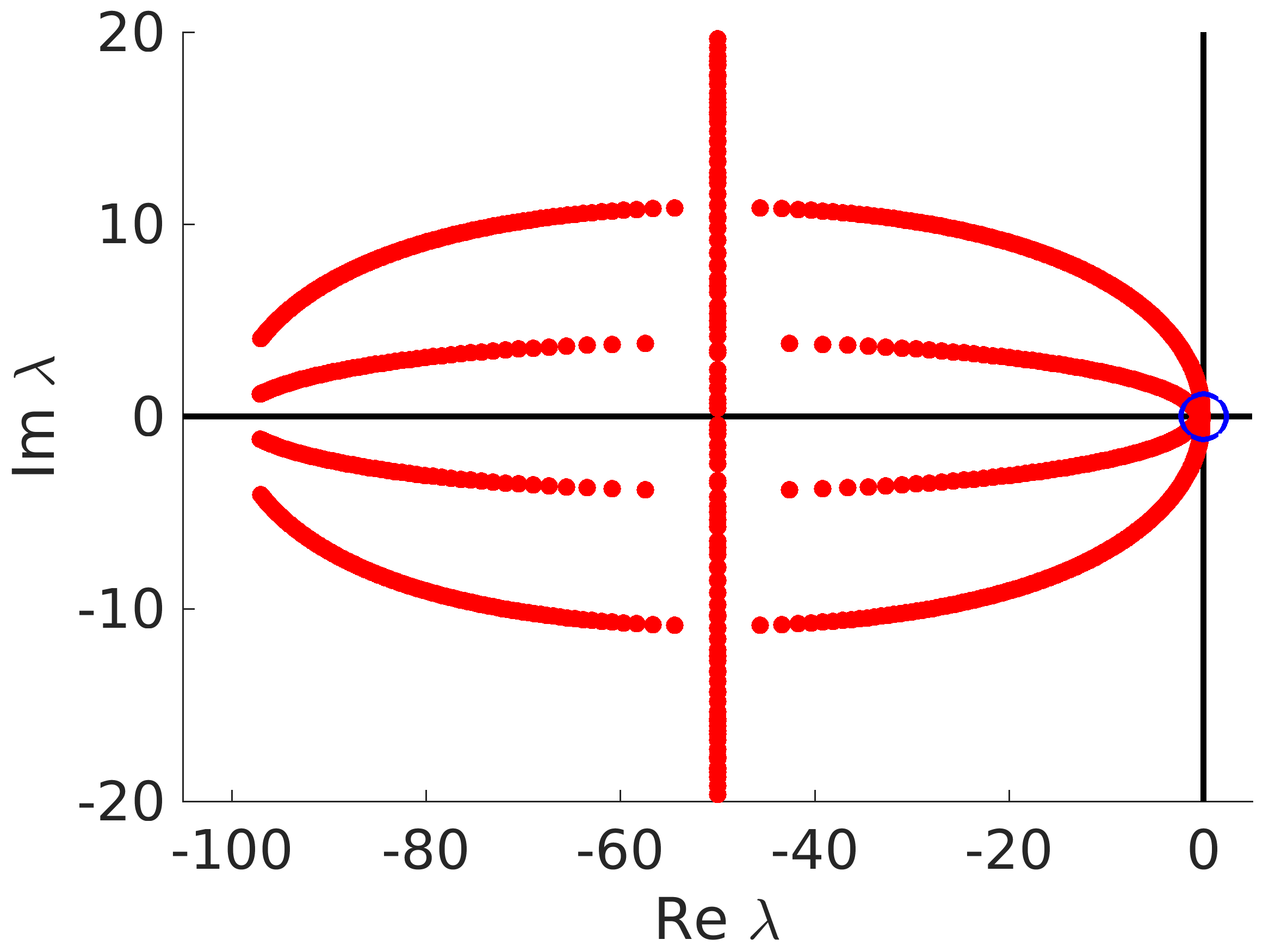} \label{fig:DampedFitzHughNagumoNumericalSpectrum}}
  \subfigure[]{\includegraphics[height=3.3cm] {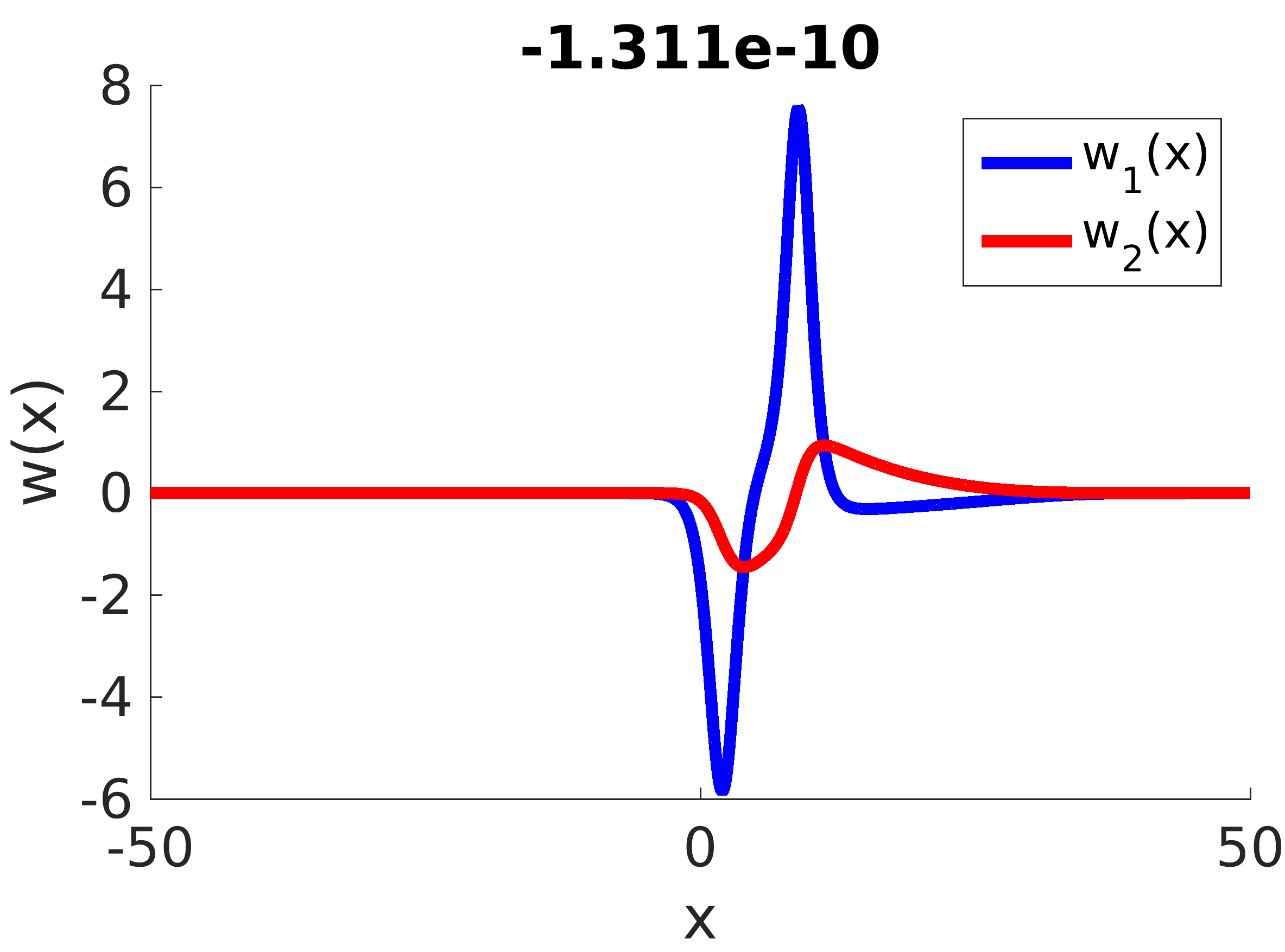}\label{fig:DampedFitzHughNagumoEigenfunction}}
  \caption{Essential spectrum of the FitzHugh-Nagumo wave system for parameters from \eqref{equ:2.40a} and $\varepsilon=10^{-2}$ (a), the numerical spectrum (b) and both components of 
  the eigenfunction belonging to $\lambda\approx 0$ (c).}
  \label{fig:DampedFitzHughNagumo_1D_Traveling1Pulse_EssentialSpectrum}
  \end{figure}

  \noindent
  Instead of this we solved numerically the quadratic eigenvalue
  problem \eqref{equ:4.15} using parameter continuation with respect to $\omega$. 
  In this way we obtain analytical
  information about the spectrum of the FitzHugh-Nagumo pulse shown in
  Figure~\ref{fig:DampedFitzHughNagumo_1D_Traveling1Pulse_EssentialSpectrum}(a)
  (red lines) for $\varepsilon=10^{-2}$.  Again part of the point
  spectrum (blue circle) is determined by Proposition~\ref{thm:4.1}
  and part of the essential spectrum (red lines) by
  Proposition~\ref{thm.4.2}. Zooming into the essential spectrum shows
  that the parabola-shaped structure contains at both ends a loop which is
  already known from the first order limit case, see \cite{BeynLorenz1999}. %Paper Jens ??? \cite{}. 
  From these results
  it is obvious that there is again a spectral gap to the imaginary
  axis, but we have no analytic expression for this gap.  The
  numerical spectrum for periodic boundary conditions is shown in
  Figure~\ref{fig:DampedFitzHughNagumo_1D_Traveling1Pulse_EssentialSpectrum}(b).
  It
  consists of the approximations of the point spectrum (blue circle)
  and of the essential spectrum (red dots).  
  Figure~\ref{fig:DampedFitzHughNagumo_1D_Traveling1Pulse_EssentialSpectrum}(c)
  shows the approximation of both components $w_1$ and $w_2$ of the
  eigenfunction $w(\xi)\approx v_{\star,\xi}(\xi)$ belonging to the
  small eigenvalue $\lambda=1.311\cdot10^{-10}$ which approximates the eigenvalue $0$. 
  Note that an approximation of $v_{\star}=(v_{\star,1},v_{\star,2})^T$ was
  provided in Figure~\ref{fig:FrozenDampedFitzHughNagumoTraveling1Pulse}(a).
\end{example}

%---------------------------------------------------------------------------------------------------------------------------------------------------
%
%  SECTION 4: (Transformation to first order systems)
%
%---------------------------------------------------------------------------------------------------------------------------------------------------
\sect{First order systems and stability of traveling waves}
\label{sec:3}
%---------------------------------------------------------------------------------------------------------------------------------------------------
In this section we transform the original second order damped wave
equation \eqref{equ:2.1} into a first order system of triple size. 
To the first order system we then apply stability results from
\cite{RottmannMatthes2012a} and derive asymptotic stability of
traveling waves for the original second order problem and the second
order freezing method. Transferring regularity and stability between these
two systems requires some care, and we will provide details of the proofs in Appendix~\ref{sec:A}. 
%The nonlinearity $f$ is assumed to be sufficiently smooth throughout this
%section and
%for concreteness we impose
%\begin{equation}\label{equ:3.0}%  f\in C^3(\R^m,\R^m).
%\end{equation}

%---------------------------------------------------------------------------------------------------------------------------------------------------
\subsection{Transformation to first order system and stability with asymptotic phase}
\label{subsec:3.1}
%---------------------------------------------------------------------------------------------------------------------------------------------------
In the following we impose the smoothness condition
\begin{assumption}\label{ass:3.0}
  The function $f:\R^{3m}\rightarrow\R^m$ satisfies $f\in C^3(\R^{3m},\R^m)$
\end{assumption}
and the following well-posedness condition
\begin{assumption}\label{ass:3.1}
  The matrix $M\in\R^{m,m}$ is invertible and $M^{-1}A$ is positive
  diagonalizable.
\end{assumption}
Assumption~\ref{ass:3.1} implies that there is a (not necessarily
unique) positive diagonalizable matrix $N\in\R^{m,m}$ satisfying
$N^2=M^{-1}A$. Let $\lambda_1\geqslant\cdots\geqslant\lambda_m>0$
denote the real positive eigenvalues of $N$.

We transform to a first order system by introducing  $U=(U_1,U_2,U_3)^{\top}\in\R^{3m}$  via
\begin{equation}
 \label{equ:2.21}
  U_1=u,\quad U_2=u_t + Nu_x,\quad U_3= u_t - N u_x+c u,
\end{equation}
where $c \in \R$ is an arbitrary constant to be determined later.
These variables transform \eqref{equ:2.1} into the first order system
\begin{equation}
  \label{equ:2.1sys}
  U_t = E U_x + F(U),
\end{equation}
with $E\in\R^{3m,3m}$ and $F:\R^{3m}\rightarrow\R^{3m}$ given by
\begin{equation}
  \label{equ:2.1N}
\begin{aligned}
  E =& \begin{pmatrix}N&0 &0\\ 0 &N & 0 \\ 0 & 0 & -N
\end{pmatrix},\quad
  F(U)=\begin{pmatrix} -cU_1 + U_3 \\ \tilde{f}(U)\\ \tilde{f}(U)+ cU_2
\end{pmatrix},\\
\tilde{f}(U):= &
M^{-1}f(U_1,\frac{1}{2}N^{-1}(U_2-U_3+cU_1),\frac{1}{2}(U_2+U_3-cU_1)).
\end{aligned}
\end{equation}
Thus we write the second-order Cauchy problem \eqref{equ:2.7} 
 as a first-order Cauchy problem for \eqref{equ:2.1sys},
\begin{equation}
  \label{equ:2.7sysA}
      U_t = E U_x + F(U),\quad   U(\cdot,0)=U_0:=(u_0, v_0+Nu_{0,x}, v_0 - N u_{0,x}+c u_0)^{\top}.
\end{equation}
\begin{remark} The transformation to a first order system has some 
 arbitrariness and does not influence the results for the second order problem
\eqref{equ:2.1}. The current transformation to a system of dimension $3m$ improves
an earlier version \cite{BeynOttenRottmannMatthes2016b} of our work
which was limited to the semilinear case \eqref{equ:2.1a}. There we
used $U_1=u,U_2=u_t - N u_x$ to obtain a system of minimal dimension $2m$.
But for this transformation the general nonlinear equation \eqref{equ:2.1} 
does not lead to a semilinear system of type \eqref{equ:2.1sys}. The 
drawback of the non-minimal dimension $3m$ is  that extra eigenvalues of the 
linearized system appear which do not correspond to those of the linearized second order system. The constant $c$ above 
will be used in Section~\ref{subsec:A2} to control these extra eigenvalues. 
\end{remark}
We emphasize that system~\eqref{equ:2.1sys} is diagonalizable
hyperbolic. More precisely, there is a nonsingular block-diagonal matrix $T\in
\R^{3m,3m}$, 
so that the change of variables $W=T^{-1}U$ transforms 
\eqref{equ:2.1sys} into diagonal hyperbolic form
\begin{equation}
  \label{equ:2.4N}
  W_t=\Lambda_E W_x+G(W),\quad \Lambda_{E}=T^{-1}ET=
\diag(\Lambda,\Lambda,-\Lambda),\quad G(W)=T^{-1}F(TW),
\end{equation}
where $\Lambda=\diag(\lambda_1,\ldots,\lambda_m)$.
For systems of type \eqref{equ:2.7sysA}, \eqref{equ:2.4N} we have local 
well-posedness of the Cauchy 
problem in suitable function spaces such as (see e.g.
\cite[Sect.~6]{Rauch12})
\begin{equation} \label{equ:2.4F}
  \mathcal{CH}^k(J;\R^n)=\bigcap_{j=0}^k
  C^{k-j}\bigl(J,H^j(\R,\R^n)\bigr),\quad J\subseteq\R
  \text{ interval},\;k\in\N_0,\;n\in\N.
\end{equation}
Our regularity condition on the traveling wave is as follows:
\begin{assumption} \label{ass:4.basic}
  The pair $(v_\star,\mu_\star)\in C^2_b(\R,\R^m)\times \R$  satisfies
  $v_{\star,\xi}\in H^3(\R,\R^m)$ and is a non-constant
  solution of the second order traveling wave equation
  \eqref{equ:2.4} with
  \begin{equation*}
    \lim_{\xi \rightarrow \pm \infty}v_{\star}(\xi)=v_{\pm},\quad 
 \lim_{\xi \rightarrow \pm \infty}v_{\star,\xi}(\xi)=0,\quad
    f(v_{\pm},0,0)=0.
  \end{equation*}
\end{assumption}
The first order system \eqref{equ:2.1sys} then has a traveling wave
\begin{equation} \label{equ:defvstar}
U_{\star}(x,t) = V_{\star}(x- \mu_{\star}t), \quad
V_\star:= \begin{pmatrix}v_\star \\(N-\mu_{\star} I_{m})v_{\star,\xi} \\
cv_{\star}-(N+\mu_{\star} I_{m})v_{\star,\xi}
\end{pmatrix} \in C^2_b(\R,\R^m)\times C^1_b(\R,\R^{2m}).
\end{equation}
The profile $V_{\star}$ solves the equation
\begin{equation}\label{equ:TWequ}
  0=(E+\mu_{\star}I_{3m}) V_{\star,\xi} + F(V_\star)
\end{equation}
and satisfies
\begin{equation} \label{equ:1systprop}
  \lim_{\xi\to\pm\infty}V_{\star}(\xi)=V_{\pm}:=(v_{\pm},0,cv_{\pm})\quad\text{and}\quad
  F(V_{\pm})=0.
\end{equation}
Our next assumption is
\begin{assumption}
\label{ass:4.reg} The matrix $A- \mu_{\star}^2 M$ is nonsingular.
\end{assumption}
It guarantees that \eqref{equ:2.4} is a regular second order system
and that $v_{\star}\in C_b^5(\R,\R^m)$ which follows from Assumptions~\ref{ass:3.0}
and~\ref{ass:4.basic}. Further, from $A- \mu_{\star}^2M=M(N-\mu_{\star}I_m)
(N+\mu_{\star}I_m)$ one infers that the matrix $E+\mu_{\star}I_{3m}$ in
\eqref{equ:TWequ} is nonsingular. This will enable us to apply the
stability results from \cite{RottmannMatthes2012a} which hold
for hyperbolic systems where the matrix $E+\mu_{\star}I_{3m}$
is real diagonalizable with nonzero but not necessarily distinct eigenvalues.
The condition also ensures that any solution $V_{\star}\in  C^1_b(\R,\R^{3m})$ of \eqref{equ:TWequ} has a first component in $C^2_b(\R,\R^m)$ which solves
the second order traveling wave equation~\eqref{equ:2.4}.
Moreover, using the limits from Assumption~\ref{ass:4.basic} one obtains
from \eqref{equ:2.4}
\begin{equation} \label{equ:limv2}
\lim_{\xi \rightarrow \pm \infty}v_{\star,\xi \xi}(\xi) = 0.
\end{equation}

Next, recall the dispersion set \eqref{equ:4.12} for the original second order
problem
\begin{equation}\label{equ:disp2nd}
  \sigma_{\mathrm{disp}}(\mathcal{P}) =
  \Bigl\{\lambda\in\C:
  \det\bigl(\lambda^2 A_2+\lambda
  A_1^{\pm}(\omega)+A_0^\pm(\omega)\bigr)=0\;
  \text{for some }\omega\in\R,\; \text{and some sign} \; \pm \Bigr\},
\end{equation}
with $A_0^\pm,A_1^{\pm},A_2$ given in \eqref{equ:4.11}. We require
\begin{assumption}\label{ass:dispstab}
  There is $\delta>0$ such that
  $\Re(\sigma_{\mathrm{disp}}(\mathcal{P})) < - \delta$.
 \end{assumption}
Finally, we  exclude nonzero eigenvalues in the right half plane.
\begin{assumption}\label{ass:evstab}
  The eigenvalue $0$ of $\mathcal{P}$ is simple and there
  is no other eigenvalue of $\mathcal{P}$ with real part greater than
  $-\delta$ with $\delta$ given by Assumption~\ref{ass:dispstab}.
\end{assumption}
With these assumptions our first main result reads:
\begin{theorem}[Stability with asymptotic phase]\label{thm:4.17}
  Let Assumptions~\ref{ass:3.0}~--~\ref{ass:evstab} hold.
  Then,
  for all $0<\eta<\delta$ there is $\rho>0$ such that for all $u_0\in
  v_\star+H^3(\R,\R^m)$, $v_0\in H^2(\R,\R^m)$ with
  \begin{equation}\label{equ:3.9}
    \|u_0-v_\star\|_{H^3}+\|v_0+\mu_\star
    v_{\star,\xi}\|_{H^2}\leq \rho, 
  \end{equation}
  the Cauchy problem~(\ref{equ:2.7}) has a unique
  global solution
  $u\in v_\star+\mathcal{CH}^2([0,\infty);\R^m)$.
  %C^2([0,\infty);v_\star+L^2)\cap
  %C^1([0,\infty);v_\star+H^1)\cap C^0([0,\infty); v_\star+H^2)\]
  Moreover, there exist $\varphi_\infty=\varphi_\infty(u_0,v_0)$ and
  $C=C(\eta,\rho)$ satisfying
  \begin{equation}\label{equ:3.10a}
    |\varphi_\infty|\leq C\Bigl(\|u_0-v_\star\|_{H^3}+\|v_0+\mu_\star
    v_{\star,\xi}\|_{H^2}\Bigr) 
  \end{equation}
  and
  \begin{equation}
  \label{equ:3.10b}
  \begin{aligned}
    \|u(\cdot,t)-v_\star(\cdot-\mu_\star t-\varphi_\infty)\|_{H^2} + &
    \|u_t(\cdot,t)+\mu_\star v_{\star,\xi}(\cdot-\mu_\star
    t-\varphi_\infty)\|_{H^1}\\
     \leq & C \Bigl(\|u_0-v_\star\|_{H^3}+\|v_0+\mu_\star
    v_{\star,\xi}\|_{H^2}\Bigr) e^{-\eta t}\quad\forall t\geq 0.
  \end{aligned}
  \end{equation}
\end{theorem}
The proof will be given in Appendix~\ref{sec:A}.
Let us note that the loss of one derivative for the solution when
compared to initial data, is typical for hyperbolic stability theorems
and results from the theory in \cite{RottmannMatthes2012a}.

\subsection{Stability of the freezing method}
Let us first apply the freezing method to the first order system
\eqref{equ:2.7sysA}. We introduce new unknowns $\gamma(t)\in\R$
and $V(\xi,t)\in\R^{3m}$ via the ansatz
\begin{equation}
  \begin{aligned}
  \label{equ:2.8Sys}
    U(x,t) & = V(\xi,t),\quad\xi:=x-\gamma(t),\quad x\in\R,\,t\geq 0.  \end{aligned}
\end{equation} 
This formally leads to
% \eqref{equ:2.8Sys} into \eqref{equ:2.7sysA} yields
\begin{subequations}\label{equ:3.1}
  \begin{align}
    V_t &= (E+\mu I_{3m})V_\xi + F(V), \label{equ:3.1a} \\
    \gamma_t &=\mu, \label{equ:3.1b} \\
    V(\cdot,0) &= V_0
    := U_0=(u_0,v_0+Nu_{0,\xi},v_0-N u_{0,\xi}+cu_0)^\top,
    %\begin{pmatrix}u_0\\v_0-Nu_{0,\xi}\end{pmatrix},
    \quad 
    \gamma(0)=0,\label{equ:3.1c}
  \end{align}
\end{subequations}
with $E$ and $F$ from \eqref{equ:2.1N}.  In \eqref{equ:3.1} we
introduced the time-dependent function $\mu(t)\in\R$ for convenience.
As before, equation~\eqref{equ:3.1b} decouples and can be solved in a
postprocessing step.  One needs an additional algebraic constraint to
compensate the extra variable $\mu$. To relate the second order freezing
equation~\eqref{equ:2.12} and the first order version \eqref{equ:3.1},
we omit the introduction of $\mu_2$ in \eqref{equ:2.12} and write it
in the form
\begin{subequations}\label{equ:3.2}
  \begin{align}
   Mv_{tt} = & (A-\mu^2 M)v_{\xi\xi} + 2\mu Mv_{\xi t} +
    \mu_t M v_{\xi} + f(v,v_{\xi},v_t-\mu v_{\xi}), &&\xi\in\R,\,t\geqslant 0,
   \label{equ:3.2a}\\
    \gamma_t &= \mu,\label{equ:3.2b}\\
    v(\cdot,0)&=u_0,\quad v_t(\cdot,0)=v_0+\mu(0)u_{0,\xi},\quad
    \gamma(0)=0. \label{equ:3.2c}
  \end{align}
\end{subequations}
Transforming \eqref{equ:3.2} into a first order system by introducing
$V=(V_1,V_2,V_3)^{\top}\in\R^{3m}$ via
\begin{equation} 
 \label{equ:3trans}
  V_1=v,\quad V_2=v_t+(N- \mu I_m)v_\xi,\quad  V_3=v_t-(N+\mu I_m)v_{\xi}+cv
\end{equation}
we again find the system \eqref{equ:3.1}.
As a consequence we obtain the equivalence of the 
freezing systems for the first and the second order formulation.
 Henceforth
we restrict to the fixed phase condition~(\ref{equ:2.15}) for which we
require the following condition.
\begin{assumption}\label{ass:3.template}
  The template function $\hat{v}:\R\to\R^m$ belongs to
  $v_\star+H^1(\R,\R^m)$ and satisfies
  \begin{subequations}\label{equ:3.16}
    \begin{align}
      \langle \hat{v}-v_\star,\hat{v}_\xi\rangle_{L^2}&=0,
      \label{equ:3.16a}\\
      \langle v_{\star,\xi}, \hat{v}_\xi\rangle_{L^2}&\neq 0.
      \label{equ:3.16b}
    \end{align}
  \end{subequations}
\end{assumption}
Condition \eqref{equ:3.16a} implies that \eqref{equ:2.14a} holds for the fixed phase
condition \eqref{equ:2.15}, so that
$(v_\star,\mu_\star,0)$ is a stationary solution of
\eqref{equ:2.19a}, \eqref{equ:2.19b} (skipping the $\gamma$-equation 
needed for reconstruction only). Condition \eqref{equ:3.16b} specifies some non-degeneracy 
used in the proof.

Now we are ready to state asymptotic stability (in the sense of
Lyapunov) of the steady state $(v_\star,\mu_\star,0)$ for the freezing
system~(\ref{equ:2.19}) that belongs to the nonlinear wave equation.
\begin{theorem}[Stability of the freezing method]\label{thm:freezestab}
 Let Assumptions~\ref{ass:3.0}~--~\ref{ass:evstab} hold
 %\ref{ass:3.1}, \ref{ass:4.basic},\ref{ass:4.reg}, \ref{ass:dispstab}, 
   and consider the phase condition $\psi^\mathrm{2nd}(v,v_t,\mu_1,\mu_2)=
  \langle v-\hat{v},  \hat{v}_\xi\rangle_{L^2}$ with a template function $\hat{v}$ which fulfills the non-degeneracy Assumption~\ref{ass:3.template}. 
  Then,
  for all $0<\eta<\delta$ there is $\rho>0$ such that for all $u_0\in
  v_\star+H^3(\R,\R^m)$, $v_0\in H^2(\R,\R^m)$  and $\mu_1^0 \in \R$ which
satisfy
  \begin{equation}\label{equ:3.17}
    \|u_0-v_\star\|_{H^3}+\|v_0+\mu_\star v_{\star,\xi}\|_{H^2}\leq \rho
  \end{equation}
and the consistency conditions \eqref{equ:2.13consist}, \eqref{equ:2.14consist}, $\langle u_0 - \hat{v}, \hat{v}_{\xi}\rangle_{L^2}=0$
the following holds.  
  The freezing system \eqref{equ:2.19}  
  has a unique global solution
  $(v,\mu_1,\mu_2,\gamma)\in (v_\star+\mathcal{CH}^2([0,\infty);\R^m)) \times
  C^1([0,\infty))\times C([0,\infty))\times C^2([0,\infty))$. 
  Moreover, there exists some $C=C(\rho,\eta)>0$ such that the following exponential stability estimate holds
  \begin{equation}\label{equ:3.18}
    \|v(\cdot,t)-v_\star\|_{H^2}+\|v_t(\cdot,t)\|_{H^1}+|\mu_1(t)-\mu_\star|
    \leq C\bigl( \|u_0-v_\star\|_{H^3}+\|v_0+\mu_\star
    v_{\star,\xi}\|_{H^2}\bigr)\;e^{-\eta t}\quad \forall t\geq 0.
  \end{equation}
  %Then, for every $0<\eta<\delta$ there exist constants $\rho>0$ and 
  %$C=C(\rho,\eta)>0$,
  %such that for all $u_0\in v_\star+H^3(\R,\R^m)$,
  %$v_0\in H^2(\R,\R^m)$ with
  %\begin{equation}\label{equ:3.17}
  %  \|u_0-v_\star\|_{H^3}^2+\|v_0+\mu_\star v_{\star,x}\|_{H^2}^2\leq \rho^2
  %\end{equation}
  %there exists a unique global (classical) solution $(v,\mu_1,\mu_2,\gamma)$
  %of \eqref{equ:2.19}, with phase condition \newline $\psi^\mathrm{2nd}(v,v_t,\mu_1,\mu_2)=
  %\langle v-\hat{v},  \hat{v}_\xi\rangle_{L^2}$.
  %% The solution is of class
  %% $v\in v_\star+\mathcal{CH}^2([0,\infty);\R^m)$ and $\mu\in
  %% C^1([0,\infty),\R)$, and \eqref{equ:2.19} holds in the classical
  %% sense.  
  %The solution satisfies the exponential stability estimate
  %\begin{equation}\label{equ:3.18}
  %  \|v(t)-v_\star\|_{H^2}+\|v_t(t)\|_{H^1}+|\mu_1(t)-\mu_\star|
  %  \leq C\bigl( \|u_0-v_\star\|_{H^3}+\|v_0+\mu_\star
  %  v_{\star,\xi}\|_{H^2}\bigr)\;e^{-\eta t}\quad \forall t\geq 0.
  %\end{equation}
\end{theorem}

The proof builds on the fact that the
original second order version \eqref{equ:2.19} and the first order
version \eqref{equ:3.1} of the freezing method for traveling
waves in
% the damped wave equation
\eqref{equ:2.1} are equivalent in suitable function spaces.
This will be detailed in Appendix~\ref{sec:A}

\appendix
\renewcommand{\theequation}{\Alph{section}.\arabic{equation}}

%---------------------------------------------------------------------------------------------------------------------------------------------------
%
%  SECTION 4: (Transformation to first order systems)
%
%---------------------------------------------------------------------------------------------------------------------------------------------------
\sect{Proof of Stability Theorems}
\label{sec:A}
%---------------------------------------------------------------------------------------------------------------------------------------------------

In this Appendix we provide a detailed proof of Theorems \ref{thm:4.17} and
\ref{thm:freezestab}.

%---------------------------------------------------------------------------------------------------------------------------------------------------
\subsection{Results for first order systems}
\label{subsec:A1}
%---------------------------------------------------------------------------------------------------------------------------------------------------

Let us recall the stability result from \cite[Thm.2.5]{RottmannMatthes2012a}
for first order systems of the general type
\begin{subequations} 
  \label{equ:diaghyp}
  \begin{align}
    & W_t= \Lambda_E W_x + G(W),
    \quad\,x\in\R,\,t\ge 0, W(x,t) \in \R^l \label{equ:hypdgl} \\
    & W(\cdot,0) = W_0. \label{equ:hypini}
  \end{align}
\end{subequations}
The assumptions are
\begin{enumerate}[label=\textup{(\roman*)},leftmargin=0.65cm]
  \setlength{\itemsep}{0cm}
  \item \label{itm:3.a} The matrix $\Lambda_E\in \R^{l,l}$ is diagonal.
\item \label{itm:3.b} The nonlinearity $G$ belongs to
  $C^3(\R^{l},\R^{l})$.
\item \label{itm:3.b1} There exists a traveling wave solution
  $W(x,t)=W_{\star}(x-\mu_{\star}t) $ of \eqref{equ:diaghyp} 
  such that $W_\star\in C^1_b(\R,\R^l)$, $W_{\star,\xi}\in H^2(\R,\R^l)$.
\item \label{itm:3.d} The matrix function
  $Y(\xi)=DG\bigl(W_\star(\xi)\bigr)$  satisfies
  $\lim_{\xi\to\pm\infty}Y(\xi)=Y_{\pm}$ and
  $\lim_{\xi\to\pm\infty}Y'(\xi)=0$.
\item \label{itm:3.c} The matrix $\Lambda_E+ \mu_\star I_l\in\R^{l,l}$ is
  nonsingular.
\item \label{itm:3.e} There is $\delta>0$ such that 
$\Re \{ s\in \C:s\in \sigma\bigl(i\omega (\Lambda_E+\mu_\star I_l)+Y_{\pm}\bigr)
 \;\; \text{for some}\; \; \omega\in\R\}\leq -\delta$.
\item \label{itm:3.f} The operator  $\mathcal{Y}_{\mathrm{1st}}=(\Lambda_E+\mu_{\star}
  I_l)\partial_\xi+Y(\cdot):H^1(\R,\R^l)\rightarrow L^2(\R,\R^l)$
   has the algebraically simple eigenvalue $0$ and 
  satisfies $\sigma_{\mathrm{point}}(\mathcal{Y}_{\mathrm{1st}})\cap \{\Re
  s>-\delta\} = \{0\}$.
\end{enumerate}

 Then for every $0<\eta<\delta$ there is $\rho_0>0$ so that for all
  $W_0\in W_\star+H^2(\R,\R^{l})$
    with $\|W_0-W_{\star}\|_{H^2}\leq\rho_0$ the Cauchy
  problem~(\ref{equ:diaghyp})  has a unique global
  solution $W\in W_{\star}+\mathcal{CH}^1([0,\infty);\R^{l})$.
    Moreover, there is $\varphi_\infty=\varphi_\infty(W_0)\in \R$ and
  $C=C(\eta,\rho_0)>0$
  %}$C=C(\eta,\|V_{\star,\xi}\|_{H^2})>0$
  such that
  \begin{equation} \label{equ:stabphase}
      |\varphi_\infty|\leq C\|W_0-W_{\star}\|_{H^2},
  \end{equation}
 \begin{equation} \label{equ:stabfunc}
    \|W(\cdot,t)-W_{\star}(\cdot-\mu_\star t-\varphi_\infty)\|_{H^1}
    \leq C\|W_0-W_{\star}\|_{H^2} e^{-\eta t}\quad\forall t\geq 0.
  \end{equation}
 In \cite[Thm.2.5]{RottmannMatthes2012a} the eigenvalues of $\Lambda_E$ are assumed to be in decreasing order. However, this was done  for convenience 
of the proof only, and the result holds verbatim without this ordering. 
Our goal is to apply the stability result to the system \eqref{equ:2.4N} where 
 $\Lambda_E$ is diagonal but  the eigenvalues are not ordered. In the following 
we show the assumptions \ref{itm:3.b}-\ref{itm:3.f}  for the system \eqref{equ:2.4N}. Our first observation is that instead of checking
assumptions \ref{itm:3.b}-\ref{itm:3.f} for the transformed data
$W_{\star}=T^{-1}V_{\star}$, $\Lambda_E=T^{-1}ET$ and $G=T^{-1}F T$, it is
sufficient to check them  for the data $V_{\star}$, $E$ and $F$ 
of the original system \eqref{equ:2.1sys}.

Condition \ref{itm:3.b} follows from Assumption \ref{ass:3.0}. Moreover,
conditon \ref{itm:3.b1} is a consequence of \eqref{equ:defvstar}
and Assumption \ref{ass:4.basic}. From \eqref{equ:2.1N} we obtain
(recall $(\star)=(v_{\star},v_{\star,\xi},-\mu_{\star}v_{\star,\xi})$)
\begin{equation} \label{equ:Zlin}
Z=DF(V_{\star}) = \begin{pmatrix}
    -cI_m & 0 & I_m \\ \Phi_1 & \Phi_2 & \Phi_3 \\ \Phi_1 & \Phi_2+cI_m & \Phi_3
    \end{pmatrix},
    \quad \begin{pmatrix} \Phi_1 \\ \Phi_2 \\ \Phi_3 \end{pmatrix}
    := \begin{pmatrix} M^{-1} D_1f(\star) - c \Phi_3 \\
    \frac{1}{2}M^{-1}(D_2f(\star)N^{-1}+D_3f(\star)) \\
    \frac{1}{2}M^{-1}(-D_2f(\star)N^{-1}+ D_3f(\star))
    \end{pmatrix}.
\end{equation}
By Assumption \ref{ass:4.basic} the limit is given by (recall $(\pm)=
(v_{\pm},0,0)$)
\begin{equation} \label{equ:Zlim}
Z_{\pm}=\lim_{\xi \rightarrow \pm \infty}Z(\xi)= \begin{pmatrix}
 -cI_m & 0 & I_m \\ \Phi^{\pm}_1 & \Phi^{\pm}_2 & \Phi^{\pm}_3 \\ \Phi^{\pm}_1 & \Phi^{\pm}_2+cI_m & \Phi^{\pm}_3
    \end{pmatrix},
    \quad \begin{pmatrix} \Phi^{\pm}_1 \\ \Phi^{\pm}_2 \\ \Phi^{\pm}_3 \end{pmatrix}
    := \begin{pmatrix} M^{-1} D_1f(\pm) -c \Phi_3^{\pm} \\
    \frac{1}{2}M^{-1}(D_2f(\pm)N^{-1}+D_3f(\pm)) \\
    \frac{1}{2}M^{-1}(-D_2f(\pm)N^{-1}+ D_3f(\pm))
    \end{pmatrix}.
    \end{equation}
Differentiating \eqref{equ:Zlin} w.r.t. $\xi$ and using Assumption \ref{ass:4.basic}
as well as \eqref{equ:limv2} then shows $Z'(\xi) \rightarrow 0$ as
$\xi \rightarrow \pm \infty$.
Further, condition \ref{itm:3.c} follows from Assumption \ref{ass:4.reg} as has
been noted in Section \ref{sec:3}. The conditions \ref{itm:3.e} and \ref{itm:3.f} are discussed in the next subsection.

\subsection{Spectral relations of first and second order problems}
\label{subsec:A2}
%---------------------------------------------------------------------------------------------------------------------------------------------------
We transfer 
the spectral properties of the original second order
problem~(\ref{equ:2.1}) to the first order
problem~(\ref{equ:2.1sys}) and vice versa.
Throughout this section we impose  
 Assumptions \ref{ass:3.0}, \ref{ass:3.1}, 
\ref{ass:4.basic}  and define $V_{\star}$ by \eqref{equ:defvstar}.

By Definition~\ref{def:3.1}, the spectral problem for the second order
problem~(\ref{equ:2.1}), considered in a co-moving frame, is given by
the solvability properties of
\begin{equation*}
% \label{equ:3.3a}
  \PL(\lambda):H^2(\R,\C^m)\to L^2(\R,\C^m),\quad
  \text{defined by \eqref{equ:4.1}}.
\end{equation*}
The analog for the first order formulation \eqref{equ:2.1sys} is the first order
differential operator
\begin{equation}
  \label{equ:3.3}
  \begin{aligned}
    \PL_{\mathrm{1st}}(\lambda)&:H^1(\R,\C^{3m}) \to L^2(\R,\C^{3m})\;\text{ given
      by}\\
    \PL_{\mathrm{1st}}(\lambda)&= \lambda I_{3m} 
       -\mathcal{Z}_{\mathrm{1st}},\quad 
\mathcal{Z}_{\mathrm{1st}}=(E + \mu_{\star}I_{3m})\partial_{\xi} + Z(\cdot),
     \end{aligned}
\end{equation}
obtained  by linearizing \eqref{equ:2.1sys} in the
co-moving frame about the traveling wave $V_\star$.
Introducing the first order operators
\begin{equation} \label{equ:abbrdiff}
\PL_{-N}(\lambda) = \lambda - (N+\mu_{\star}I_m) \partial_{\xi},
\quad
\PL_{+N}(\lambda) = \lambda + (N-\mu_{\star}I_m) \partial_{\xi},
\end{equation}
we may write $\PL_{\mathrm{1st}}(\lambda)$ as a block operator
\begin{equation} \label{equ:P1block}
\PL_{\mathrm{1st}}(\lambda)=
\begin{pmatrix} \PL_{-N}(\lambda) + c I_m & 0 & - I_m \\
-\Phi_1 & \PL_{-N}(\lambda)- \Phi_2 & - \Phi_3 \\
- \Phi_1 & - \Phi_2 -c I_m & \PL_{+N}(\lambda) - \Phi_3
\end{pmatrix}.
\end{equation}
Finally, it is convenient to introduce the normalized operator polynomial
\begin{align*}
\tilde{\PL}(\lambda) = M^{-1} \PL(\lambda), \quad \lambda \in \C,
\end{align*}
which has exactly the same spectrum as $\PL(\lambda)$.
The key to the relation of spectra is the following factorization
\begin{equation} \label{equ:factorization}
\begin{pmatrix} 0 & 0 & I_m \\ 0 & I_m & -I_m \\
I_m & 0 & 0
\end{pmatrix}   \PL_{\mathrm{1st}}(\lambda)=
\begin{pmatrix}
\tilde{\PL}(\lambda) & -\Phi_2-cI_m &  \PL_{+N}(\lambda)- \Phi_3 \\
0 & \PL_{-N}(\lambda)+c I_m & -\PL_{+N}(\lambda) \\
0 & 0 & - I_m
\end{pmatrix}
\begin{pmatrix}
I_m & 0 & 0 \\
- \PL_{+N}(\lambda) & I_m & 0 \\
-\PL_{-N}(\lambda)-cI_m & 0 & I_m
\end{pmatrix}.
\end{equation}
This follows from  \eqref{equ:Zlin} and \eqref{equ:P1block} by 
a straightforward but somewhat lengthy calculation.
The factorization \eqref{equ:factorization}
is motivated by the equivalence notion for matrix polynomials
(see e.g. \cite[Chapter S1.6]{GohbergLancasterRodman2009}).

Let us recall a well-known result on Fredholm properties for first
order operators from Palmer \cite{Palmer:1984}:

\begin{proposition} \label{propA1}
Consider a first order system
\begin{equation} \label{equ:syst1}
(\partial_{\xi}- Q(\xi))V= R \in L^2(\R,\C^N),
\end{equation}
where the matrix-valued function $Q: \R \rightarrow \C^{N,N}$ is continuous
and has limits
\begin{equation} \label{equ:limQ}
Q_{\pm}= \lim_{\xi \rightarrow \pm \infty}Q(\xi).
\end{equation}
Further assume that $Q_{\pm}$ have no eigenvalues on the imaginary axis. Then the operator
\begin{equation*}
\mathcal{Q}=\partial_{\xi}-Q(\cdot): H^1(\R,\C^N) \rightarrow L^2(\R,\C^N)
\end{equation*}
is Fredholm of index $\dim E_+^s- \dim E_-^s$, where $E_{\pm}^s\subseteq \C^N$
is the stable subspace of $Q_{\pm}$ (i.e. the maximal invariant
subspace associated with eigenvalues of negative real part). 
\end{proposition} 
A consequence of this result for parametrized systems is the following

\begin{proposition} \label{propA2}
Consider a first order system
\begin{equation} \label{equ:syst2}
\mathcal{Q}(\lambda)V= (\partial_{\xi}- Q(\xi,\lambda))V= R \in L^2(\R,\C^l),
\end{equation}
with a matrix polynomial $Q(\xi,\lambda)=\sum_{j=0}^q Q_j(\xi) \lambda^j$,
$Q_j \in C(\R,\C^{l,l})$. Assume that  the limits
$\lim_{\xi \rightarrow \pm \infty}Q_j(\xi) = Q_j^{\pm}$ exist and let
$Q^{\pm}(\lambda)=\sum_{j=0}^q Q_j^{\pm} \lambda^j$.
Then the dispersion set
\begin{equation} \label{equ:Adisp}
\sigma_{\mathrm{disp}}(\mathcal{Q}) = \{ \lambda \in \C: \det(i \omega I -
Q^{\pm}(\lambda))=0\; \text{for some}\; \omega \in \R
\; \text{and some sign} \; \pm \}
\end{equation}
is contained in the essential spectrum $\sigma_{\mathrm{ess}}(\mathcal{Q})$.
For $\lambda \notin \sigma_{\mathrm{disp}}(\mathcal{Q})$, the operator
$\mathcal{Q}(\lambda):H^1(\R,\C^l)\rightarrow L^2(\R,\C^l)$ is Fredholm
of index $\dim E_+^s(\lambda)-\dim E_-^{s}(\lambda)$ where $E_{\pm}^s(\lambda)$
denotes the stable subspace of $Q^{\pm}(\lambda)$.
\end{proposition}
This result may be found in \cite[Theorem 3.1.13]{KapitulaPromislov2013}
(note that the dispersion set is called the Fredholm border there).

If we replace $\partial_{\xi}$ by $i \omega$ and let $\xi \rightarrow \pm \infty$ in \eqref{equ:factorization}
then the left and right factors in \eqref{equ:factorization} are 
$\lambda$-dependent matrices with a constant determinant (see the
equivalence notion of matrix polynomials in
\cite[Chapter S1.6]{GohbergLancasterRodman2009}).
Hence the dispersion set  of the first order operator
$\PL_{\mathrm{1st}}(\lambda)$ is completely
determined by the dispersion set \eqref{equ:4.11a} of the second order
operator $\tilde{\PL}(\lambda)$
and the first order operator $\PL_{-N}(\lambda)+cI_m$.
Since $N+\mu_{\star}I_m$ has nonzero real eigenvalues
$\lambda_j+\mu_{\star},j=1,\ldots,m$ by \eqref{equ:2.4N} we find from
Propositions \ref{propA1} and \ref{propA2}  
\begin{align*}
\sigma(\PL_{-N}+cI_m)= \sigma_{\mathrm{disp}}(\PL_{-N}+cI_m)=\{
c+(\lambda_j+\mu_{\star})i \omega: \omega \in \R, j=1,\ldots,m\}= c + i \R.
\end{align*}
This yields the following result.
\begin{proposition}\label{prop:3.9}
 The dispersion sets satisfy
 \begin{equation} \label{equ:disprelate}
  \sigma_{\mathrm{disp}}(\PL_\mathrm{1st})= 
  \sigma_{\mathrm{disp}}(\PL) \cup (c + i \R).
  \end{equation}
 \end{proposition}
 
This proposition leads to a proper choice of the shift parameter $c$.
Taking $c < -\delta$ , condition \ref{itm:3.e} immediately follows
from Assumption \ref{ass:dispstab}.
The following proposition relates the point spectra of the second order 
operator $\PL$ and the first order operator $\PL_{\mathrm{1st}}$ to each
other.
\begin{proposition}\label{prop:specneu}
 The following assertions hold:
  \begin{enumerate}[label=\textup{(\alph*)},leftmargin=0.85cm]
\item \label{itm:thmA:1}
There exists a $\lambda_{\star}>c$ such that
$\sigma_{\mathrm{disp}}(\PL_{\mathrm{1st}}) \cap [\lambda_{\star}, \infty)=
\emptyset$.
  \item \label{itm:thmA:ii} Let $\rho_+$ be the connected component
of $\{\lambda \in \C:\Re \lambda > c\} \setminus
 \sigma_{\mathrm{disp}}(\PL_{\mathrm{1st}})$ containing $[\lambda_{\star},\infty)$.
  Then the operator $\PL_\mathrm{1st}(\lambda):H^1(\R,\C^{2m})\to L^2(\R,\C^{2m})$ is  Fredholm of index $0$ for all $\lambda \in \rho_+$.
  \item \label{itm:thmA:iii}
  The point spectra of $\PL_{\mathrm{1st}}$  and $\PL(\lambda):H^2(\R,\C^m)\to L^2(\R,\C^m)$ in $\rho_+$ coincide, i.e.
  \begin{equation} \label{equA:point}
  \sigma_{\mathrm{point}}(\PL)\cap \rho_+=
  \sigma_{\mathrm{point}}(\PL_{\mathrm{1st}})\cap \rho_+.
  \end{equation}
  Eigenvalues in these sets have the same geometric and maximum partial multiplicity.
\end{enumerate}
\end{proposition}
Let us first note that this proposition implies  condition \ref{itm:3.f}. 
For the choice $c < -\delta$ the set  $\rho_+$ contains 
$\{\Re \lambda >-\delta\}$ by Assumption \ref{ass:dispstab} and
Proposition \ref{prop:3.9}. 
Condition \ref{itm:3.f}  is then a consequence of Assumption
\ref{ass:evstab} and assertion \ref{itm:thmA:iii} of Proposition
\ref{prop:specneu}.

\begin{proof}
  Using Assumption \ref{ass:4.reg} we can rewrite the operator from  \eqref{equ:3.3} as follows
\begin{align*}
\PL_{\mathrm{1st}}(\lambda) = - (E+ \mu_{\star}I_{3m})(\partial_\xi
-  (E+ \mu_{\star}I_{3m})^{-1}(\lambda I_{3m} - DF(V_{\star}))).
\end{align*}
 The matrix $(E+ \mu_{\star}I_{3m})^{-1}$ is hyperbolic by Assumption \ref{ass:4.reg} and this property persists for the matrix
$(E+ \mu_{\star}I_{3m})^{-1}(\lambda I_{3m} - DF(V_{\pm}))$ for $\lambda \ge \lambda_{\star}$
sufficiently large, independently of the sign $\pm$ and with the same number of stable and unstable eigenvalues. Therefore $\PL_{\mathrm{1st}}(\lambda)$
is Fredholm of index $0$ by Proposition \ref{propA2} for
$\lambda \in [\lambda_{\star},\infty)$. Since the Fredholm
index is continuous in $\rho_+$ and can only change at $\sigma_{\mathrm{disp}}(\PL_{\mathrm{1st}})$ or at  $c+i\R$,
assertion \ref{itm:thmA:ii} also follows.

Consider an eigenvalue $\lambda_0 \in \sigma_{\mathrm{point}}(\PL_{\mathrm{1st}})\cap \rho_+$ with eigenfunction
$V=(V_1,V_2,V_3)^{\top} \in H^1(\R,\C^{3m}), V \ne 0$. The first block equation
reads $(\PL_{-N}(\lambda_0)+ c I_m) V_1 = V_3 \in H^1$ from which we infer
$V_1 \in H^2(\R,\C^m)$. In the following let us write the factorization \eqref{equ:factorization}
in the short form
\begin{align} \label{equ:factorshort}
T_1 \PL_{\mathrm{1st}}(\lambda)= R(\lambda) T_2(\lambda)
\end{align}
and apply it to $V$. Then $W(\lambda_0):=T_2(\lambda_0)V$ satisfies
$R(\lambda_0)W(\lambda_0)=0$, and from the triangular structure of $R$
and the invertibility of $\PL_{-N}(\lambda_0)+cI_m$ we obtain
$W_3=0,W_2=0$ as well as $\tilde{\PL}(\lambda_0)V_1=\tilde{\PL}(\lambda_0)W_1
=0$. If $V_1=0$ then $V_2=0,V_3=0$ follows from $W_2=0,W_3=0$, hence 
$V_1 \ne0$. In a similar manner, if $\tilde{\PL}(\lambda_0)W_1=0$ for
some $W_1 \in H^2(\R,\C^m), W_1 \ne 0$ then $\PL_{\mathrm{1st}}(\lambda_0)V=0$ and $V \ne 0$ for 
$V=T_2(\lambda_0)^{-1}\begin{pmatrix} W_1 & 0 & 0 \end{pmatrix}^{\top}$. By the same argument the null spaces
$\mathcal{N}(\PL_{\mathrm{1st}}(\lambda_0))$ and 
$\mathcal{N}(\tilde{\PL}(\lambda_0)$) have 
equal dimension.

 Finally, consider a root polynomial $V(\lambda)=
\sum_{j=0}^n V_{[j]}(\lambda -\lambda_0)^j$ with $V_{[j]}\in H^1(\R,\C^{3m})$
satisfying
\begin{align*}
V(\lambda_0)=V_{[0]} \ne 0, \quad 
(\PL_{\mathrm{1st}}V)^{(\nu)}(\lambda_0) =0, \nu=0,\ldots,n-1.
\end{align*}
As above we find $V_{[0],1} \in H^2(\R,\C^m)$, $V_{[0],1}\ne 0$ and then by induction
$V_{[j],1} \in H^2(\R,\C^m), j=1,\ldots,n$ from the equations
\begin{align*}
\nu! \ \PL_{\mathrm{1st}}(\lambda_0)  V_{[\nu]}=
- \sum_{\ell =1}^{\nu} {\nu \choose \ell}
\PL_{\mathrm{1st}}^{(\ell)}(\lambda_0) 
V^{(\nu-\ell)}(\lambda_0).
\end{align*}
Note that the right-hand side is in $H^1(\R,\C^{3m})$ since the $\lambda$-derivative
of $\PL_{\mathrm{1st}}$ is $I_{3m}$. Setting $W(\lambda)=T_2(\lambda)V(\lambda)$ then leads
via  \eqref{equ:factorshort} to
\begin{align*}
(R W)^{(\nu)}(\lambda_0)=0, \nu=0,\ldots,n-1.
\end{align*} 
Working backwards through the components of this  equation gives
$W^{(\nu)}_k(\lambda_0)=0,\nu=0,\ldots, n-1$ for $k=3,2$, and therefore,
\begin{align*}
0= (\tilde{\PL} W_1)^{(\nu)}(\lambda_0), \nu=0,\ldots,n-1,
\end{align*}
with $W_1(\lambda_0)=V_1(\lambda_0) \ne 0$.

Conversely, let $W_1(\lambda)= \sum_{j=0}^{n-1} (\lambda -\lambda_0)^j W_{[j],1}$
be a root polynomial of $\tilde{\PL}$ in $H^2(\R,\C^m)$ with $W_{[0],1}\ne 0$. Then
we set $W(\lambda)= \begin{pmatrix} W_1(\lambda) & 0 & 0 
\end{pmatrix}^{\top}$ and find that
\begin{align*}
V(\lambda)= T_2(\lambda)^{-1}W(\lambda) = \begin{pmatrix}
W_1(\lambda) & \PL_{+N}(\lambda)W_1(\lambda) & (-\PL_{-N}(\lambda)+cI_m)
W_1(\lambda) \end{pmatrix}^{\top} \in H^1(\R,\C^{3m})
\end{align*}
satisfies $V(\lambda_0)\ne 0$ and 
\begin{align*}
T_1 (\PL_{\mathrm{1st}}V)^{(\nu)}(\lambda_0)=
(R W)^{(\nu)}(\lambda_0)=0 , \nu=0,\ldots,n-1.
\end{align*}
\end{proof}

\subsection{Stability for the second order system}
\label{subsec:A3}

In the following we consider the Cauchy problem \eqref{equ:2.7sysA}
and recall the function spaces \eqref{equ:2.4F}. We need two auxiliary results.
The first one is regularity of solutions with respect to source terms
taken from the theory of linear
first order systems (see \cite[Cor.2.2.2]{Rauch12}).
\begin{lemma} \label{LemmaA1}
Consider a first order system
\begin{equation} \label{equA:firstorder}
u_t = A_1 u_x + B_1 u + r, \quad u(x,0)=u_0(x), \quad x\in \R, t\ge 0,
\end{equation}
where $A_1 \in \R^{l,l}$ is real diagonalizable and $B_1\in \R^{l,l}$.
If $u_0\in H^k(\R,\R^l)$ for some $k \ge 1$ and $r \in \mathcal{CH}^{k-1}([0,\infty);\R^l)$ then the system \eqref{equA:firstorder}
has a unique solution in $u\in \mathcal{CH}^k([0,\infty);\R^l)$.
\end{lemma}
The second one concerns commuting weak and strong  derivatives
with respect to space and time.
\begin{lemma} \label{lemA:commutederiv}
For $u \in  C^1([0,\infty);H^1(\R,\R^l)))$ 
let $ \frac{d}{dt}u \in C^{0}([0,\infty);H^1(\R,\R^l))) $ be its time derivative 
and let $\frac{\partial}{\partial x}u(\cdot,t)$ be its
weak space derivative pointwise in $t \in [0,\infty)$.
Then $ \frac{\partial}{\partial x}u \in C^1([0,\infty);L^2(\R,\R^l))$ and
its time derivative agrees with the weak spatial derivative of $\frac{d}{dt}u$
evaluated pointwise in $t \in [0,\infty)$, i.e.
\begin{equation}\label{equA:weakstrong}
\frac{d}{dt}(\frac{\partial}{\partial x} u)  = \frac{\partial}{ \partial x}
(\frac{d}{dt}u).
\end{equation}  
\end{lemma}
\begin{proof}
Let $t,t+h \in [0,\infty)$ with $h \neq 0$ and note that
\begin{align*}
\| \frac{1}{h}(\frac{\partial}{\partial x}u(\cdot,t+h) - \frac{\partial}{\partial x} u(\cdot,t))
- \frac{\partial}{\partial x} (\frac{d}{dt}u(\cdot,t))\|_{L^2} \le 
\| \frac{1}{h}(u(\cdot,t+h)-u(\cdot,t)) - \frac{d}{dt}u(\cdot,t) \|_{H^1},
\end{align*}
 where the right-hand side converges to zero as $h \rightarrow 0$ by assumption.
Therefore, the derivative $\frac{d}{dt}(\frac{\partial}{\partial x} u)$ 
exists in $L^2(\R,\R^l)$ for all $t\in [0,\infty)$ and
coincides with $\frac{\partial}{ \partial x}
(\frac{d}{dt}u) \in C^0([0,\infty);L^2(\R,\R^l))$.
\end{proof}
\begin{remark} In a loose sense we may write \eqref{equA:weakstrong}
as commuting partial dervatives $u_{xt}=u_{tx}$. However, this equality
has to be interpreted with care since time and space derivatives are
taken with respect to different norms.
\end{remark}

We proceed with the proof of Theorem \ref{thm:4.17} by using the stability
statements from \eqref{equ:stabphase},\eqref{equ:stabfunc}.
From \eqref{equ:defvstar} and \eqref{equ:3.1c} we obtain
\begin{align} \label{equA:diffrep}
V_0 - V_{\star}= (u_0-v_{\star},v_0+\mu_{\star}v_{\star,\xi}+N(u_{0,\xi}-v_{\star,\xi}),
v_0+\mu_{\star}v_{\star,\xi}-N(u_{0,\xi}-v_{\star,\xi}) + c(u_0-v_{\star}))^{\top}.
\end{align}
Therefore, we have a constant $C_{\star}=C_{\star}(c,\|N\|)$ with  
\begin{align} \label{equA:estinitdata}
\|V_0 - V_{\star}\|_{H^2} \le C_{\star} (\|u_0-v_{\star}\|_{H^3} +
\|v_0+ \mu_{\star} v_{\star,\xi} \|_{H^2}) \le C_{\star} \rho,
\end{align}
and we take $\rho$ such that $C_{\star} \rho \le \rho_0$.
Let $V\in V_{\star} + \mathcal{CH}^1([0,\infty);\R^{3m})$ be the unique
solution of \eqref{equ:2.7sysA} for $\|V_0 - V_{\star}\|_{H^2}\le \rho_0$.
The first component $V_1$ satisfies
\begin{align} \label{equA:V3eq} 
V_{1,t}= N V_{1,x} - c V_1 + V_3, \quad V_1(\cdot,0)=u_0,
\end{align}
so that $\tilde{V}_1=V_1 -v_{\star}$ solves the Cauchy problem
\begin{align*}
\tilde{V}_{1,t}= N \tilde{V}_{1,x}-c \tilde{V}_1 + V_3-V_{\star,3} - \mu_{\star} v_{\star,x},\quad
\tilde{V}_1(\cdot,0)= u_0 -v_{\star}.
\end{align*}
Then Lemma \ref{LemmaA1} applies with $k=2, A_1=N, B_1=-cI_m$,
$r=V_3-V_{\star,3}-\mu_{\star}v_{\star,x}$ and yields $\tilde{V}_1 = V_1- v_{\star} \in \mathcal{CH}^2([0,\infty);\R^m)$. By Lemma \ref{lemA:commutederiv}
we obtain $\tilde{V}_{1,x} \in C^1([0,\infty);L^2(\R,\R^m))$ as well as
$\tilde{V}_{1,tx}=\tilde{V}_{1,xt} \in C^0([0,\infty);L^2(\R,\R^m))$. Since $v_{\star}$ does
not depend on $t$ we also have $V_{1,tx}=\tilde{V}_{1,tx}=\tilde{V}_{1,xt}=V_{1,xt}$.
For the same reason $\tilde{V}_{1,tt}=V_{1,tt} \in C^0([0,\infty);L^2(\R,\R^m))$, and
$\tilde{V}_{1,xx}=(V_{1}-v_{\star})_{xx} \in C^0([0,\infty);L^2(\R,\R^m))$ implies
$ V_{1,xx} \in C^0([0,\infty);L^2(\R,\R^m))$ since $v_{\star,xx}\in H^2(\R,\R^m)$
by Assumption \ref{ass:4.basic}.
 Thus we can take space and
time derivative of equation \eqref{equA:V3eq} and obtain from the third row
of  \eqref{equ:2.7sysA} 
\begin{equation}\label{equA:prewave}  
\begin{aligned} 
\tilde{f}(V)= & V_{3,t}+NV_{3,x} -c V_2 \\
= & V_{1,tt}-N^2V_{1,xx} - N V_{1,xt} + cV_{1,t} + N V_{1,tx} + cN V_{1,x}-cV_2\\
= & V_{1,tt}-N^2V_{1,xx} -c (V_2 -V_{1,t}-N V_{1,x}).
\end{aligned}
\end{equation}
Next introduce the functions
\begin{equation} \label{equA:Wdef}
W_2=V_2-V_{1,t}-NV_{1,x}, \quad W_3 = V_3 -V_{1,t} +NV_{1,x} -cV_1.
\end{equation}
Using \eqref{equA:prewave}, the last two rows of \eqref{equ:2.7sysA}  and Lemma \ref{lemA:commutederiv} again,
these functions solve the hyperbolic system
\begin{equation} \label{equA:Wsyst}
\begin{aligned} 
W_{2,t} - N W_{2,x} = & V_{2,t}-V_{1,tt} -NV_{1,xt} -NV_{2,x} +NV_{1,tx} + N^2 V_{1,xx}
=-cW_2 \\
W_{3,t}+NW_{3,x} = & V_{3,t}+NV_{3,x} -V_{1,tt}-N V_{1,tx}  +N V_{1,xt}+ N^2 V_{1,xx}
-cV_{1,t}-cNV_{1,x}\\
= & \tilde{f}(V)+cV_2 -(\tilde{f}(V)+c W_2)-c(V_{1,t}+NV_{1,x})=0.
\end{aligned}
\end{equation}
Using from \eqref{equ:2.7sysA} the initial data  and the differential equation
at $t=0$  one finds that 
$W_2(\cdot,0)=0$, $W_3(\cdot,0)=0$. Since \eqref{equA:Wsyst}
with homogeneous initial data has only the trivial solution we conclude
$W_2\equiv0, W_3\equiv 0$. Therefore, by setting $u=V_1$,  equations \eqref{equA:prewave} and \eqref{equA:Wdef} finally lead to
\begin{align*}
u_{tt}-N^2 u_{xx} = \tilde{f}(V)
=  M^{-1}f(V_1,\frac{1}{2}N^{-1}(V_2-V_3+cV_1),\frac{1}{2}(V_2+V_3-cV_1))
=M^{-1}f(u,u_x,u_t).
\end{align*}
Applying \eqref{equ:stabphase} and using \eqref{equA:estinitdata} we obtain that the asymptotic phase $\varphi_{\infty}$  satisfies an estimate
\begin{align*}
|\varphi_{\infty}| \le C \|V_0-V_{\star}\|_{H^2} \le C C_{\star}
(\|u_0-v_{\star}\|_{H^3} +
\|v_0+ \mu_{\star} v_{\star,x} \|_{H^2}).
\end{align*}
 Further, we have for $t\ge 0$ the stability estimate
\begin{align*}
\|V(\cdot,t)-V_{\star}(\cdot-\mu_{\star}t -\varphi_{\infty})\|_{H^1}
\le C C_{\star} (\|u_0-v_{\star}\|_{H^3} +
\|v_0+ \mu_{\star} v_{\star,x} \|_{H^2}) e^{-\eta t},
\end{align*}
where $C$ depends only on $\eta,\rho$.
From this we retrieve the estimate \eqref{equ:3.10a} for the original
variables by taking the $H^1$-norm of the equation
\begin{align} \label{equA:transest}
V(\cdot,t) - V_{\star}(\cdot - \mu_{\star}t - \varphi_{\infty})=
\begin{pmatrix} I_m & 0 & 0 \\ 0 & N & I_m \\ cI_m & -N & I_m \end{pmatrix}
\begin{pmatrix} u(\cdot,t) -v_{\star}(\cdot-\mu_{\star}t -\varphi_{\infty}) \\ 
                u_x(\cdot,t) -v_{\star,x}(\cdot-\mu_{\star}t-\varphi_{\infty})\\ 
u_t(\cdot,t)+ \mu_{\star}v_{\star,x}(\cdot-\mu_{\star}t-\varphi_{\infty})
\end{pmatrix}
\end{align} 
and using that the left factor of the right-hand side is invertible.
%---------------------------------------------------------------------------------------------------------------------------------------------------
\subsection{Lyapunov stability of the freezing method}
\label{subsec:3.4}
Let us first recall from \cite[Thm.2.7]{RottmannMatthes2012a}  the stability
theorem for the freezing method
associated with the first order formulation \eqref{equ:diaghyp}
\begin{subequations} 
  \label{equA:diagfreeze}
  \begin{align}
    & W_t= \Lambda_E W_x + G(W) + \mu W_x,
    \quad\,x\in\R,\,t\ge 0, W(x,t) \in \R^l \label{equA:freezedgl} \\
    & W(\cdot,0) = W_0, \label{equA:freezeini}\\
    & \Psi(W-\hat{W})=0. \label{equA:freezephase}
  \end{align}
\end{subequations}
Here, $\Psi:L^2(\R,\R^l)\rightarrow \R$ is a linear functional
and $\hat{W}:\R \rightarrow \R^l$ is a template function for which we
assume
\begin{enumerate}[label=\textup{(\roman*)},leftmargin=0.65cm]
  \setlength{\itemsep}{0cm}
  \setcounter{enumi}{7}
  \item \label{itm:3.g} $\Psi(W_{\star,\xi}) \neq 0$, $\Psi$ is bounded,
  \item \label{itm:3.h} $\hat{W} \in W_{\star}+ H^1(\R,\R^l)$ and
$\Psi(W_{\star}- \hat{W}) = 0$.
\end{enumerate}
Under the combined assumptions of \ref{itm:3.a}-\ref{itm:3.f} and
\ref{itm:3.g}, \ref{itm:3.h} the result is the following. For
every $0 < \eta < \delta$ there exists $\rho_0 >0$ such that for all
initial data $W_0 \in W_{\star} + H^2(\R,\R^l)$ with $\|W_0 - W_{\star}\|_{H^2}
\le \rho_0$ the system \eqref{equA:diagfreeze} has a unique solution $(W,\mu)$ 
in $\Big(W_{\star}+ \mathcal{CH}^1([0,\infty); \R^{3m}) \Big) \times C([0,\infty),\R)$.
Moreover, there is a constant $C= C(\eta)$ such that the solution
satisfies
\begin{equation} \label{equA:estsolpdae}
\|W(t) - W_{\star}\|_{H^1} + | \mu(t) -\mu_{\star}| \le
C(\eta) \|W_0 - W_{\star} \|_{H^2} e^{- \eta t}, \quad t \ge 0.
\end{equation}
We apply this to the frozen version of  \eqref{equ:2.4N} with the
functional $\Psi$ and the function $\hat{W}$ defined by
\begin{equation} \label{equA:deftemplate}
\begin{aligned} 
\hat{V}= & \begin{pmatrix} \hat{v} & 0 & 0 \end{pmatrix}^{\top},
\hat{W}= T^{-1}\hat{V}, \\
\Psi(W-\hat{W}) = & \langle T(W-\hat{W}), T \hat{W}_{\xi} \rangle_{L^2}.
\end{aligned}
\end{equation}

While conditions \ref{itm:3.a}-\ref{itm:3.f} have already been verified,
the conditions \ref{itm:3.g} and \ref{itm:3.h} easily follow from Assumption
\ref{ass:3.template} and the settings $W_{\star}= T^{-1}V_{\star}$ and
\eqref{equ:defvstar}.  Thus the above result applies. By $(W,\mu)$ we
denote the unique solution of \eqref{equA:diagfreeze} for
$\|W_0-W_{\star}\|_{H^2}\le \rho_0$, and we let $(V=TW,\mu)$ be the unique solution
in $\Big(V_{\star}+ \mathcal{CH}^1([0,\infty); \R^{3m}) \Big) \times C([0,\infty),\R)$ of the transformed equation
\begin{subequations} 
  \label{equA:nondiagfreeze}
  \begin{align}
    & V_t= E V_{\xi} + F(V) + \mu V_{\xi},
    \quad\,\xi \in\R,\,t\ge 0,  \label{equA:freezenondgl} \\
    & V(\cdot,0) = V_0, \label{equA:freezenonini}\\
    & \langle V_1 - \hat{v}, \hat{v}_{\xi} \rangle_{L^2}=0. \label{equA:freezenonphase}
  \end{align}
\end{subequations}
We impose two
conditions on the radius $\rho$ appearing in \eqref{equ:3.17}.
The first one is $C_{\star}\rho \le \rho_0$ as in the argument following
\eqref{equA:estinitdata}. The second one is to ensure for some constant $\underline{C}>0$
\begin{equation} \label{equA:phaseuniform}
|\langle V_{1,\xi}(\cdot,t),\hat{v}_{\xi} \rangle_{L^2}| \ge \underline{C},\quad
\forall t \ge 0
\end{equation}
for all solutions satisfying \eqref{equ:3.17}. 
In fact, from \eqref{equA:estsolpdae}, \eqref{equA:estinitdata} and
Assumption \ref{ass:3.template} we obtain
\begin{align*}
| \langle V_{1,\xi}(\cdot,t),\hat{v}_{\xi}\rangle_{L^2}| \ge & | \langle v_{\star,\xi},\hat{v}_{\xi}\rangle_{L^2}| - \|V_{1}(\cdot,t)-v_{\star}\|_{H^1} \|\hat{v}_{\xi}\|_{L^2}\\
\ge &| \langle v_{\star,\xi},\hat{v}_{\xi}\rangle_{L^2}| - 
 C(\eta) e^{-\eta t} \|T\| \|W_0-W_{\star}\|_{H^2}  \|\hat{v}_{\xi}\|_{L^2}\\
\ge & | \langle v_{\star,\xi},\hat{v}_{\xi}\rangle_{L^2}| - 
\rho  C(\eta) \|T\| \|T^{-1}\| C_{\star} \|\hat{v}_{\xi}\|_{L^2}.
\end{align*}
Our next step is to prove regularity of the solution in the sense that
\begin{align} \label{equA:regularity} 
V_1 \in v_{\star}+ \mathcal{CH}^2([0,\infty); \R^{m}) , \quad
\mu \in  C^1([0,\infty),\R).
\end{align}
For this we define $\gamma \in C^1([0,\infty),\R)$ by $\gamma(t)=\int_0^t \mu(s)ds$ and return to the original variables via 
$U(x,t):=V(x-\gamma(t),t)$ for $x\in \R,t\ge 0$. Then we have that
$U \in V_{\star} + \mathcal{CH}^1([0,\infty);\R^{3m})$ solves the first order
system \eqref{equ:2.7sysA}. Hence the regularity $U_1 \in v_{\star} + \mathcal{CH}^2([0,\infty);\R^{m})$ is obtained via Lemma \ref{LemmaA1}
 by the same arguments as those following \eqref{equA:V3eq}.
In particular, $U_{1,x} \in \mathcal{CH}^1([0,\infty);\R^m)$ and thus
$V_{1,\xi} \in \mathcal{CH}^1([0,\infty);\R^m)$ since $V_{1,\xi}(\cdot,t)=
U_{1,x}(\cdot+\gamma(t),t)$ and $\gamma\in C^1([0,\infty),\R)$. 
For the smoothness of $\mu$ we differentiate the phase
condition \eqref{equA:freezenonphase} with respect to $t$ and use 
\eqref{equA:freezenondgl}
\begin{align*}
0=&\langle V_{1,t},\hat{v}_{\xi} \rangle_{L^2} = \langle
NV_{1,\xi} -cV_1 +V_3,\hat{v}_{\xi} \rangle_{L^2} +
\mu \langle V_{1,\xi},\hat{v}_{\xi} \rangle_{L^2}.
\end{align*}
By \eqref{equA:phaseuniform} this can be solved for $\mu$ and yields
$\mu \in C^1([0,\infty),\R)$ since the other terms are known to be $C^1$-smooth. Thus we have $\gamma\in C^2([0,\infty),\R)$ and then finally
$V_1 \in v_{\star}+ \mathcal{CH}^2([0,\infty); \R^{m})$ from
$V(\xi,t)=U(\xi+\gamma(t),t)$ and 
$U_1 \in v_{\star} + \mathcal{CH}^2([0,\infty);\R^{m})$.

Retrieving the frozen second order equation \eqref{equ:3.2} now uses
 the same arguments as in the nonfrozen case. Therefore we only indicate
the revised equations and leave out computations.
Equation \eqref{equA:prewave} is replaced by
\begin{align} \label{equA:prefrozen}
\tilde{f}(V) = V_{1,tt}-N^2 V_{1,\xi\xi}+\mu^2 V_{1,\xi\xi} - 2 \mu V_{1,t\xi} - \mu_t 
V_{1,\xi} + c(V_{1,t}-V_2+(N-\mu I_m)  V_{1,\xi}).
\end{align}
In view of \eqref{equ:3trans} the analogous functions of \eqref{equA:Wdef}
are defined as follows
\begin{equation} \label{equA:Wnewdef}
W_2=V_2-V_{1,t}-(N-\mu I_m)V_{1,\xi}, \quad W_3 = V_3 -V_{1,t} +(N+ \mu I_m)V_{1,\xi} -cV_1.
\end{equation}
They solve the hyperbolic homogeneous Cauchy problem
\begin{align*}
W_{2,t}-(N+\mu I_m) W_{2,\xi} =& -c W_2, \quad W_2(\cdot,0)=0, \\
W_{3,t} +(N- \mu I_m)W_{3,\xi} = & 0, \quad W_3(\cdot,0)=0,
\end{align*}
hence vanish identically. Inserting this in \eqref{equA:prefrozen}
shows that $v = V_1 \in v_{\star} + \mathcal{CH}^2([0,\infty);\R^m)$
and $\mu$ solve the frozen second order system \eqref{equ:3.2}.

Concerning the estimate \eqref{equ:3.18}, we note the following relation which
replaces \eqref{equA:transest}
\begin{align} \label{equA:transestfreeze}
V(\cdot,t) - V_{\star}=
\begin{pmatrix} I_m & 0 & 0 \\ 0 & N & I_m \\ cI_m & -N & I_m \end{pmatrix}
\begin{pmatrix} v(\cdot,t) -v_{\star} \\ 
                v_{\xi}(\cdot,t) -v_{\star,\xi}\\ 
v_t(\cdot,t)+ \mu_{\star}v_{\star,\xi} - \mu(t) v_{\xi}(\cdot,t)
\end{pmatrix}.
\end{align}
Taking the $H^1$-norm of this equation and using the estimate
\eqref{equA:estsolpdae} with $V,V_{\star},V_0$ instead of $W,W_{\star},W_0$
then gives the exponential estimate in \eqref{equ:3.18}
for $\|v(\cdot,t)-v_{\star}\|_{H^2}$, $|\mu- \mu_{\star}|$ and
  $\|v_t(\cdot,t)+ \mu_{\star}v_{\star,\xi} - \mu(t) v_{\xi}(\cdot,t)\|_{H^1}$.
Using the estimates for the first two terms and the triangle inequality
on the last term then yields an exponential estimate for
$\|v_t(\cdot,t)\|_{H^1}$. This finishes the proof. 

%  ------------------
% |   Bibliography   |
%  ------------------
%\bibliographystyle{abbrv}
%\bibliography{literature}

\begin{thebibliography}{10}

\bibitem{ComsolMultiphysics52}
\textsc{Comsol Multiphysics 5.2}, 2015.
\newblock http://www.comsol.com.

\bibitem{blr14}
W.-J. Beyn, Y.~Latushkin, and J.~Rottmann-Matthes.
\newblock Finding eigenvalues of holomorphic {F}redholm operator pencils using
  boundary value problems and contour integrals.
\newblock {\em Integral Equations Operator Theory}, 78(2):155--211, 2014.

\bibitem{BeynLorenz1999}
W.-J. Beyn and J.~Lorenz.
\newblock Stability of traveling waves: dichotomies and eigenvalue conditions
  on finite intervals.
\newblock {\em Numer. Funct. Anal. Optim.}, 20(3-4):201--244, 1999.


\bibitem{BeynOttenRottmannMatthes2013}
W.-J. Beyn, D.~Otten, and J.~Rottmann-Matthes.
\newblock Stability and computation of dynamic patterns in {PDE}s.
\newblock In {\em Current Challenges in Stability Issues for Numerical
  Differential Equations}, Lecture Notes in Mathematics, pages 89--172.
  Springer International Publishing, 2014.


\bibitem{BeynOttenRottmannMatthes2016b}
W.-J. Beyn, D.~Otten, and J.~Rottmann-Matthes.
\newblock  Computation and stability of traveling waves in second order Evolution Equations
\newblock Preprint 16-022, CRC 701 (Bielefeld University)  \normalfont{\url{http://arXiv:1606.08844v1}}
\newblock  2016.



\bibitem{BeynOttenRottmannMatthes2016}
W.-J. Beyn, D.~Otten, and J.~Rottmann-Matthes.
\newblock Freezing traveling and rotating waves in second order evolution equation.
\newblock Preprint 16-039, CRC 701 (Bielefeld University)  \normalfont{\url{http://arXiv:1611.09402}}
\newblock (submitted), 2016.

\bibitem{BeynThuemmler2004}
W.-J. Beyn and V.~Th{\"u}mmler.
\newblock Freezing solutions of equivariant evolution equations.
\newblock {\em SIAM J. Appl. Dyn. Syst.}, 3(2):85--116 (electronic), 2004.

\bibitem{BeynThuemmler2007}
W.-J. Beyn and V.~Th{\"u}mmler.
\newblock Phase conditions, symmetries and {PDE} continuation.
\newblock In {\em Numerical continuation methods for dynamical systems},
  Underst. Complex Syst., pages 301--330. Springer, Dordrecht, 2007.

\bibitem{BrenanCampbellPetzold1996}
K.~E. Brenan, S.~L. Campbell, and L.~R. Petzold.
\newblock {\em Numerical solution of initial-value problems in
  differential-algebraic equations}, volume~14 of {\em Classics in Applied
  Mathematics}.
\newblock Society for Industrial and Applied Mathematics (SIAM), Philadelphia,
  PA, 1996.
\newblock Revised and corrected reprint of the 1989 original.

\bibitem{FitzHugh1961}
R.~FitzHugh.
\newblock Impulses and physiological states in theoretical models of nerve
  membrane.
\newblock {\em Biophys. J.}, 1:445--466, 1961.

\bibitem{GarrayJoly2009}
T.~Gallay and R.~Joly.
\newblock Global stability of travelling fronts for a damped wave equation with
  bistable nonlinearity.
\newblock {\em Ann. Sci. \'Ec. Norm. Sup\'er. (4)}, 42(1):103--140, 2009.

\bibitem{GallayRaugel1997}
T.~Gallay and G.~Raugel.
\newblock Stability of travelling waves for a damped hyperbolic equation.
\newblock {\em Z. Angew. Math. Phys.}, 48(3):451--479, 1997.


\bibitem{GohbergLancasterRodman2009}
Y.Z.~Gohberg and P. Lancaster and L. Rodman
\newblock{Matrix polynomials}
\newblock{\em Classics in applied mathematics} 58, SIAM, 2009.

\bibitem{GrillakisShatahStrauss1987}
M.~Grillakis, J.~Shatah, and W.~Strauss.
\newblock Stability theory of solitary waves in the presence of symmetry. {I}.
\newblock {\em J. Funct. Anal.}, 74(1):160--197, 1987.

\bibitem{GrillakisShatahStrauss1990}
M.~Grillakis, J.~Shatah, and W.~Strauss.
\newblock Stability theory of solitary waves in the presence of symmetry. {II}.
\newblock {\em J. Funct. Anal.}, 94(2):308--348, 1990.

\bibitem{Hadeler1988}
K.~P. Hadeler.
\newblock Hyperbolic travelling fronts.
\newblock {\em Proc. Edinburgh Math. Soc. (2)}, 31(1):89--97, 1988.

\bibitem{HairerLubichRoche1989}
E.~Hairer, C.~Lubich, and M.~Roch\'{e}.
\newblock {\em The numerical solution of differential algebraic systems by
  Runge-Kutta methods}, volume 1409 of {\em Lecture notes in mathematics ;
  1409}.
\newblock Springer, Berlin [u.a.], 1989.

\bibitem{Henry1981}
D.~Henry.
\newblock {\em Geometric theory of semilinear parabolic equations}, volume 840
  of {\em Lecture Notes in Mathematics}.
\newblock Springer-Verlag, Berlin, 1981.

\bibitem{KapitulaPromislov2013}
T.~Kapitula and K.~Promislow.
\newblock {\em Spectral and Dynamical Stability of Nonlinear Waves}.
\newblock Applied Mathematical Sciences 185. Springer New York, New York, NY,
  2013.

\bibitem{KozlovMazja1999}
V.~Kozlov and V.~G. Mazja.
\newblock {\em Differential equations with operator coefficients}.
\newblock Springer monographs in mathematics. Springer, Berlin [u.a.], 1999.

\bibitem{Markus1988}
A.~S. Markus.
\newblock {\em Introduction to the spectral theory of polynomial operator
  pencils}, volume~71 of {\em Translations of Mathematical Monographs}.
\newblock American Mathematical Society, Providence, RI, 1988.
\newblock Translated from the Russian by H. H. McFaden, Translation edited by
  Ben Silver, With an appendix by M. V. Keldysh.

\bibitem{MennickenMoeller2003}
R.~Mennicken and M.~M\"oller.
\newblock {\em Non-self adjoint boundary eigenvalue problems}, volume 192 of
  {\em North-Holland mathematics studies ; 192}.
\newblock Elsevier, Amsterdam [u.a.], 1. ed. edition, 2003.

\bibitem{Miura1981}
R.~M. Miura.
\newblock Accurate computation of the stable solitary wave for the
  {F}itz{H}ugh-{N}agumo equations.
\newblock {\em J. Math. Biol.}, 13(3):247--269, 1981/82.

\bibitem{Murray1989}
J.~D. Murray.
\newblock {\em Mathematical biology}, volume~19 of {\em Biomathematics}.
\newblock Springer-Verlag, Berlin, 1989.

\bibitem{Palmer:1984}
K.~J. Palmer.
\newblock Exponential dichotomies and transversal homoclinic points.
\newblock {\em J. Differential Equations}, 55(2):225--256, 1984.

\bibitem{Rauch12}
J.~Rauch.
\newblock {\em Hyperbolic partial differential equations and geometric optics},
  volume 133 of {\em Graduate Studies in Mathematics}.
\newblock American Mathematical Society, Providence, RI, 2012.

\bibitem{RottmannMatthes2010}
J.~Rottmann-Matthes.
\newblock {\em Computation and stability of patterns in hyperbolic-parabolic
  Systems}.
\newblock PhD thesis, 2010.

\bibitem{RottmannMatthes2012a}
J.~Rottmann-Matthes.
\newblock Stability and freezing of nonlinear waves in first order hyperbolic
  {PDE}s.
\newblock {\em Journal of Dynamics and Differential Equations}, 24(2):341--367,
  2012.

\bibitem{RottmannMatthes2012c}
J.~Rottmann-Matthes.
\newblock Stability of parabolic-hyperbolic traveling waves.
\newblock {\em Dyn. Partial Differ. Equ.}, 9(1):29--62, 2012.

\bibitem{RowleyKevrekidisMarsden2003}
C.~W. Rowley, I.~G. Kevrekidis, J.~E. Marsden, and K.~Lust.
\newblock Reduction and reconstruction for self-similar dynamical systems.
\newblock {\em Nonlinearity}, 16(4):1257--1275, 2003.

\bibitem{Sandstede2002}
B.~Sandstede.
\newblock Stability of travelling waves.
\newblock In {\em Handbook of dynamical systems, {V}ol. 2}, pages 983--1055.
  North-Holland, Amsterdam, 2002.

\bibitem{Sandstede2007}
B.~Sandstede.
\newblock Evans functions and nonlinear stability of traveling waves in
  neuronal network models.
\newblock {\em Internat. J. Bifur. Chaos Appl. Sci. Engrg.}, 17(8):2693--2704,
  2007.

\end{thebibliography}

\def\cprime{$'$}

\end{document}